\documentclass[10pt]{amsart}
\usepackage{geometry}                
\usepackage{amsmath, amsfonts, amssymb, amsthm, graphicx, color, ulem, enumerate}
\usepackage{multicol}
\usepackage{subfigure}
\usepackage[all]{xy}
\usepackage{cleveref}
\newcommand{\Tspace}{ {\rule{0pt}{2.6ex} } }
\newcommand{\Bspace}{ {\rule[-1.2ex]{0pt}{0pt} } }

\newcommand{\pd}[2]{\frac{\partial#1}{\partial #2}}
\newcommand{\pdtwo}[2]{\frac{\partial^2#1}{\partial #2^2}}

\newcommand{\deriv}[2]{\frac{d#1}{d#2}}

\newcommand{\R}{\mathbb{R}}
\newcommand{\C}{\mathbb{C}}

\newcommand{\Z}{\mathbb{Z}}
\newcommand{\N}{\mathbb{N}}

\newcommand{\U}{\mathcal{U}}
\newcommand{\DD}{\mathcal{D}}

\newcommand{\lp}{\left(}
\newcommand{\rp}{\right)}
\newcommand{\dd}{\, \mathrm{d} }
\newcommand{\eps}{\epsilon}
\newcommand{\bigo}[1]{\mathcal{O} \lp #1 \rp   }

\newcommand{\tbyt}[4]{\begin{pmatrix} #1 & #2 \\ #3 & #4 \end{pmatrix}}
\newcommand{\triu}[1]{\begin{pmatrix} 1 & #1 \\ 0 & 1 \end{pmatrix}}
\newcommand{\tril}[1]{\begin{pmatrix} 1 & 0 \\ #1 & 1 \end{pmatrix}}
\newcommand{\diag}[2]{\begin{pmatrix} #1 & 0 \\ 0 & #2 \end{pmatrix}}
\newcommand{\offdiag}[2]{\begin{pmatrix} 0 & #1 \\ #2 & 0 \end{pmatrix}}

\newcommand{\sig}{{\sigma_3}}

\newcommand{\vect}[2]{ \begin{pmatrix} #1 \\ #2 \end{pmatrix} }

\newcommand{\PCsolnUp}[4]{\begin{pmatrix}
            e^{-\frac{3\pi#1}{4}} D_{\scriptscriptstyle i#1} \lp e^{\frac{-3i\pi}{4}} #2 \rp &
            \frac{\partial_{#2}-\frac{i}{2}#2}{#4} e^{\frac{\pi#1}{4}} D_{\scriptscriptstyle -i#1} \lp
            e^{-\frac{i\pi}{4}}#2 \rp \\
            \frac{\partial_{#2}+\frac{i}{2}#2}{#3} e^{\frac{-3\pi#1}{4}} D_{\scriptscriptstyle i#1} \lp
            e^{\frac{-3i\pi}{4}}#2 \rp &
            e^{\frac{\pi#1}{4}} D_{\scriptscriptstyle -i#1} \lp e^{\frac{-i\pi}{4}}#2 \rp
\end{pmatrix}}

\newcommand{\PCsolnDown}[4]{\begin{pmatrix}
            e^{\frac{\pi#1}{4}} D_{\scriptscriptstyle i#1} \lp e^{\frac{i\pi}{4}} #2 \rp &
            \frac{\partial_{#2}-\frac{i}{2}#2}{#4} e^{\frac{-3\pi#1}{4}} D_{\scriptscriptstyle -i#1} \lp e^{\frac{3i\pi}{4}}#2 \rp \\
            \frac{\partial_{#2}+\frac{i}{2}#2}{#3} e^{\frac{\pi#1}{4}} D_{\scriptscriptstyle i#1} \lp e^{\frac{i\pi}{4}} #2
            \rp &
            e^{\frac{-3\pi#1}{4}} D_{\scriptscriptstyle -i#1} \lp e^{\frac{3i\pi}{4}}#2 \rp
\end{pmatrix}}

\newcommand{\piece}[1]{ \left\{ \begin{array} {l@{\quad:\quad}l} #1 \end{array} \right.}

\DeclareMathOperator*{\res}{Res}
\DeclareMathOperator{\sech}{sech}
\DeclareMathOperator{\sgn}{sgn}
\DeclareMathOperator{\ad}{ad}
\DeclareMathOperator{\imag}{Im}
\DeclareMathOperator{\re}{Re}
\DeclareMathOperator{\Ai}{ {Ai} }

\DeclareMathOperator{\Jac}{Jac}

\DeclareMathOperator{\interior}{Int}

\DeclareMathOperator{\Li}{Li}

\theoremstyle{plain}
\newtheorem{thm}{Theorem}

\newtheorem{lem}[thm]{Lemma}
\newtheorem{prop}[thm]{Proposition}
\theoremstyle{definition}

\theoremstyle{remark}
\newtheorem*{rem}{Remark}

\newtheoremstyle{RHP}
  {3pt}
  {3pt}
  {\itshape}
  {}
  {\bfseries}
  {}
  {.5em}
  {Riemann-Hilbert Problem #2 #3}

\theoremstyle{RHP}
\newtheorem{rhp}{}[section]

\begin{document}
\normalem
\title{The semi-classical limit of focusing NLS for a family of non-analytic initial data}

\author{Robert Jenkins}
\address{Department of Mathematics, University of Arizona, Tucson, AZ, 85721, USA}
\curraddr{Department of Mathematics, University of Michigan, Ann Arbor, MI, 48109, USA} 
\email{ \texttt{rmjenkin@umich.edu}} 

\author{Kenneth D. T-R McLaughlin}
\address{Department of Mathematics, University of Arizona, Tucson, AZ, 85721, USA} 
\email{\texttt{mcl@math.arizona.edu}}  

\date{\today}                                           

\begin{abstract}
The small dispersion limit of the focusing nonlinear Schro\"odinger equation (NLS) exhibits a rich structure of sharply separated regions exhibiting disparate rapid oscillations at microscopic scales. The non self-adjoint scattering problem and ill-posed limiting Whitham equations associated to focusing NLS make rigorous asymptotic results difficult. Previous studies \cite{KMM03, TVZ04, TVZ06} have focused on special classes of analytic initial data for which the limiting elliptic Whitham equations are well-posed. In this paper we consider another exactly solvable family of initial data, the family of square barriers, $\psi_0(x) = q\chi_{[-L,L]}$ for real amplitudes $q$. Using Riemann-Hilbert techniques we obtain rigorous pointwise asymptotics for the semiclassical limit of focusing NLS globally in space and up to an $\mathcal{O}(1)$ maximal time. In particular, we show that the discontinuities in our initial data regularize by the immediate generation of genus one oscillations emitted into the support of the initial data. To the best of our knowledge, this is the first case in which the genus structure of the semiclassical asymptotics for fNLS have been calculated for non-analytic initial data.
\end{abstract}

\maketitle

\section{Introduction and background}
In this paper we study the semi-classical limit of the Cauchy problem for the focusing nonlinear Schr\"{o}dinger equation (NLS)  
\begin{equation}\label{nls}
 	i \eps \psi_t + \frac{\eps^2}{2} \psi_{xx} + \psi |\psi |^2 = 0, \qquad \psi(x, 0) = \psi_0(x),
\end{equation}
for the family of square barrier initial data
\begin{equation}\label{square barrier}
	\psi_0(x) = \begin{cases} q & |x| \leq L \\ 0 & |x| > L \end{cases}	
\end{equation}	
for any choice of constants $q, L >0.$ In this scaling $\eps$ is a dispersion parameter which we take to be small, $0 < \eps \ll 1$. Specifically, our goal in studying the semiclassical limit is to derive a uniform asymptotic description of the limiting behavior of the $\eps$-parameterized family of solutions $\psi(x,t; \eps)$ of the above Cauchy problem as we let $\eps \downarrow 0$ for $(x,t)$ in compact sets.  

The focusing NLS equation is of course very well known, describing phenomena in a diverse variety of fields, from the theory of deep water waves \cite{Benney} to the study of Langmuir turbulence in plasmas \cite{Langmuir}.  In fact, under suitable assumptions the NLS equation arrises generically in he modeling of any weakly nonlinear, nearly monochromatic, dispersive wave propagation. In particular, it has been used extensively in nonlinear optics where it has been used to describe the evolution of the electric field envelope of picosecond pulses in monomodal fibers \cite{hasegawa}. In this last setting, the small dispersion limit is an appropriate scaling to model the expanding use of dispersion-shifted fiber in applications.
   
An important feature of the focusing NLS equation is its modulational instability. The instability can be understood in a number of ways, we present here the nonlinear perspective supported by Whitham \cite{Whitham99}. Without loss of generality we can write the solution of \eqref{nls} in the form $\psi(x,t) = A(x,t)e^{i S(x,t)/ \eps}$ with $A(x,t)$ and $S(x,t)$ real. If we assume that $A(x,t)$ and $S(x,t)$ evolve slowly with $\eps$, then the solution $q(x,t)$ looks locally like a plane wave with amplitude $A(x,t)$, wavenumber $\partial_x S(x,t)$, and frequency $-\partial_t S(x,t)$. Writing $\psi(x,t)$ in this form and setting $\rho(x,t) = | \psi(x,t) |^2 = A(x,t)^2$ and $u(x,t) = \imag \partial_x \log \psi(x,t) = S_x(x,t)$ we get the following coupled system of equations which is completely equivalent to NLS:
\begin{equation}\label{fluid equations}
	\begin{split}
	\pd{ \rho}{t}  + u\pd{\rho}{x} + \rho \pd{u}{x}  &= 0 \\
	\pd{u}{t} - \pd{\rho}{x} + u\pd{u}{x} &= \frac{\eps^2}{2} \lp \frac{1}{2\rho} \pdtwo{\rho}{x} - \lp \frac{1}{2\rho} \pd{\rho}{x} \rp^2 \rp
	\end{split}
\end{equation}
Assuming that the wave parameters evolve slowly with $\eps$ we can regard the right-hand side as a perturbative correction which we may formally drop. Doing so we arrive at the (genus zero) modulation equations associated with \eqref{nls},
\begin{equation}\label{modulation equations}
	\begin{split}
	\begin{bmatrix} \rho \\ u \end{bmatrix}_t + \begin{bmatrix} u & \rho \\ -1 & u \end{bmatrix}
	\begin{bmatrix} \rho \\ u \end{bmatrix}_x = 0  
	\end{split}
 \end{equation}
 This is a quasilinear system of partial differential equations and it is easy to check that its characteristic speeds are given by $u \pm i \sqrt{\rho}$ which are complex at any point at which $\rho \neq 0$. Thus the limiting system of equations is elliptic for which the Cauchy problem is ill-posed. In fact, the Cauchy problem for an elliptic systems is solvable (by the Cauchy-Kovalevskaya method) only if we insist that the initial data $\rho(x,0)$ and $u(x,0)$ are analytic functions of $x$. This is an overly restrictive condition for a physical system, and one which our square barrier initial data \eqref{square barrier}, with jump discontinuities, does not satisfy. The ill-posedness of the limiting equations makes the semiclassical limit a very delicate question in general and particularly delicate in the case of interest where the initial data is discontinuous. In fact, our results imply that even away from the discontinuities the limiting behavior cannot be described by a guess of the form $\psi = Ae^{ i S/\eps}$ without admitting multi-scale expansions of $A$ and $S$ from the onset.         
 
The original observation was that the modulational instability of the NLS equation would lead to ostensibly chaotic wave-forms. However, a number of numerical studies \cite{MK98, BK99, LM07} observed that the wave forms show a large degree of order exhibiting disparate regions of regular quasiperiodic oscillatory $\eps$-scaled microstructures separated be increasingly sharp curves in the $(x,t)$-plane called caustics or breaking curves, whose scale is set by the initial data. Our main goal in studying the semiclassical limit is to analytically find the caustics and to describe the asymptotic behavior of the oscillations in the interlaying regions offset by these caustics.

Analytic results concerning the behavior of NLS are possible because of its integrable structure which allows one, in principle, to solve the Cauchy problem by the forward/inverse scattering transform using the Lax pair representation \cite{ZS72} for NLS. This procedure consists of three steps: 
\begin{enumerate}[Step 1.]
 \item From the eigenvalue problem in the Lax pair \eqref{lax 1} determine the "spectral data" produced by the given initial data $\psi_0$. This is the forward scattering step in the procedure.  
\item Use the time-flow part of the Lax pair \eqref{lax 2} to describe the spectral data's time evolution as the potential $\psi$ evolves under the NLS equation.
\item Given the scattering data at time $t$, invert the forward scattering transform of Step $1$ to find the solution $\psi(x,t)$ of the NLS equation at a later time. This step for NLS, as is often the case, may be characterized as a matrix-valued Riemann-Hilbert problem, reducing the inversion problem to a question of analytic function theory.         
\end{enumerate}
For a non-integrable problem there is no reason to believe such a procedure should exist. The amazing fact is that for integrable problems the evolution of the spectral data in step 2 is explicit and relatively simple, and moreover, the inversion of the spectral map required in step 3 is actually a well-defined map. 

The semiclassical limit greatly complicates the analysis in both Steps 1 and 3 In the scattering step one must determine both the discrete spectrum ($L^2$ eigenvalues) and the so called reflection coefficient defined on the real line (the continuous spectrum) of the Zakharov-Shabat (hereafter ZS) scattering problem \eqref{lax 1}. The focusing ZS problem, unlike the scattering problems for KdV or defocusing NLS, is not self-adjoint and this makes the determination of the scattering data highly nontrivial. Numerical evidence \cite{Bronski96} and asymptotic results \cite{Miller01} suggest that the number of eigenvalues scales as $1/\eps$ and asymptotically accumulate on a ``Y"-shaped spectral curve for a large class of initial data. The analysis of the inverse scattering step is also frustrated by the vast amount of spectral data. Among other issues, the eigenvalues of the scattering step become poles of the Riemann-Hilbert problem and one is faced with seeking a function with an ever growing number of poles. The question of how to best interpolate the poles is one of the first difficulties one has to address in the inverse scattering step. 

The situation is somewhat simplified for real data with a single maximum, in this case it has been shown \cite{KS03} that the eigenvalues are confined to the imaginary axis. However, exact determination of the eigenvalues is possible only in special cases, among them the Satsuma-Yajima data and its generalization \cite{SY74, TV00}, and piecewise constant data. The exact solvability of the scattering problem for square barrier data \eqref{square barrier} is one of the principle motivations for using it as a model of how the semiclassical evolution smooths out discontinuities. The somewhat common practice of replacing the exact scattering data with simpler approximate data is a dubious procedure for studying initial data with discontinuities. The scattering problem has been shown to be very sensitive to perturbations of the initial data \cite{Bronski01}, and its not clear how using "nearby" scattering data would effect the discontinuities of the initial data. For this reason we found it crucial to solve the true Cauchy problem; we make no approximations to either the initial or scattering data. The only approximations we make are in the inverse problem where the error can be controlled. 

Despite the numerous frustrations, in recent years there has been considerable progress on studying the semiclassical focusing NLS equation. In \cite{KMM03} the authors considered the semiclassical limit of certain soliton ensembles, that is, families of real, analytic, reflectionless initial data whose eigenvalues are given by a quantization rule. In a separate study \cite{TVZ04, TVZ06} the initial data $\psi_0(x) = \sech(x)^{1+i\mu/\eps}$ was considered to study the effect of a nontrivial complex phase on the evolution. In both cases, the authors found sharply separated regions of the $(x,t)$-plane inside which the leading order asymptotic behavior of the solution is described in terms of slowly modulated wave trains of a particular genus whose Riemann invariants' evolution is described by the Whitham equations of the corresponding genus. For each fixed $x$ and increasing $t$ one encounters these sharp transitions, or caustics, across which the genus of the asymptotic description changes. In particular, in both cases it was observed that up to some primary caustic $T(x)$, that is for  $t< T(x)$, the evolution was initial described by slowly modulated plane wave with invariants satisfying the genus zero Whitham equations \eqref{modulation equations}. 

In both the above cases the initial data considered is analytic, and this plays an important role in the analysis of the inverse problem. Additionally, we note that due to the analyticity of the initial data the limiting elliptic modulation equations are locally well-posed. For non-analytic initial data it's not clear what should happen for small (but fixed) times in the semiclassical limit; in principle, the evolution could be described by arbitrarily high genus waves from the onset of evolution. The situation for square barrier data is necessarily complicated; in the small time limit it was shown first in \cite{DiFM05} for the defocusing case and in \cite{Adrian06} for the Manakov system (of which focusing NLS is a special case) that the small time limit of the evolution with initial data given by \eqref{square barrier} is equal to the evolution of the linear problem and Gibbs phenomena occur at the discontinuities. These Gibbs phenomena generate oscillations of arbitrarily short wavelength which propagate away from the discontinuities. Of course, in the semi-classical scaling any fixed time independent of $\eps$ is not a small time in the limit, but the existence of rapid oscillations from the onset is a strong indication that something singular must be occurring near the discontinuities. One of the principal goals in this study is to understand how the semiclassical NLS evolution smooths the discontinuities in the square barrier \eqref{square barrier} initial data. 

In this paper as in all of the above mentioned studies the bulk of the work is contained in analyzing the Riemann-Hilbert problem (hereafter RHP) associated with the inverse scattering transform. The results are asymptotic in nature and rely on the nonlinear steepest-descent method developed by Deift and Zhou first in \cite{DZ93} and then generalized in \cite{DZ94, DVZ94} and in the works of many others. The method, in a nutshell, seeks to construct an explicit parametrix $P(z)$ such that the transformation $m(z) = E(z) P(z)$ maps the solution $m(z)$ of the original RHP to a new unknown $E(z)$ which satisfies a small-norm RHP whose solution and asymptotic behavior can be given in terms of certain singular integrals. In each of the semiclassical and small-time studies mentioned above, as in most cases where the steepest-descent method has been successful, the construction of the parametrix revolves around the study of a single complex phase function of the type $e^{i \theta(z)/\eps}$ which must be controlled. 

The fact that the scattering data in the previous studies can be reduced to a single complex phase is a consequence of the choice to study very `nice' initial data. By studying square barrier data we hope to gain an understanding of solutions for more generic types of initial data, i.e. data outside the analytic class for which the limiting problem \eqref{modulation equations} is ill-posed. As soon as we consider initial data outside the analytic class, we immediately encounter scattering data which cannot be reduced to a single dominant complex phase. For the square barrier data considered here the reflection coefficient takes the form of a whole series: $\sum_{k=0}^\infty r_k(z) e^{i \theta_k(z)/\eps}$ where the $r_k(z)$ are $\eps$-independent algebraically decaying weights (c.f. Sec. \ref{sec: harmonic expansion}). As we will show, as $x$ and $t$ evolve, different terms in the expansion give the dominant contribution in different parts of the spectral $z$-plane. In the course of our analysis we show that for times $t$ up to the secondary breaking the analysis is dominated by the two lowest phases $\theta_0$ and $\theta_1$. However, we see no a-priori way to rule out the possibility that at later times the higher order harmonics emerge to contribute at leading order.

In addition to being multi-phased, the reflection coefficient for rough initial data is necessarily slowly decaying; just as the Fourier transform the reflection coefficient, $r(z)$, for discontinuous initial data decays as $1/z$ for $z \in \R$. The slow decay forces us to consider the contribution of the reflection coefficient everywhere on the real line. This makes the analysis more difficult than those considered in \cite{KMM03, TVZ04} where the reflection coefficient, at worst, contributes only on a finite interval. To complete the inverse analysis we had to marry techniques for analyzing multi-soliton RHPs with those for pure reflection problems. This complication is greatest at the real stationary phase points introduced by the steepest-descent method. On the real line no phase is completely dominant and the multiphase nature of the reflection coefficient lead us to construct delicate `annular' local models in neighborhoods of these stationary phase points. This model is discussed in Section~\ref{sec: outside xi_1}; to our knowledge our annular construction is the first of its kind. 
 
\subsection{Statement of the main results}
The culmination of our work in this paper is the following theorem which describes the leading order asymptotic behavior of the solution of the Cauchy problem defined by eqs.~\eqref{nls} and \eqref{square barrier} for arbitrarily large times $t$ outside the initial support and up to a secondary caustic for $x$ inside the support of the initial data.
\begin{figure}[htb]
	\begin{center}
		\includegraphics[width=.6\textwidth]{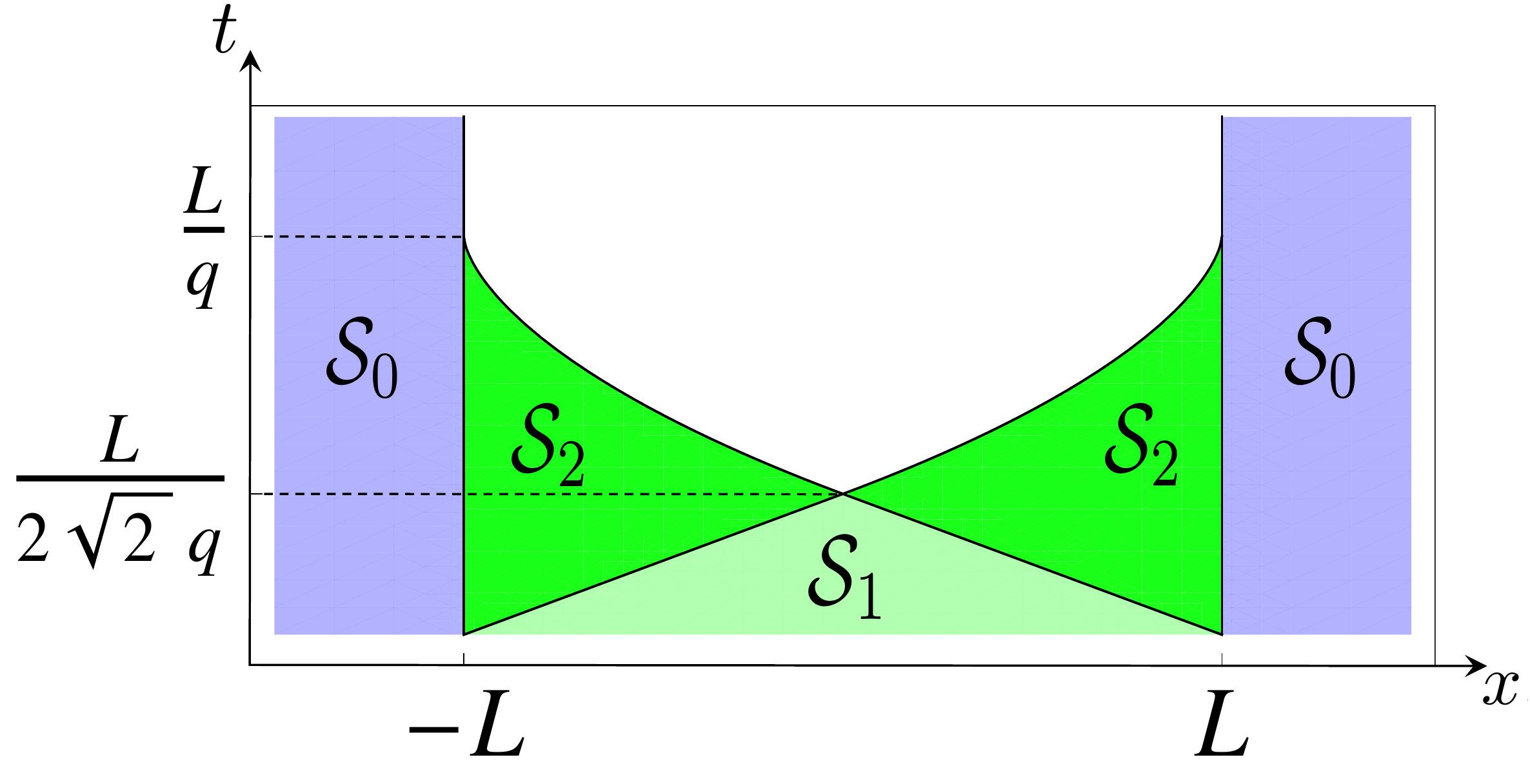}
		\caption{The regions $\mathcal{S}_k,\ k=0,1,2$ in which we have a description of the semi-classical limiting behavior for the initial data \eqref{square barrier} and the caustics $T_1$ and $T_2$ separating the regions.
		\label{fig: spacetime sets}
		}
	\end{center}
\end{figure}

\begin{thm}\label{thm: main}
Consider the curves $t = T_1(|x|)$ and $t=T_2(|x|)$ defined by \eqref{S1} and \eqref{second breaking} respectively. Label by $\mathcal{S}_k,\, k=0,1,2,$ the following subsets of $\R \times [0,\infty)$ (see Figure \ref{fig: spacetime sets}):
\begin{align*}
	\mathcal{S}_0 &= \left\{ (x,t)\, :\, |x| > L, t  \geq 0 \right\} \\
	\mathcal{S}_1 &= \left\{ (x,t)\, :\, |x| < L, 0 \leq t < T_1(x)   \right\} \\   
	\mathcal{S}_2 &= \left\{ (x,t)\, :\, |x| < L,  T_1(x) < T <  T_2(x)  \right\} 
\end{align*}	
Then for all sufficiently small $\eps$, the solution $\psi(x,t)$ of the focusing NLS equation \eqref{nls} with initial date given by \eqref{square barrier} satisfies the the following asymptotic description on any compact $K_j \subset \mathcal{S}_j,\ j=0,1,2$:
\begin{equation*}
	\left| \psi(x,t) - \psi_{asy}(x,t) \right| = \bigo{ \eps^{1/2} \log \eps },
\end{equation*}	 
where,
\begin{enumerate}[i.]
	\item For $(x,t) \in \mathcal{S}_0$,
		\begin{equation*}
			\psi_{asy}(x,t) = 0.
		\end{equation*}
	\item For $(x,t) \in \mathcal{S}_1$, the limiting behavior is described by the nearly plane wave solution
		\begin{equation*}
			\psi_{asy}(x,t) = q \exp \left[ \frac{i}{\eps} \lp q^2 t + \eps \omega(x,t) \rp \right],
		\end{equation*}
		where 		
		\begin{equation*}
			\omega(x,t) = -\frac{1}{ \pi} \lp \int_{-\infty}^{\xi_1} - \int_{\xi_0}^\infty 
			\rp \frac{ \log(1+ |r_0(\lambda)|^2) }{ \sqrt{\lambda^2 + q^2} } \dd \lambda,
		\end{equation*}
		and $r_0,\ \xi_1$, and $\xi_0$ are defined by \eqref{r_0}, \eqref{one band xi_1}, and \eqref{one band xi_0} respectively.  \\
		
	\item For $(x,t) \in \mathcal{S}_2$, the limiting behavior is described by the slowly modulating one-phase wave-train 
	\begin{equation*}
		\psi_{asy}(x,t) =  (q - \imag \alpha)) \frac{\Theta(0)}{\Theta(2A(\infty))}\frac{\Theta(2A(\infty) + i T_0 - i\eps^{-1}\Omega)}{\Theta(iT_0 - i\eps^{-1} \Omega) } e^{-2iY_0 } .
	\end{equation*}
Here, $\Theta(z)$ is the Reimann theta function, see \eqref{Theta}, for the Riemann surface associated with the function
\begin{equation*}
	\mathcal{R}(z) = \sqrt{(z-iq)(z+iq)(z-\alpha(x,t))(z-\alpha(x,t)^*)}
\end{equation*}	  
The points $\alpha(x,t)$ and $\alpha(x,t)^*$ are complex Riemann invariants described by Propostition \ref{prop: two band endpoints} and whose evolution satisfies the genus-one Whitham equation \eqref{one phase Whitham evolution}. The remaining parameters are all slowly varying functions of $(x,t)$ described in terms of certain Abelian integrals:
\begin{equation*}
	A(\infty) = -i\pi \int_{iq}^\infty \frac{\dd z}{\mathcal{R}(z)} \bigg/
	\int_{\alpha^*}^{\alpha} \frac{\dd z}{\mathcal{R}(z)} 
\end{equation*}
and $\Omega,\ T_0,$ and $Y_0$ are given by \eqref{omega def}, \eqref{T}, and \eqref{Y} respectively.
\end{enumerate} 
\end{thm}
We make the following observations concerning our results:
\begin{enumerate}[1.]
\item The barrier does not asymptotically spread out in time, though at any time $t>0$ the solution is no longer compact supported. Our theorem establishes an asymptotically vanishing bound on the amplitude of the solution for $x$ outside the support of the initial data. However, it does not provide any information about the nature of the small amplitude oscillations that occur in the exterior region. These smaller corrections could be found by calculating the higher order terms in the expansion of the small norm RHP \ref{rhp: outside error} but we did not pursue this here. \\

\item In the regions $\mathcal{S}_0$ (trivially) and $\mathcal{S}_1$ the asymptotic limiting solution $\psi_{asy}(x,t)$ is a slowly modulating plane wave solutions of  \eqref{modulation equations}, the genus-zero Whitham equations. The most interesting feature of the solution in this region is the slowly evolving correction to the phase $\omega(x,t)$. This correction has in important consequence: the Whitham modulation theory (discussed around \eqref{fluid equations} and \eqref{modulation equations} above) fails to capture the behavior of the solution. The modulation theory cannot recover the slow phase correction $\omega(x,t)$ which constitutes an $\bigo{1}$ correction to the solution's phase; the Riemann-Hilbert analysis naturally recovers this phase. This is explained in detail in Section~\ref{sec: omega correction} where we show that this correction does not correspond to a regular correction to the geometric optics approximation predicted by WKB. \\

\item The evolution regularizes the initial shocks in the square barrier data at $x=\pm L$ by the instantaneous onset of genus-one oscillations in the growing region $\mathcal{S}_2$. This is consistent with the previous studies \cite{KMM03, TVZ04} where the mechanism for formation of the primary caustic is the formation of an elliptical shock in the solution of the genus zero modulation equations valid for times before the primary caustic.  \\

\item The fact that the reflection coefficient is NOT single phase is not merely a technical difficulty.  The second breaking time $T_2(x)$ which gives an upper bound on the genus-one region $\mathcal{S}_2$ can be characterized in terms of a critical transition related to the {\it second phase}. 
\end{enumerate}

\subsection{Outline of the rest of the paper}
The remainder of the paper is concerned with establishing the proof of Theorem \ref{thm: main}. In Section \ref{sec: scattering} we quickly review the details of the forward scattering theory pertinent to our analysis and show how one explicitly calculates the scattering data for initial data given by \eqref{square barrier}. In section \ref{sec: no band} we consider the inverse problem for $x$ outside the support of the initial data. Section \ref{sec: one band} concerns the inverse problem inside the support up the first caustic while Section \ref{sec: two band} covers the inverse problem between the first and secondary caustics. The analysis in each of the last three sections is largely self-contained though for the sake of brevity we use results and estimates from the previous sections whenever possible. In particular, the properties of the reflection coefficient and resulting jump matrix factorizations discussed in Sections~\ref{sec: harmonic expansion} and \ref{sec: matrix factorizations} are used throughout the paper.
	    
Finally, a quick note on notation. Throughout the paper we make extensive use of the spin matrices:
\begin{equation*}
	\sigma_1 = \offdiag{1}{1}, \quad \sigma_2 = \offdiag{-i}{i}, \quad  \sigma_3 = \diag{1}{-1}.
\end{equation*}
Additionally, given any nonzero scalar $\lambda$ and matrix $A$ let
\begin{equation*}
	\lambda^\sig = \diag{\lambda}{\lambda^{-1}}, \quad \text{and} \quad \lambda^{\ad \sig}A = \lambda^\sig A \lambda^{-\sig}.
\end{equation*}	    
With respect to complex variable notation, we let $z^*$ denote complex conjugate of any complex number or matrix $z$ and let $A^\dagger$ denote the hermitian conjugate of any matrix $A$. For complex valued functions $f(z)$ of a complex variable $z$ we denote by $f^*(z)$ or simply $f^*$ the Schwarz reflection: $f(z^*)^*$. In the rare case when only the function value is to be conjugated we write $f(z)^*$.

\section{Forward scattering of square barrier data \label{sec: scattering}}
We here quickly review the forward scattering theory necessary to define the Riemann-Hilbert problem associated with the inverse scattering problem. We omit proofs, preferring to get as quickly as possible to the analysis of the inverse problem. The interested reader is referred to \cite{BC84, BDT88} for a comprehensive treatment of the subject.

The NLS equation is completely integrable in the sense that it has a Lax pair structure. In \cite{ZS72} Zakharov and Shabat showed that the focusing NLS equation \eqref{nls} can be recognized as the compatibility condition for the following overdetermined pair of equations:
\begin{subequations}\label{lax pair}
\begin{align}
	\label{lax 1} \eps W_x &= \mathcal{L} W ,\qquad\qquad \mathcal{L} := -iz \sigma_3 + \offdiag{\psi}{-\psi^*}\\
	\label{lax 2} i\eps W_t &= \mathcal{B} W, \qquad\qquad \mathcal{B} := iz \mathcal{L} - \frac{1}{2} \tbyt{|\psi|^2}{\eps \psi_x}{\eps \psi^*_x}{-|\psi|^2}  
\end{align}
\end{subequations}
The Lax pair allows one, in principal, to solve the NLS equation using the forward/inverse scattering procedure. In the forward scattering step, one takes given initial data $\psi(x,t=0)= \psi_0(x)$ that decays suitable quickly and seeks a solution $W$ of the first equation in the Lax pair \eqref{lax 1}, known as the Zakharov-Shabat eigenvalue problem, in the form 
\begin{equation}\label{w scat def}
	W = m e^{-\frac{i}{\eps}xz\sig}
\end{equation}
such that,
\begin{equation}\label{w scat cond}
	\begin{split}
		&i.\ m \rightarrow I, \quad \text{as } x\rightarrow \infty \\
		&ii.\ m( \cdot ,z) \text{ is bounded and absolutely continuous}
	\end{split}   
\end{equation}
The construction of $m$ from initial data is a nontrivial step which can be calculated explicitly only in special cases. However, for any potential $\psi_0 \in L^1(\R)$, the solution $m$ must have the following properties \cite{BC84}:
\begin{enumerate}[1.]
\item There exist a discrete set $Z$ in $\C \backslash \R$ such that for each $z \in \C \backslash (\R \cup Z)$ there exist a unique solution of $W$ of $\eps W_x = \mathcal{L}W$ in the form \eqref{w scat def} satisfying \eqref{w scat cond}, 

\item $m(x,z)$ is a meromorphic function of $z$ in $\C \backslash \R$ whose poles are precisely the elements of $Z$ and on $\C \backslash \R,\  \lim_{z \rightarrow \infty} m(x,z) = I$.  

\item The function $m(x,z)$ takes continuous boundary values $m_{\pm}(x,z) := \lim_{\delta \rightarrow 0^+} m(x,z\pm i \delta)$ for each $z \in \R$ which satisfy a ``jump relation", $m_+ =m_- v(x,z)$, where: 
	\begin{equation}\label{generic jump}
		v(x,z) = e^{-\frac{i}{\eps} xz\sigma_3} \tbyt{1+|r(z)|^2}{r^*(z)}{r(z)}{1} e^{\frac{i}{\eps}xz \sigma_3}.
	\end{equation} 

\item The points in $Z$ come in complex conjugate pairs. At each simple pole (the generic case) the residues of $m$ satisfy a relation of the form  
	\begin{equation}\label{generic residues}
		\begin{split}
			&\text{for } z_k \in \C^+,\ \res_{z_k} m = \lim_{z_k} m(x,z) \offdiag{0}{ c_k e^{\frac{2izx}{\eps} }}, \\
			&\text{for } z_k^* \in \C^-,\ \res_{z_k^*} m= \lim_{z_k^*} m(x,z)\offdiag{-c_k^* e^{-\frac{2izx}{\eps} }}{0}.
		\end{split}
	\end{equation}
\end{enumerate}
The form of the jump matrix and the poles coming in conjugate pairs is a consequence of the reflection symmetry 
\begin{equation}\label{m symmetry}
	m(x,z^*) = \sigma_2\, m(x,z)^*\, \sigma_2.
\end{equation}
implied by \eqref{lax 1}.    

The collection of poles $\{ z_k \}$, the connection coefficients $\{c_k\}$, and the function $r(z)$, called the reflection coefficient, appearing in the jump matrix \eqref{generic jump} constitute the totality of the scattering data generated by the initial potential $\psi_0(x)$.

Though one is interested in building $m(x,z)$ for $z$ off the real axis, the construction of the scattering data usually begins by restricting $z$ to the real axis. For real $z$ one seeks solutions $W^{(+)}$ and $W^{(-)}$ of \eqref{lax 1} normalization such that:
\begin{equation}\label{Jost norm}
	\lim_{x \rightarrow \pm \infty} W^{(\pm)} e^{\frac{i}{\eps}xz\sig} = I.
\end{equation}  
Each of these so called Jost solution constitutes a fundamental solution of the differential equation for $z \in \R$ and thus the two solutions satisfy a linear relation of the form 
\begin{equation}\label{scattering relation}
	W^{(-)}(x,z) = W^{(+)}(x,z) S(z), \qquad z \in \R.  
\end{equation} 
The matrix $S$, called the scattering matrix, is independent of $x$ and takes the form
\begin{equation}\label{scattering matrix}
	S(z) = \tbyt{a(z)}{-b^*(z)}{b(z)}{a^*(z)}, \qquad \det S = |a(z)|^2 + |b(z)|^2 \equiv 1.
\end{equation}
The normalized solution $m$ is then constructed by examining the integral equation representations of $W^{(\pm)}$ and observing that different columns of each solution extend analytically to $\C^+$ or $\C^-$. In the process of this construction one finds that the function $a(z)$ extends analytic to $\C^+$, while $b(z)$ is, in general, defined only on the real axis. Moreover, the poles $z_k$ are precisely the zeros of $a(z)$ in $\C^+$ and the reflection coefficient $r(z)$ is defined by the ratio 
\begin{equation}\label{r scattering def}
 	r(z) := b(z)/a(z), \qquad z \in \R.
\end{equation}

The time evolution of the scattering data is remarkably simple. As $\psi$ evolves according to \eqref{nls} the coefficients in the Lax operators $\mathcal{L}$ and $\mathcal{B}$ gain time dependence which in turn is inherited by the spectral data. A basic consequence of the Lax pair structure, is that the time-flow is iso-spectral \cite{Lax68}, that is, the eigenvalues $\{ z_k \}$ and their conjugates are invariant in time. The remaining scattering data evolve as follows:
\begin{equation}\label{spectral evolution}
	r(z,t) = r(z) e^{\frac{2itz^2}{\eps}}, \quad\text{and}\quad c_k(t) = c_k e^{\frac{2itz^2}{\eps}},   
\end{equation}
where $r(z)$ and $c_k$ above are the values of each datum at time $t=0$.  

Given the time evolved spectral data, the problem of reconstructing the potential $\psi$ can be cast as a meromorphic Riemann-Hilbert problem (RHP). For each $x$ and $t$ one seeks a piecewise meromorphic matrix-valued solution $m(z;x,t)$ of the following properties: 
\begin{rhp}[for focusing NLS.] Given the intitial scattering data $r(z)$, $\{z_k\}$, and $\{c_k\}$, find a function $m(z;x,t)$ such that:
\begin{enumerate} [1.]
\item $m$ is meromorphic in $\C \backslash \R$ with poles at the points $z_k$ and their complex conjugates.
\item As $z \rightarrow \infty$ in $\C \backslash \R$, $m \rightarrow I$.
\item As $z$ approaches $\R$ from above and below $m$ assumes continuous boundary values $m_\pm$ satisfying the jump relation $m_+ = m_- v$, where
\begin{equation}\label{rhp jump}
	v = e^{-\frac{i}{\eps}(tz^2+xz)} \tbyt{1+|r(z)|^2}{r^*(z)}{r(z)}{1} e^{\frac{i}{\eps}(tz^2+xz)}. 
\end{equation}
\item The poles of $m$ at each $z_k$ and $z_k^*$ are simple and satisfy the residue conditions
\begin{equation}\label{rhp residues}
	\begin{split}
	\res_{z_k} m &= \lim_{z_k} m \offdiag{0}{c_k e^{\frac{i}{\eps}(2tz^2+2xz)}} \\
	\res_{z_k^*} m &= \lim_{z_k^*} m \offdiag{-c_k^* e^{-\frac{i}{\eps}(2tz^2+2xz)}}{0}
	\end{split}
\end{equation}
\end{enumerate}
If $m(z;x,t)$ is the solution of this problem for the giving scattering data generated by initial data $\psi_0(x)$ then the function defined by the limit
\begin{equation}\label{nls recovery}
	\psi(x,t) = 2i \lim_{z \rightarrow \infty} z\, m_{12}(z;x,t)
\end{equation}
exists and is a solution of the focusing NLS equation \eqref{nls}. 
\end{rhp}

\subsubsection{Scattering for compactly support data, exact solutions for square barrier data}
Here we record some additional properties of the scattering data when the initial potential $\psi_0(x)$ is compactly supported. Additionally, we find explicit formulas for the scattering data when the initial data is piecewise constant and compactly supported which includes the square barrier initial data with which we are interested.

For bounded compact potentials $\psi_0$, the Jost functions $W^{(\pm)}(x,z)$ are entire functions of $z$. The scattering relation \eqref{scattering relation} thus holds not only on the real axis, but in the entire complex plane, and the scattering coefficients $a(z)$ and $b(z)$ are both entire functions for compact initial data. Additionally, the connection coefficients $c_k$ can be expressed explicitly in terms of the scattering coefficients as 
\begin{equation} \label{compact connection}
	c_k = \frac{b(z_k)}{a'(z_k)} = \res_{z_k} r(z).
\end{equation}   
This relation between the connection coefficients and the residues of the reflection coefficient, which is false for generic $L^1$ initial data, will later play a fundamental role in our analysis of the inverse problem.

As discussed previously, calculating explicit formulae for the scattering data generated by generic initial data is intractable, particularly in light of the singular dependence of the scattering data on the dispersion parameter $\eps$. However, for initial data which is compact and piecewise constant the forward scattering problem is simple. We can assume freely that the initial data is supported on an interval $[-L,L]$, and write it as
\begin{equation*}
	\psi_0(x) = \sum_{k=1}^n  q_k 1_{[x_{k-1}, x_k)},
\end{equation*}
where $-L = x_0 < x_1 < \ldots < x_n=L$ is a partition of the interval $[-L,L]$, the $q_k$ are complex constants, and $1_{[a,b)}$ denotes the characteristic function on $[a,b)$. On each interval $[x_{k-1},x_k)$, the ZS scattering problem \eqref{lax 1} reduces to a constant coefficient ODE. The Jost solutions $W^{(\pm)}(x,z)$ are calculated by solving the ZS problem on each interval and patching the solution together by demanding continuity at the partition points. That calculation leads to the following formula for the scattering matrix:
\begin{multline*}
	S(z)= e^{izL\sigma_3/\eps} \, e^{\frac{1}{\eps}(-iz\sigma_3+Q_1)(x_1-x_0)} \times \ldots \\ \ldots\times e^{\frac{1}{\eps}(-iz\sigma_3+Q_n)(x_n-x_{n-1})}  \, e^{izL\sigma_3/\eps}, \qquad \qquad \qquad Q_k = \offdiag{q_k}{-q_k^*}.
\end{multline*}
In the simplest case, the potential consist of a single barrier $\psi_0(x) = q\,  1_{[-L,L]}$. The scattering coefficients $a(z)$ and $b(z)$ are then 
\begin{align}
	a(z) =& \frac{ \nu(z) \cos\lp \frac{2L\nu(z)}{\eps} \rp - i z \sin \lp \frac{2L\nu(z)}{\eps} \rp }{ \nu(z) } 
		e^{2Liz/\eps} ,\label{square a} \\
	b(z) &= \frac{ -q \sin \lp \frac{2L\nu(z)}{\eps} \rp }{\nu(z)} \label{square b},
\end{align}
where
\begin{equation}\label{nu}
	\nu = \sqrt{z^2 + q^2}.
\end{equation}
From the explicit formalae it is clear that both $a(z)$ and $b(z)$ are both entire functions of $z$. For real values of the constant $q$, the initial potential $\psi_0(x)$ falls in to the class of real potentials of a single maximum; as such the poles $z_k$, which are the zeros of $a(z)$ in the upper half plane, must lie on the imaginary axis \cite{KS02}. Using the explicit formula for $a(z)$ we find that the $z_k$ are confined to the imaginary interval $i[0,q)$, and asymptotically there are roughly $1/\eps$ poles in the the interval; as $\eps \downarrow 0$ they collect according to the asymptotic density 
\begin{equation}\label{asymp density}
	\rho_0 (z) = \frac{2Lq}{\pi} \frac{z}{\sqrt{z^2+q^2}},
\end{equation}
which shows that the density has integrable singularities at $z =\pm iq$. For convenience we exclude the countable set of values at which eigenvalues are `born' at the origin
\begin{equation*}
	\eps \neq \eps_n, \qquad \eps_n = \frac{4Lq}{(2n+1)\pi}, \qquad n=0,1,2,\ldots.
\end{equation*}
Making use of these explicit formulae for the scattering data we state the exact Riemann-Hilbert problem for the inverse scattering given square barrier initial data defined by \eqref{square barrier}:
\begin{rhp}[ for focusing NLS with square barrier initial data]\label{rhp: square barrier}
	Find a matrix valued function $m(z;x,t)$ satisfying the following properties:
	\begin{enumerate}[1.]
	\item $m$ is meromorphic for $z \in \C \backslash \R$, with simple poles at each $z_k \in \C^+$ and its complex conjugate. Here, the points $z_k$ enumerate the zeros in $\C^+$ of the function $a(z)$ defined by \eqref{square a}.
	\item As $z \rightarrow \infty$, $m = I + \bigo{1/z}$.
	\item $m$ takes continuous boundary values for $z \in \R$, $m_+$ and $m_-$ satisfying the jump relation, $m_+ = m_- v$ where 
		\begin{equation} \label{square jump rhp}
			v = \tbyt {1+ |r|^2}{r^* e^{-\frac{i}{\eps} \theta}}{ r e^{\frac{i}{\eps}\theta} }{1}, \qquad \theta = 2tz^2+2xz.
		\end{equation}
	\item $m$ satisfies the residue conditions
		\begin{equation} \label{square residues}
			\begin{split}
				\res_{z= z_k} m  &=  \lim_{z \rightarrow z_k} m \offdiag{0}{\res\limits_{z=z_k} r e^{\frac{i}{\eps}\theta} }, \\
				\res_{z= z_k^*} m &=  \lim_{z \rightarrow z_k^*} m \offdiag{-\res\limits_{z=z_k^*} r^* e^{-\frac{i}{\eps}\theta } }{0}.  
			\end{split}
		\end{equation}
	\end{enumerate}
	In \eqref{square jump rhp} and \eqref{square residues}, the reflection coefficient $r(z)$ is given by
	\begin{equation}\label{square r}
		r(z) = \frac{  -q \sin \lp \frac{2L\nu(z)}{\eps} \rp  }{ \nu(z) \cos\lp \frac{2L\nu(z)}{\eps} \rp - i z \sin \lp \frac{2L\nu(z)}{\eps} \rp } e^{-2Liz/	\eps} .
	\end{equation}
	Finally, if $m(z;x,t)$ is the solution of this problem, then \eqref{nls recovery} gives the solution of focusing NLS \eqref{nls} given the square barrier initial data \eqref{square barrier}.
\end{rhp}

The goal now is to use the Deift-Zhou steepest-descent method to construct an approximation $P(z)$ to the solution of this RHP such that the ratio $E(z)$ defined by $m(z) = E(z) P(z)$ is uniformly accurate in the semiclassical limit as $\eps \to  0^+$. The approximation depends parametrically on $(x,t)$ via the function $\theta(x,t)$ appearing in the RHP. In the following sections we will construct such an approximation in each of the space time regions $\mathcal{S}_k,\ k=0,1,2,$ described in Theorem \ref{thm: main}. This procedure consists of four major steps. 1. We introduce a pole removing transformation $m \mapsto M$ such that the new unknown $M$ satisfies a holomorphic RHP. 2. The transformation to $M$ typically introduces contours on which the jump matrices are asymptotically unstable in the semiclassical limit; we therefore introduce a so-called $g$-function transformation  to remove the unstable jump matrices. 3. By explicit factorization of the jump matrices we deform the remaining oscillatory jumps onto contours where they are exponentially near identity. 4. We build a global approximation $A(z)$ to the remaining problem and show that the resulting error $E(z)$ satisfies a small norm RHP. Once the problem is reduced to small-norm form, the theory of small norm RHPs gives a uniform asymptotic expansion of $E(z)$. By unravelling the series of explicit transformations leading from $m$ to $E$ we get an asymptotic expansion for the solution of the original RHP \ref{rhp: square barrier} and through \eqref{nls recovery} the solution $\psi(x,t)$ of NLS.


\section{Inverse problem outside the support of the initial data \label {sec: no band} }

Here we consider the Riemann-Hilbert problem \ref{rhp: square barrier} with scattering data generated by the square barrier initial data \eqref{square barrier} for values of $x$ outside the support of the initial data. Because the NLS equation preserves even initial data we will consider only $x > L$. As will become apparent this is the simplest case in which to carry out the inverse analysis and much of what we do here will be used in Sections \ref{sec: one band} and \ref{sec: two band} where we consider the inverse analysis for $x$ inside the support of the initial data.

\subsection{Reduction to a holomorphic RHP}
To utilize the nonlinear steepest descent method we first need to remove the poles from the Riemann-Hilbert problem. In the reflectionless case this can be accomplished by introducing contours $C$ and $C^*$ enclosing the locus of pole accumulation in $\C^+$ and $\C^-$ respectively. A new unknown $M$ is constructed inside these contours from an interpolate of the connection coefficients and an explicit Blaschke factor term encoding the $z_k$ in such a way that the new unknown has no poles, see \cite{KMM03} for such a construction. For compact initial data there is a more direct construction; the reflection coefficient extends meromorphically off the real line and the relation \eqref{compact  connection} shows that the reflection coefficient naturally interpolates the poles of $m$. 
\begin{figure}[thb]
	\begin{center}
		\includegraphics[width =.4\textwidth]{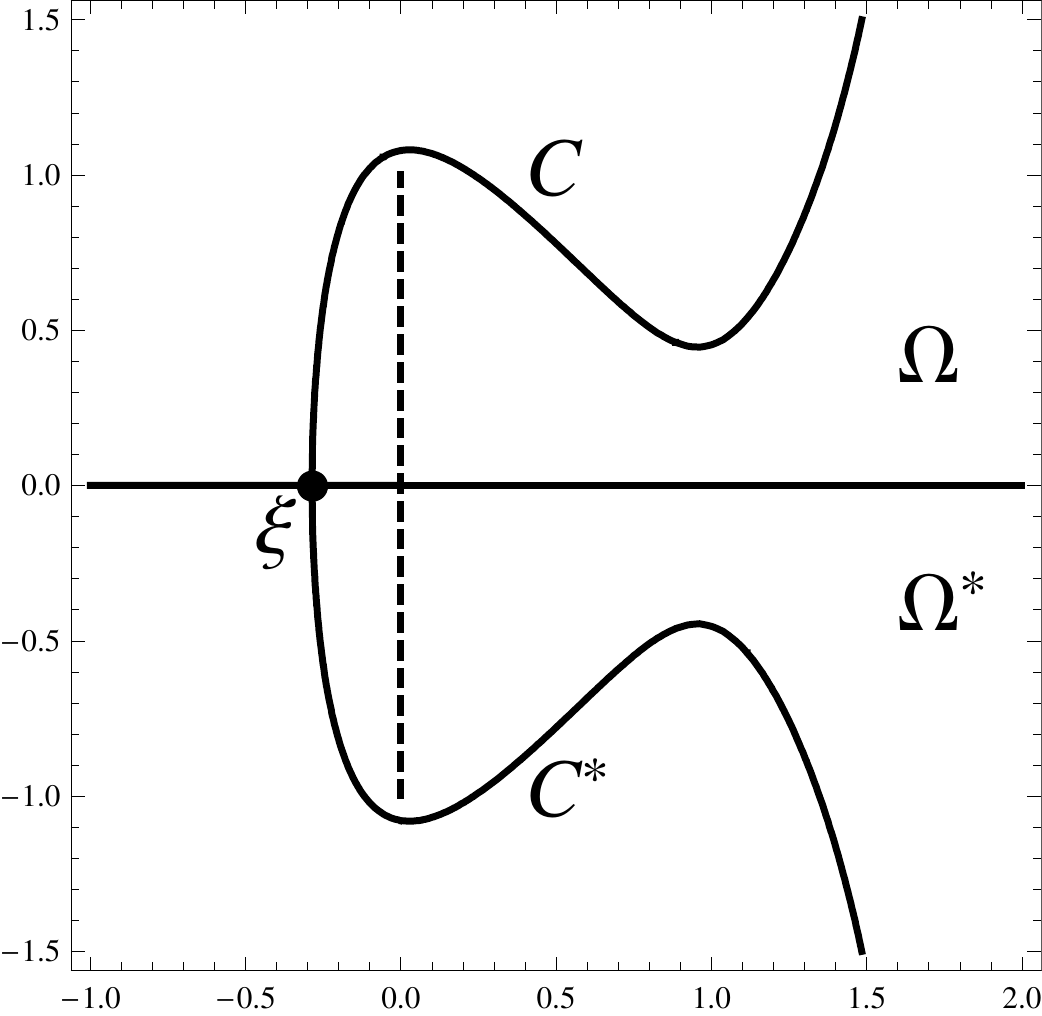}
		\caption{Schematic representation of the contours $C$ and $C^*$ and the regions $\Omega$ and $\Omega^*$ involved in the pole removing factorization. The dashed line represents the loci of pole accumulation. \label{fig: interpolant} }
	\end{center}
\end{figure}
 
Fix a point $\xi < 0$, and take $C$ a semi-infinite contour in $\C^+$ terminating at $\xi$ such that the region $\Omega$ enclosed by $C$ and the real interval $[\xi, \infty)$ contains the locus of pole accumulation $i[0,q]$, see Figure \ref{fig: interpolant}. Let $C^*$ and $\Omega^*$ denote the respective complex conjugate contour and region. Then the new unknown  
\begin{equation}\label{pole removing fact.}
	H= \piece{ 
	m \tril{- r e^{\frac{i}{\eps}\theta} }_{\Bspace} & z \in \Omega \\
	m\triu{ r^* e^{-\frac{i}{\eps}\theta} } & z \in \Omega^* \\
	m & \text{elsewhere}
	}
\end{equation}
has no poles. The price one pays to remove the poles in this way is the appearance of new jumps along the contours $C$ and $C^*$ in the resulting RHP for $H$. In order to preserve the normalization of the RHP and arrive at a problem amenable to steepest descent analysis the transformation from $m$ to $H$ needs to be asymptotically near-identity for large $z$. This condition amounts to understanding how the factor $r e^{i\theta/ \eps}$ appearing in the off-diagonal entries of the transformation behaves in the complex plane. 

\subsubsection{Multi-harmonic expansion of the reflection coefficient}\label{sec: harmonic expansion}
For $z \in \R$ the reflection coefficient \eqref{square r} is rapidly oscillatory, but once $z$ moves off the real axis any finite distance it has a simple, $\eps$ independent,  asymptotic limit. As previously noted, the values of the reflection coefficient are independent of the branch cut given to $\nu = \sqrt{z^2+q^2}$.  For convenience we take $\nu$ to be branched on some simple finite contour $\gamma$ and normalized to $\sim z$ for $z$ large. For $z \in \C^+$ outside the region enclosed by $\gamma$ and the imaginary axis (i.e. the set where $\imag \nu >0$)
\begin{equation}\label{r expansion}
	r(z) e^{i\theta/ \eps}  =\sum_{k=0}^\infty r_k(z) e^{i\theta_k / \eps},
\end{equation}
where
\begin{align}\label{r_0}
	r_0(z) &= \frac{-iq}{\nu(z) + z}  
\end{align}
and
\begin{equation}\label{theta_k def} 
	\begin{split}
	\qquad r_k(z) &= (-1)^{k-1}  r_0(z)^{2k-1} \lp 1+ |r_0(z)|^2 \rp  , \quad k \geq 1;  \\
	\theta_k &= 2tz^2  +2(x-L)z + 4kL\nu, \quad k \geq 0. 
	\end{split}
\end{equation}
In particular, the expansion is uniformly convergent for $z \in \C^+$ with $|z|$ large. The expansion highlights the central technical difficulty of this problem; in previous studies where the semi-classical inverse problem has been successfully solved \cite{KMM03}, \cite{TVZ06}, a single exponential factor emerges in the jump matrices which one must control. Here, we have an infinite sum of different harmonics---one can think of \eqref{r expansion} as a generalized Fourier series for $r$---each contributing to the analysis. A priori, we have no way of knowing which harmonics contribute for each choice of parameters $(x,t)$. However, in the course of our analysis we will see that, at least up to the second caustic, the first two harmonics $\theta_0$ and $\theta_1$ dominate the analysis.

\begin{figure}[thb]
	\centering
	\includegraphics[width = .25\textwidth]{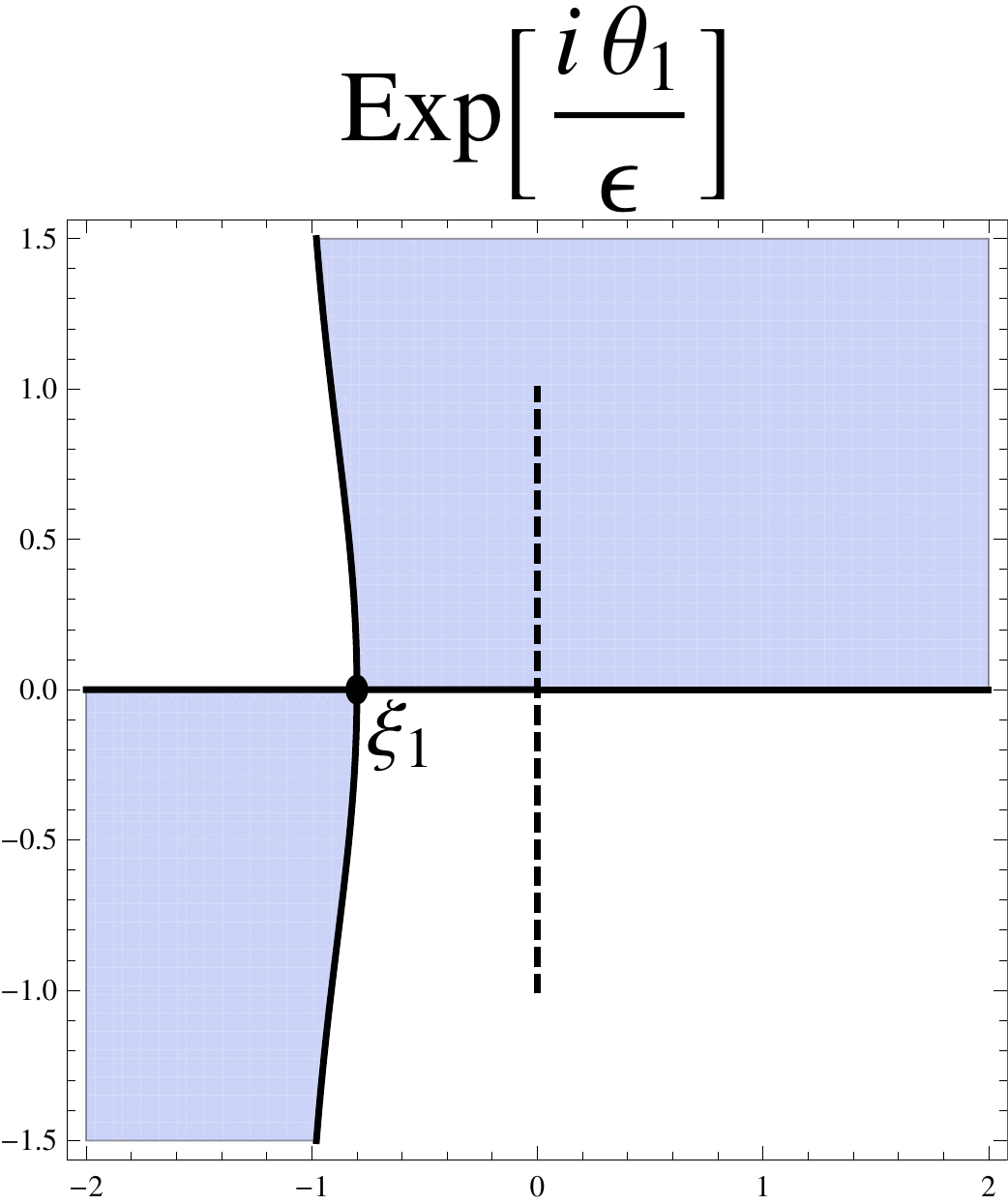}
	\hspace{.1\textwidth}
	\includegraphics[width = .25\textwidth]{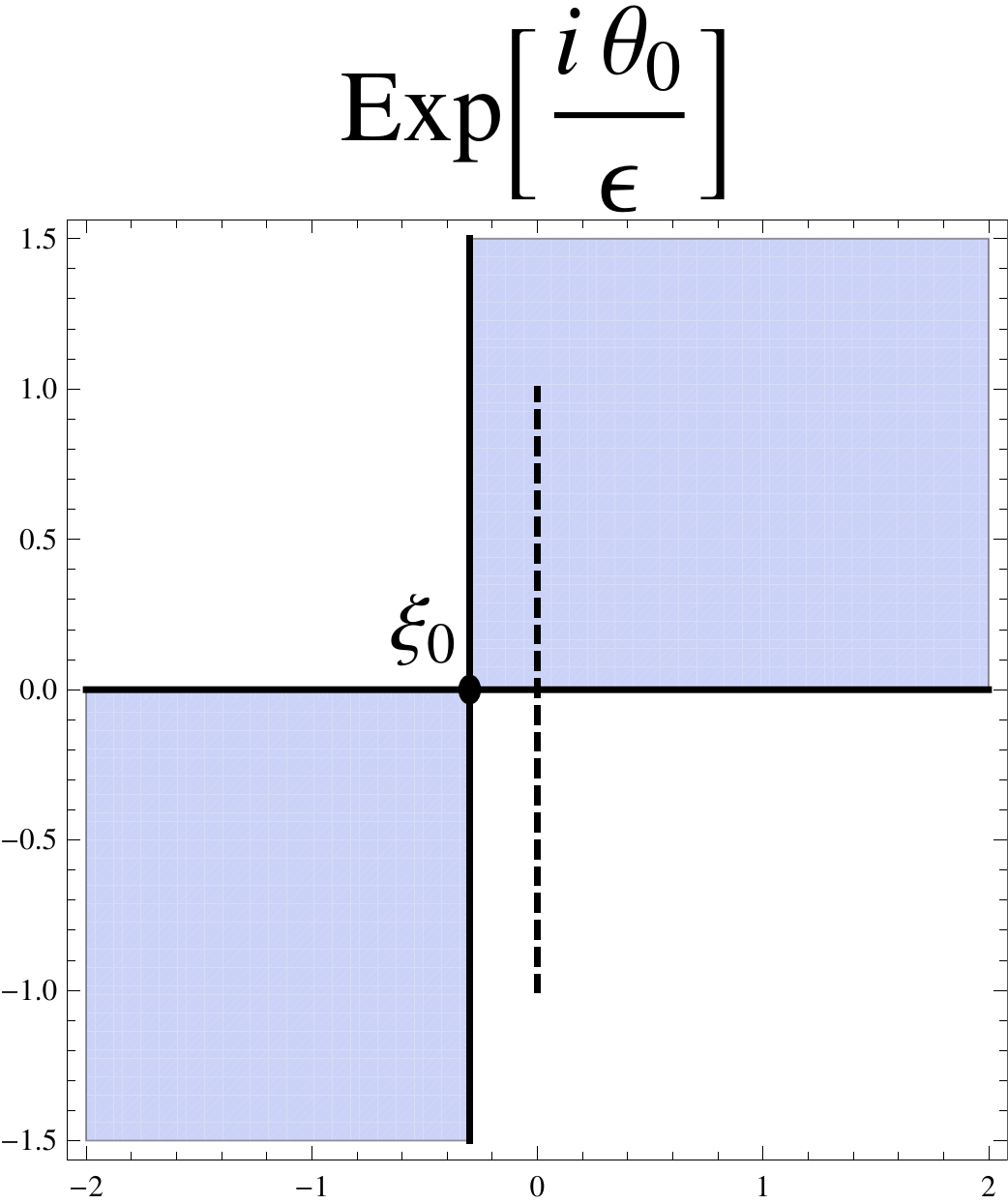}   
	\caption{The regions of growth (white) and decay (grey) of the two principal harmonics for $x>L$ and $t>0$. The dashed line represents the locus of pole accumulation for $m(z)$. \label{fig: phases outside support} }
\end{figure}

The pole removing factorization \eqref{pole removing fact.} works if the region $\Omega$, except for possibly a compact subset, lies in a region in which $r e^{i \theta / \eps}$ is exponentially decaying. This amounts to understanding for each $\theta_k$ where $\imag \theta_k$ is positive (decay) or negative (growth). For $x > L$ we fix $\gamma$, the finite branch cut of $\nu$, to lie on the imaginary axis. This choice of branch makes $\imag \nu(z) > 0$ for all $z$ in $\C^+ \backslash\, i(0, q]$ and thus $\imag \theta_k(z)$ is an increasing function of $k$ for each fixed $z$; if we let $D_k = \{ z \in \C^+\, :\, \imag \theta_k >0 \}$ then $D_{k-1} \subset D_k$ for $k \geq 0$. Furthermore, it is easily shown that each $\theta_k$ has a single, real-valued, stationary phase point $\xi_k$, defined as the unique solution of $\frac{d}{dz}  \theta_k =0$, and as $k$ increases, the $\xi_k$'s  decreases.  

The regions of exponential decay and growth of $e^{i\theta_0 / \eps}$ and $e^{i\theta_1/ \eps}$ for generic values of $x>L$ and $t>0$ are plotted in Figure~\ref{fig: phases outside support}. The $\xi_k$ depend parametrically on $x$ and $t$ and we have
\begin{equation}\label{xi0}
	\deriv{\theta_0}{z}(\xi_0) = 0 \quad \Rightarrow  \quad \xi_0 = -\frac{x-L}{2t}.
\end{equation}
For $k \geq 1$ no such simple formula for the $\xi_k$ can be given, however, for small times we have
\begin{equation}\label{xik small time}
	\xi_k = \frac{ x + (2k-1)L}{2t} + \bigo{t}.
\end{equation}
In particular, for $x>L$, $\xi_0 < 0$ and the locus of pole accumulation in the upper half plane, $i[0,q]$, lies entirely in the region of exponential decay of all the $\theta_k$.  Motivated by these observations, let $\Gamma_0$ be a semi-infinite contour in $\C^+$ emerging from $\xi_0$ such that it encloses the interval $i[0,q]$, stays completely within the region of exponential decay of $e^{i\theta_0/ \eps}$ and is oriented away from $\xi_0$; let $\Omega_0$ be the region between $\Gamma_0$ and the positive real axis; let $\Gamma_0^*$ and  $\Omega_0^*$ be their respective complex conjugates, see Figure~\ref{fig: outside lens openings}. Define 
\begin{equation}\label{M outside}
	M = \piece{
	m \tril{- r e^{\frac{i}{\eps}\theta} }_{\Bspace} & z \in \Omega_0 \\
	m\triu{ r^* e^{-\frac{i}{\eps}\theta} } & z \in \Omega^*_0 \\
	m & \text{elsewhere}
	}.
\end{equation}
As observed previously, it follows from \eqref{square residues} that the new unknown $M$ will have no poles. It follows from its definition that $M$ satisfies the following problem.
\begin{rhp}[for $\mathbf{M}$:] \label{rhp: M outside}
Find a matrix valued function $M$ such that:
	\begin{enumerate}[1.]
	\item $M$ is holomorphic for $z \in \C \backslash \Gamma_M$ where $\Gamma_M = \Gamma_0 \cup \Gamma_0^* \cup (-\infty, \xi_0] $.
	\item As $z \rightarrow \infty$, $M = I + \bigo{1/z}$.
	\item $M$ takes continuous boundary values for $z \in \Gamma_M$, $M_+$ and $M_-$ satisfying the jump relation, $M_+ = M_- V_M$ where 
		\begin{equation} \label{square jump}
			V_M = \piece{ 
			\begin{pmatrix} 1 + |r|^2 & r^* e^{-i\theta/\eps} \\ r e^{i\theta/\eps} & 1 \end{pmatrix} 
			& z \in (-\infty, \xi_0) \\
			\tril{ r e^{i\theta/\eps} }^{\Tspace}_{\Bspace} & z \in \Gamma_0 \\
			\triu{r^* e^{-i\theta /\eps} }  & z \in \Gamma^*_0 
			}
		\end{equation}
	\end{enumerate}	
\end{rhp}
The new jumps follow from calculating $(M_-)^{-1}M_+$ on each contour. Observe that the new jumps on $\Gamma_0$ and $\Gamma_0^*$ are exponentially near identity away from the real axis. Thus for $x>L$, the poles can be removed without introducing any poorly conditioned jumps; as we will see in Sections \ref{sec: one band} and \ref{sec: two band} this is markedly different from the situation for $x \in [0, L)$ and accounts, to large extent, for the difference in the resulting asymptotic description of the solutions inside and outside the initial support.   

\subsection{Reduction to Model Problem, deformation of the jump contours \label{sec: matrix factorizations}} Now that the poles have been removed, we are left with a RHP with rapidly oscillatory jumps on $(-\infty, \xi_0]$. To employ the steepest descent method we seek by explicit transformations to deform these jumps to regions where they are exponentially decaying. Our principal tool for finding such transformations is matrix factorization. Given an oscillatory jump matrix $v$ on a contour $C$, we seek factorizations of the form $v = b_- \, \tilde{v}\, b_+^{-1}$ such that the factors $b_\pm$ are analytically extendable off $C$. If we then introduce contours $C_+$ and $C_-$, to the left and right of $C$ respectively, see Figure~\ref{fig: generic lens opening}, we can define a new unknown
\begin{equation*}
	m_{new} = \piece{ 
	m\,  b_- & z \in \interior \lp  C \cup C_-  \rp \\
	m\,  b_+ & z \in \interior \lp  C \cup C_+  \rp \\
	m & \text{elsewhere}
	}.
\end{equation*}	
\begin{figure}[hbt]
\includegraphics[width = .35\textwidth]{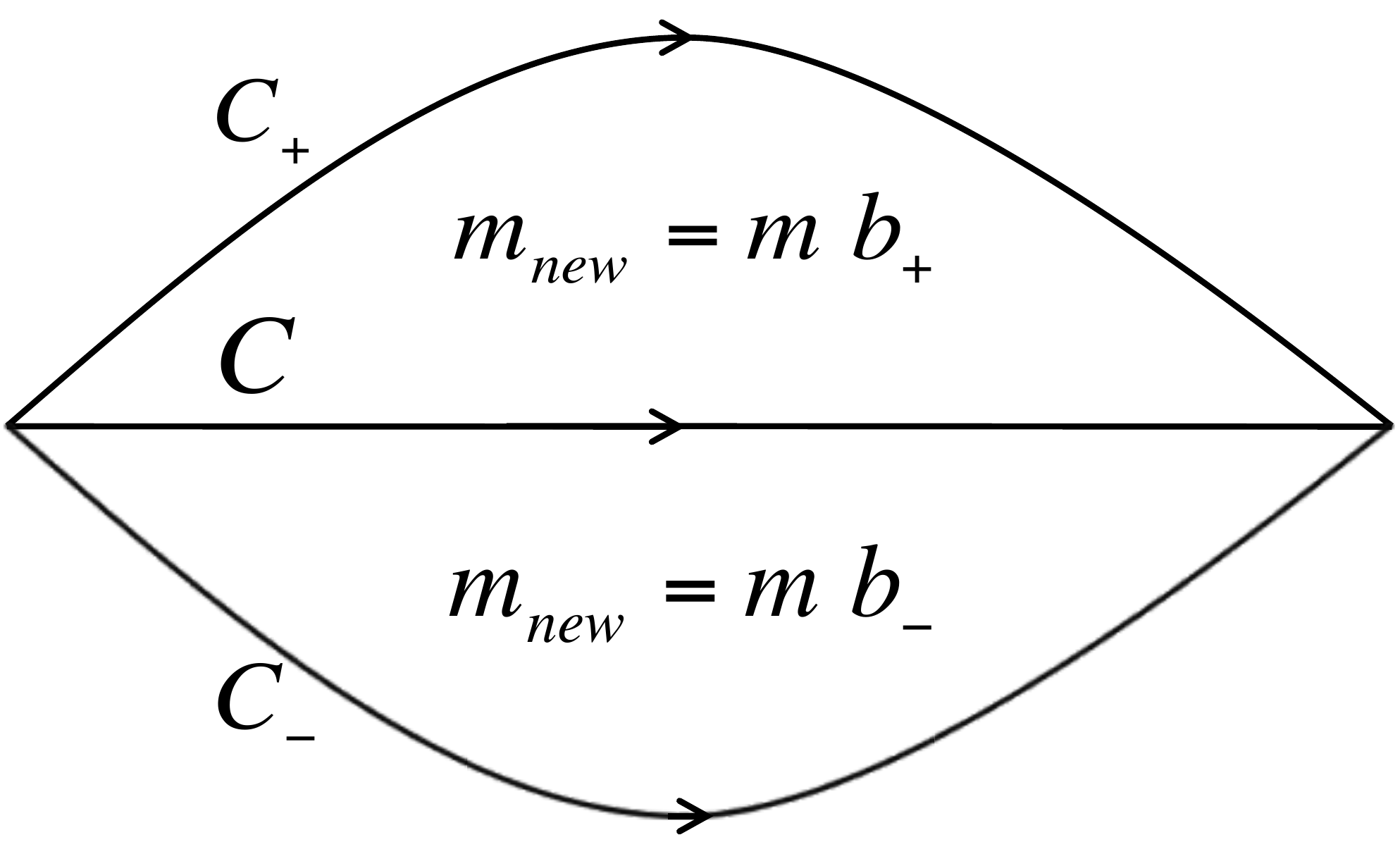}
\hspace{.05\textwidth}
\includegraphics[width = .35\textwidth]{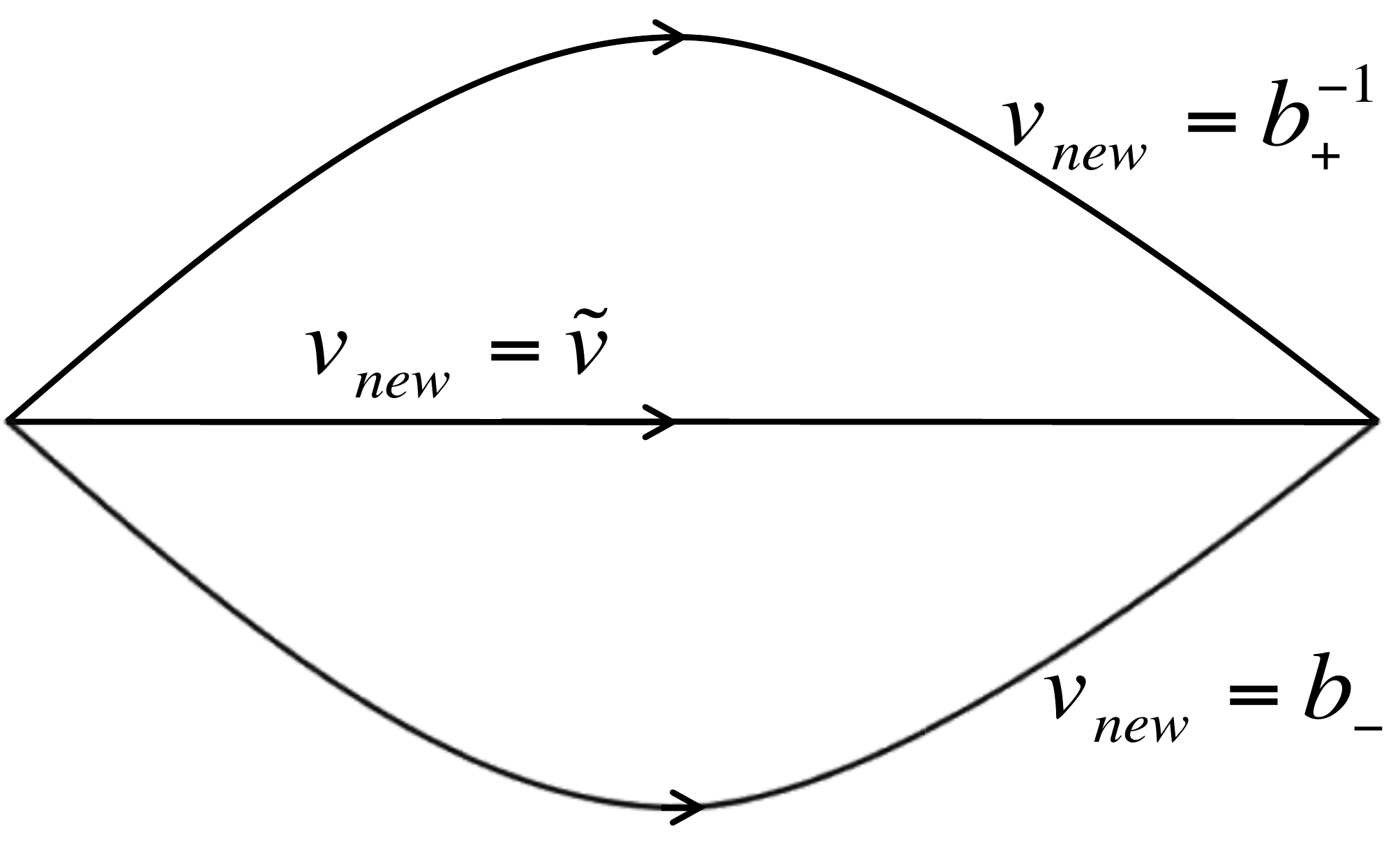}
\caption[Generic lens opening]{Given a factorization $v = b_- \, \tilde{v} \, b_+^{-1}$ of the original jump matrix on $C$, the definition of a new unknown (left) and the resulting new jump matrices (right) in a generic lens opening.
\label{fig: generic lens opening} } 
\end{figure}
The new unknown $m_{new}$ acquires new jumps $v_{new} = (m_{new})_-^{-1}(m_{new})_+$ equal to  
\begin{equation*}
	v_{new} = \piece{
		\tilde{v} & z \in C \\
		b_+^{-1} 	& z \in C_+ \\
		b_-	& z \in C_- 
		}.
\end{equation*}
Such a factorization is useful provided that $b_+$ and $b_-$ are near identity on their respective contours and the jump remaining on $C$, $\tilde{v}$, is no longer rapidly oscillatory. 

\subsubsection{Factorization to the left of $\xi_1$. \label{sec: left factorization}} The original jump matrix $v$ (cf. \eqref{square jump rhp}) has two common factorizations:  
\begin{align}
	&v  = R^\dagger R,  \label {two factor}\\
	 &v = \widehat{R} \, (1+r r^*)^{\sig} \, \widehat{R}^\dagger. \label{three factor}
\end{align}
Where
\begin{equation}\label{R}
	R:= \tril{ r e^{i\theta / \eps}} \quad \text{and} \quad \widehat{R} := \tril{\frac{r}{1+r r^*}  e^{i\theta / \eps} }. 
\end{equation}
The first factorization was used already in the factorization that removed the poles from the original RHP. For $\re z < \xi_1$ we are interested in the second of these factorizations. The rightmost factor $\widehat{R}^\dagger$ is the factor which will factor into $\C^+$, we need to understand the behavior of its off-diagonal entry in the upper half-plane. Recalling the definition \eqref{r scattering def} of the reflection coefficient and the unimodularity \eqref{scattering relation} of the scattering matrix we have
\begin{align*}
	1+r r^* &= 1 + \frac{b b^*}{ a a^* }  = \frac{1}{a a^*},  \\
	\frac{r^*}{1+r r^*} &=  b^* a.
\end{align*}   
Using \eqref{square a} and \eqref{square b} we can express $a$ and $b^*$ in the form 
\begin{align}
	&a = \frac{z+\nu}{2\nu}e^{2Li(z-\nu)/ \eps} \left[ 1 - r_0^2 e^{4Li\nu / \eps} \right], \label{a expansion} \\
	&b^* = -\frac{iq}{2\nu} e^{-2Li\nu / \eps} \left[ 1 -  e^{4Li \nu / \eps} \right], \nonumber
\end{align}
and these expressions give the off-diagonal entry in $\hat{R}^\dagger$ in the form
\begin{equation}\label{r hat behavior}
\frac{r^*}{1+r r^*} e^{-i\theta/ \eps} = \frac{ r_0^*}{ (1 + r_0 r_0^*)^2} e^{-i\theta_1 / \eps} 
 \times \left[ 1 -  e^{4Li \nu / \eps} \right]  \left[ 1 - r_0^2 e^{4Li\nu / \eps} \right].  
\end{equation}
Recalling that $\imag \nu > 0$ for each $z \in \C^+ \backslash \, i(0, q]$ it follows that $\hat R^\dagger$ is exponentially near identity in any region in which the exponential $e^{i\theta_1 / \eps}$ is growing. Thus, \eqref{three factor} is a good candidate factorization for $\re z < \xi_1$.  The only issue is that the central diagonal factor $( 1+ r r^*)^\sig$ to be left on the real axis is still rapidly oscillatory as $\eps \downarrow 0$: this term must also be factored. Using \eqref{a expansion} define
\begin{equation} 
	a_0 := \frac{z+\nu}{2\nu} e^{2Li(z-\nu)/ \eps} = \frac{1}{(1+r_0 r_0^*)}e^{2Li(z-\nu)/ \eps}
\end{equation}
which captures the leading order behavior of the function $a$ for $z \in \C^+$. Note that for $z \in \C^+\backslash (\R \cup(0,iq]) \imag \nu > 0$ so the quantity $a/a_0 = 1-r_0^2e^{4Li\nu/\eps}$ is exponentially near one. Motivated by this observation we introduce the factorization
\begin{equation}\label{left factorization}
	v = \underbrace{ \lp \widehat{R} (a^*/a^*_0)^{-\sig} \rp}_{ \text{factors into } \C^-} \,  .\,
	\underbrace{(1+r_0 r_0^*)^{2\sig}}_{\text{stays on }\R} \,  . \,  
	\underbrace{\lp (a/a_0)^{-\sig}  \widehat{R}^\dagger \rp }_{ \text{factors into } \C^+}
\end{equation}
of the original jump \eqref{square jump rhp} on the interval $(-\infty, \xi_1)$. This factorization greatly simplifies the analysis of the RHP; it reduces the problem to a consideration of only the first two harmonics $\theta_0$ and $\theta_1$ in the expansion \eqref{r expansion} of the reflection coefficient.

\subsubsection{Factorization on the interval $(\xi_1, \xi_0)$ \label{sec: middle factorization}} This interval is the boundary on the real axis of an open region of $\C^+$ in which only the zero$^{th}$ order harmonic $e^{i\theta_0/\eps}$ is large, see Figure \ref{fig: phases outside support}. Motivated by \eqref{r expansion} and the matrix factorization \eqref{two factor} we introduce the factorization
\begin{equation}\label{between factorization}
	v =\lp R^\dagger R_0^{-\dagger} \rp\, .\,  \tbyt{1+r_0 r_0^*}{ r_0^* e^{-i\theta_0/\eps} }{r_0 e^{i\theta_0/\eps}}{1}\, . \, \lp  R_0^{-1} R\rp.
\end{equation}
By $R_0$, and later $\widehat{R}_0$, we denote matrices of the same form as $R$ and $\widehat{R}$ respectively, where in the off diagonal entries $re^{i\theta/ \eps}$ is replaced by its leading order approximation $r_0 e^{i\theta_0/\eps}$:
\begin{equation}\label{R_0}
	R_0 := \tril{r_0 e^{ i \theta_0/ \eps}} \quad \text{and}
	\quad \widehat{R}_0 := \tril{ \frac{ r_0}{1+r_0 r_0^*} e^{i\theta_0 / \eps} }.
\end{equation}
The exterior factors $R R_0^{-1}$ and its hermitian conjugate in \eqref{between factorization} `subtract out' the dependence on $e^{i\theta_0 / \eps}$, in the $(2,1)$-entry we have
\begin{align}\label{r-r_0}
	r e^{i\theta/ \eps} - r_0 e^{i\theta_0/ \eps} = r_0 \left[ \frac{ 1 - e^{4Li\nu / \eps} } { 1 - r_0^2 e^{4Li\nu / \eps}} -1 \right] e^{i\theta_0/ \eps} = \frac{-r_0(1-r_0^2)}{1-r_0^2 e^{4Li\nu / \eps}} e^{i\theta_1 / \eps}
\end{align} 
which at leading order depends on $e^{i\theta_1/\eps}$ which is asymptotically small in $\C^+$ for $\re z > \xi_1$. This allows us to open lenses (see \eqref{Q def outside}) which deform these exterior factors into the regions $\Omega_1$ and $\Omega_1^*$ depicted in Figure \ref{fig: outside lens openings}. 

The remaining middle factor in \eqref{between factorization} takes the same form as the original jump matrix with the reflection coefficient replaced with its leading order behavior. In the region of the upper half-plane  bounded by $\xi_1$ and $\xi_0$ the harmonic $e^{i\theta_0/\eps}$ is exponentially large. This factor is deformed off the axis by an analog of three term factorization \eqref{three factor} (the same terms with $r e^{i \theta/ \eps} $ replaced by $r_0 e^{i\theta_0/\eps}$) yielding the final factorization
\begin{equation}\label{intermediate factorization}
  v = 
  \underbrace{ \lp R^\dagger R_0^{-\dagger} \widehat{R}_0 \rp }_{\text{factors into }\C^-} \, .\, 
  \underbrace{(1+r_0 r_0^*)^\sig}_{\text{stays on }\R} . \, 
  \underbrace{\lp \widehat{R}^\dagger_0  R_0^{-1} R \rp}_{\text{factors into }\C^+}.
\end{equation}

\subsubsection{Reduction from a holomorphic RHP to a stable model problem}

\begin{figure}[htb]
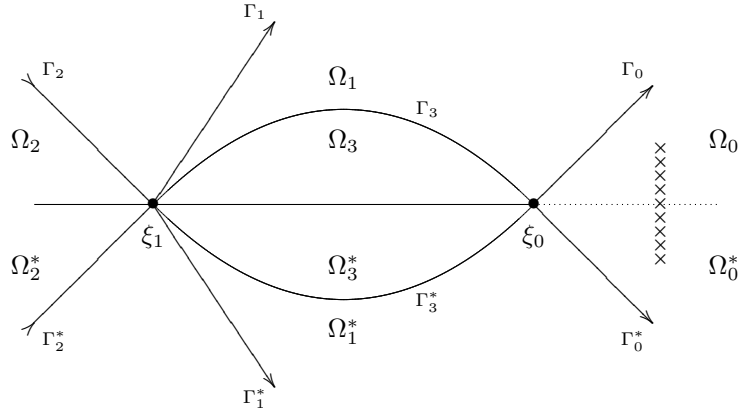

\centerline{
\begin{xymatrix}@!0{
& & & & \ar@{<-}[3,-2]+0_<{\Gamma_1}& & & & & & & \\
\ar@{>-}[2,2]+0^<{\Gamma_2} & &  & &  & \Omega_1& &  && &\ar@{<-}[2,-2]+0_<{\Gamma_0} & \\
\Omega_2 & &  & & & \Omega_3 & & & & & \ar@{{}{x}{}}[2,0]&   \Omega_0 \\
\ar@{-}[0,8]+0 & & \bullet \ar@{}[1,0]|{\displaystyle{\xi_1}}\ar@{-}@/^3pc/[0,6]+0^(.7){\Gamma_3} \ar@{-}@/_3pc/[0,6]+0_(.7){\Gamma_3^*} & & & & & & \bullet \ar@{}[1,0]|{\displaystyle{\xi_0}} \ar@{-}@/^3pc/[0,-6]+0 \ar@{-}@/_3pc/[0,-6]+0 & & & \ar@{.}[0,-3]+0 \\
\Omega_2^* & & & & & \Omega_3^* & & & & && \Omega_0^*\\
\ar@{>-}[-2,2]+0_<{\Gamma_2^*} & & & & & \Omega_1^*& & & & &\ar@{<-}[-2,-2]+0^<{\Gamma_0^*} &\\
& & & &\ar@{<-}[-3,-2]+0^<{\Gamma_1^*}& & & & & & &\\
}\end{xymatrix} 
}
\caption[lens openings for $x>L$]{ The contours $\Gamma_k$ and regions $\Omega_k$ used to define the lens opening transformation $M \mapsto Q$. Their exact shapes are not important, each is chosen to lie in regions where the corresponding factorizations given in \eqref{Q def outside} are asymptotically near identity. Note, that the contour $\Gamma_3$ is oriented from $\xi_1$ to $\xi_0$.  
\label{fig: outside lens openings}}
\end{figure} 

Using the factorizations given above we are ready to define a new unknown $Q(z)$. To define $Q$ we introduce the following set of contours and regions and their complex conjugates. Let $\Gamma_1$ and $\Gamma_2$ be rays in $\C^+$ meeting the real axis at $\xi_1$ such that away from the axis they remain completely in regions where $e^{i\theta_1/\eps}$ is decaying and growing respectively. Additionally we take $\Gamma_1$ to be bounded away from $\Gamma_0$. Let $\Gamma_3$ be a contour in $\C^+$ terminating at $\xi_0$ and $\xi_1$ such that it lies entirely in the region where $e^{i\theta_0/\eps}$ is growing and does not intersect either $\Gamma_0$ or $\Gamma_1$. Finally, let $\Omega_1$ be the region bounded between $\Gamma_0,\ \Gamma_1$, and $\Gamma_3$; similarly let $\Omega_2$ and $\Omega_3$ be the region between the real axis and $\Gamma_2$ or $\Gamma_3$ respectively. The contours $\Gamma_k$, their orientations, and the sets $\Omega_k$ are depicted in Figure \ref{fig: outside lens openings}.  Define
\begin{equation}\label{Q def outside}
Q = \begin{cases} 
M R_0 R^{-1} & z \in \Omega_1 \\
M \widehat{R}^{-\dagger} (a/a_0)^{\sig} & z \in \Omega_2 \\
M R^{-1} R_0 \widehat{R}_0^{-\dagger} & z \in \Omega_3 \\
M R^\dagger R_0^{-\dagger}\widehat{R}_0 & z \in \Omega_3^* \\
M \widehat{R} (a^*/a_0^*)^{-\sig} & z \in \Omega_2^* \\
M R^\dagger R_0^{-\dagger} & z \in \Omega_1^* \\
M & \text{elsewhere}
\end{cases}.
\end{equation} 
Let $\Gamma_Q^{0}$ denote the union of contours used to open lenses, $\Gamma_Q^{0} =  \bigcup_{k=0}^3 ( \Gamma_k \cup \Gamma_k^* )$ and let $\Gamma_Q$ denote the totality of contours on which $Q$ is non-analytic, $\Gamma_Q = \Gamma_Q^0 \cup (-\infty, \xi_0]$ . Using the factorizations \eqref{left factorization}, \eqref{between factorization}, and \eqref{intermediate factorization} it follows that $Q$ satisfies
\begin{rhp}[ for $Q(z)$:]\label{rhp: Q} Find a matrix-valued function $Q$ such that: 
\begin{enumerate}[1.]
	\item $Q$ is holomorphic for $z \in \C \backslash \Gamma_Q$. 
	\item As $z \rightarrow \infty$, $Q(z) = I + \bigo{1/z}$.
	\item $Q$ takes continuous boundary values $Q_+$ and $Q_-$ for $z \in \Gamma_Q$ which satisfy the jump relation $Q_+ = Q_- V_Q$ where
	\begin{equation}\label{Q jumps}
	V_Q = \begin{cases}
		(1+ |r_0 |^2)^{2\sig} & z \in (-\infty, \xi_1) \\
		(1+ | r_0 |^2)^\sig & z \in (\xi_1, \xi_0)  \\
		V_Q^{(0)} & z \in \Gamma_Q^0
		\end{cases}
	\end{equation}
	and
	 \begin{equation}\label{V_Q^0 jumps}
	V_Q^{(0)}=  \begin{cases}
		R_0 & z\in \Gamma_0 \\
		R_0^{-1} R & z\in \Gamma_1 \\
		(a/a_0)^{-\sig}\widehat{R}^\dagger &  z\in \Gamma_2 \\
		\widehat{R}_0^\dagger & z \in \Gamma_3 \\
		\end{cases}
	\qquad
	V_Q^{(0)} = \begin{cases}
		R_0^\dagger & z  \in \Gamma_0^* \\
		R^\dagger R_0^{-\dagger} & z \in \Gamma_1^* \\
		\widehat{R}(a^*/a_0^*)^{-\sig} & z\in \Gamma_2^* \\
		\widehat{R}_0  & z \in \Gamma_3^* 
		\end{cases}
	\end{equation}
\end{enumerate}
\end{rhp}

\subsection{Model RHP outside the support}
The resulting RHP is an asymptotically stable problem in the following sense: as $\eps \to 0$ the jumps remaining on the real axis are non-oscillatory and $\eps$ independent, while along the contours in $\Gamma_Q^0$ the jump matrices $V_Q^{(0)}$ converge to identity both for large $z$ and for fixed, nonreal, $z$ as $\eps \to 0$. The convergence in $\eps$ is uniform on any set bounded away from the two stationary phase points $\xi_0$ and $\xi_1$ where the contours return to the real axis. Motivated by these comments we consider the following global model problem which we will later prove is a uniformly valid approximation to the solution $Q$ of the RHP defined by \eqref{Q jumps}. 
Let $\U_0$ and $\U_1$ be, for now arbitrary, fixed size neighborhoods of $\xi_0$ and $\xi_1$ respectively. We construct a parametrix of the form:
\begin{equation}\label{outside parametrix}
	P(z) = \begin{cases}
		A_0(z) & z \in \U_0 \text{ (cf. \ref{sec: outside xi_0}) } \\
		A_1(z) & z \in \U_1 \text{ (cf. \ref{sec: outside xi_1}) } \\
		O(z) & \text{elsewhere} \text{ (cf. \ref{sec: outside outer}) },
		\end{cases}
\end{equation}
where the outer model $O$ and the local models $A_0$ and $A_1$ will be introduced below in the indicated subsections.

\subsubsection{The outer model problem \label{sec: outside outer} }
Away from the real axis , the jumps on $\Gamma_k \cup \Gamma_k^*$ are exponentially small perturbations of identity. By replacing the jump matrices in \eqref{Q jumps} with their point-wise limits, we are lead to the following outer model RHP.
\begin{rhp}[for the outer model $O(z)$:]\label{rhp: O outside} 
Find a $2 \times 2$ matrix-valued function $O$ such that
\begin{enumerate}[1.]
	\item $O$ is a bounded function, analytic on $\C \backslash (-\infty, \xi_0]$.
	\item $O(z) = I + \bigo{1/z}$ as $z \rightarrow \infty$.
	\item For $z \in (-\infty, \xi_0)$, $O$ takes continuous boundary values satisfying $O_+ = O_- V_O$, where
	\begin{equation}\label{outside outer jumps}
		V_O = \begin{cases}
			(1+ |r_0|^2)^{2\sig} & z \in (-\infty, \xi_1) \\
			(1+|r_0|^2)^{\sig} & z \in (\xi_1, \xi_0).
		\end{cases}
	\end{equation}
\end{enumerate}
\end{rhp}
This RHP appears in the analysis of the long time problem for NLS \cite{DZ94}.  As the jumps are diagonal and non-vanishing, one can easily check using the Plemelj formulae \cite{Musk46} that that solution is given by 
\begin{align} \label{outer model outside} 
	O(z) =  \delta(z)^{\sig} 
\end{align}
where,
\begin{align}\label{delta}
\delta(z) = \exp \left[ \frac{1}{2\pi i} \int_{-\infty}^{\xi_1}  \frac{ \log ( 1+|r_0(s)|^2)}{s-z} \, ds
		+ \frac{1}{2\pi i}  \int_{-\infty}^{\xi_0} \frac{ \log ( 1+|r_0(s)|^2)}{s-z} \, ds \right]. 
\end{align}
The jump matrices for $O$ are discontinuous at $\xi_0$ and $\xi_1$ which causes $\delta$ to behave singularly at these points. The following lemma describes the nature of the singularities.
\begin{prop}\label{prop: delta expansions}
Define 
\begin{align}\label{kappa}
	\chi(z,a) := i \int_{-\infty}^a \frac{\kappa(s)}{s-z} \, ds 
	\qquad \text{where}\qquad 
	\kappa(z):=-\frac{1}{2\pi}\log(1+|r_0(z)|^2),
\end{align}
and suppose $a$ and $b$ are fixed real numbers with $b<a$ such that $\kappa$ is analytic in a neighborhood of each point. Then near these points:
\begin{enumerate}[i.]
	\item As $z \rightarrow a$, $\chi$ has the uniform expansion,
	\begin{equation*}\label{chi endpoint}
		\chi(z,a) = i\kappa(z)\log(z-a) + \hat \chi(z,a) ,
	\end{equation*}
	where $\hat \chi$ is a bounded holomorphic function for $z$ in a neighborhood of $a$.
	\item In any suitably small, $\eps$ independent, neighborhood of $b$, $\mathcal{N}_b$, the boundary values $\chi_+$ and $\chi_-$ naturally extend as analytic functions to all of $\mathcal{N}_b$ and $\chi_+(z,a) - \chi_-(z,a) = \log(1+|r_0(z)|^2)$ at each $z \in \mathcal{N}_b$. 	
\end{enumerate}
\end{prop}

\begin{proof} Both results follow from the fact that the weight $\kappa$ in the Cauchy integral defining $\chi$ is analytic at the points $a$ and $b$. To prove the first part it is sufficient to observe that the difference $\chi - i\kappa(z)\log(z-a)$ has a vanishing jump for all  $z$ near $a$. To construct the analytic extensions of $\chi_\pm$ needed near $z=b$ the analyticity of $\kappa$ is used to deform the contour of integration so that individually the upper and lower boundary values extend analytically to the opposing half-plane. That the extensions satisfy the jump relation follows from the Plemelj formulae. \end{proof}

Expressed in terms of $\chi$, $\delta(z) = \exp \left[ \chi(z,\xi_0) + \chi(z, \xi_1) \right]$; it follows immediately form Prop. \ref{prop: delta expansions} that $\delta$ has the local expansions 
 \begin{equation}\label{delta expansions}
 	\begin{split}
	\delta(z) &= (z-\xi_0)^{i\kappa(z)} \delta_0^{hol}(z) \qquad \text{near } z=\xi_0, \\ 
	\delta_\pm(z) &= (z-\xi_1)^{i\kappa(z)}\delta_{1\pm}^{hol}(z) \qquad \text{near } z=\xi_1 ,
	\end{split}
\end{equation}
where $\delta_0^{hol}$ is holomorphic in a neighborhood of $\xi_0$ and $\delta_{1\pm}^{hol} $ are each holomorphic and satisfy $\delta_{1+}^{hol}/\delta_{1-}^{hol} = 1+r_0 r_0^*$ in any sufficiently small neighborhood of $\xi_1$. In particular, observe that \eqref{delta expansions} shows that $\delta$ is bounded at each singularity. The singularity manifests as rapid oscillations as $z$ approaches $\xi_0$ or $\xi_1$.

\subsubsection{Local model near $z = \xi_0$.}\label{sec: outside xi_0}
\begin{figure}
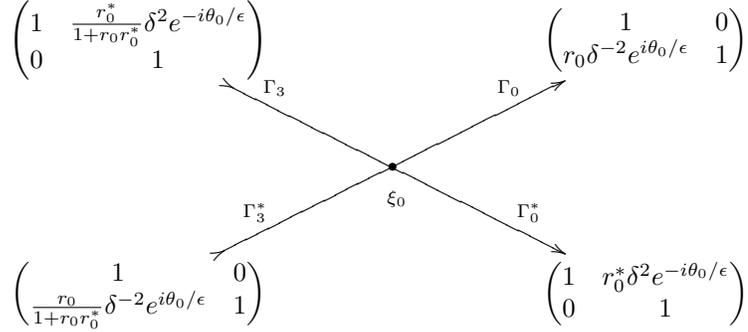

\centerline{
\begin{xymatrix}@!0{
    & & & & & & & & \\
    {\triu{\frac{r^*_0}{1+r_0r_0^*} \delta^{2}e^{-i \theta_0/ \eps} } } \ar@{>->}'[2,4]+0^{\Gamma_3}[0,8]^{\Gamma_0} & & & & & & & & {\tril{r_0\delta^{-2} e^{i \theta_0 / \eps } } } \\
    & & & & & & & &\\
    {\ar@{{}{}{*}}[0,4]+0} & & & & & & & & \\
    & & & & \ar@{}@<-2ex>[-1,0]^{\xi_0} & & & & \\
    {\tril{\frac{r_0}{1+r_0r_0^*} \delta^{-2}e^{i\theta_0/ \eps } } }  \ar@{>->}'[-2,4]+0^{\Gamma^*_3}[0,8]^{\Gamma_0^*} & & & & & & & & {\triu{r^*_0\delta^{2}e^{-i\theta_0 / \eps}}} \\
    & & & & & & & & \\
}\end{xymatrix} } \caption[Local jumps near $z_0$ outside
support]{Local (exact) jump matrices near
the stationary point $\xi_0$ after introducing $\delta$ . \label{fig: local jumps xi_0}}
\end{figure}
Near $z = \xi_0$ we seek our local model $A_0(z)$ in the form:
\begin{equation}\label{A_0 hat outside}
	A_0(z) = \hat{A}_0(z) \delta(z)^{\sig}.
\end{equation}
This has two advantages. First, it simplifies the matching condition, to match the outer model $O(z)$ the new unknown $\hat{A}_0$ needs to be asymptotically near identity on the boundary where $A_0(z)$ and $O(z)$ meet. Second, as was shown in the construction of the outer model, this factorization removes the jump along the real axis. Of course, this comes at the cost of modifying the jumps along the other contours. The exact jumps after left multiplication of $Q$ by $\delta^\sig$ are given in Figure \ref{fig: local jumps xi_0}. At any fixed distance from $\xi_0$ the jump matrices are near identity due to the decay of the exponential factors $e^{\pm i\theta_0/\eps}$ along each ray. The point $\xi_0$ is a stationary phase point of $\theta_0$, (cf. \eqref{xi0}), and thus $\theta_0$ is locally quadratic. This allows us to introduce the following locally analytic and invertible change of variables
\begin{align}\label{zeta_0 outside}
	\frac{1}{2}\zeta_0^2 := \frac{1}{\eps} \lp \theta_0(z) - \theta_0(\xi_0) \rp,  \quad \text{or} \quad \zeta_0 = \sqrt{\frac{\theta''(\xi_0)}{\eps}} (z-\xi_0). 
\end{align}
The domain $\U_0$ on which the local model $A_0$ is defined is selected as follows. We take $\U_0$ to be any suitably small fixed size neighborhood of $\xi_0$ such that $\U_0$ is bounded away from the contour $\Gamma_1$ and chosen so that under the map $\zeta = \zeta_0(z)$ its image $\zeta_0(\U_0)$ is a disk centered at the origin in the $\zeta$-plane (which necessarily has a radius $\sim \eps^{-1/2}$). Additionally, we use the freedom to deform the contours $\Gamma_0$ and $\Gamma_3$ such that the images $\zeta_0(\Gamma_k \cap \U_0),\ k=1,3$ are straight lines leaving the origin at angles $\pi/4$ and $3\pi/4$ respectively.  

The model inside $\U_0$ is constructed by approximating the function $r_0(z)$ by its value at $\xi_0$, and using \eqref{delta expansions} and \eqref{zeta_0 outside} to approximate $\delta$ by 
\begin{align}\label{delta approx}
	\delta \mapsto \zeta_0^{i\kappa(\xi_0)} \lp \frac{\eps}{\theta_0''(\xi_0) } \rp^{i\kappa(\xi_0)/2} \delta_0^{hol}(\xi_0).
\end{align}
Finally, we make the following change of variables which maps to the $\zeta$-plane and removes many of the constant factors from the jumps:
\begin{equation}\label{A hat to PC}
	\hat{A}_0 =  \left[ \lp \frac{\eps}{\theta_0''(\xi_0)} \rp^{i\kappa(\xi_0)/2} \delta_0^{hol}(z)\, e^{-i\theta_0(\xi_0)/2\eps} \right]^{\ad \sig} \Psi_{PC}(\zeta_0(z); r_0(\xi_0) ).
\end{equation}
The resulting RHP for the new unknown $\Psi_{PC}(\zeta, a)$ is one of the canonical model Riemann-Hilbert problems.  
\begin{rhp}[for $\Psi_{PC}$ (The Parabolic Cylinder RHP):]\label{rhp: PC} Given a fixed constant $a$, find a matrix $\Psi_{PC}(\zeta; a)$ such that 
\begin{enumerate}[1.]
	\item $\Psi_{PC}$ is analytic for $\zeta \in \C \backslash \left\{ \zeta\, :\, \arg(\zeta) = \pm \pi/4, \pm3\pi/4 \right\}$.
	\item As $\zeta \rightarrow \infty,\ \Psi_{PC} =  I + \Psi^{(1)}/ \zeta + \bigo{1/\zeta^2} $.
	\item On the jump contours, the boundary values satisfy $\Psi_{PC+} = \Psi_{PC-} V_{\Psi_{PC}}$, where
	\begin{equation}\label{Psi_0 jumps}
	V_{\Psi_{PC}} = \begin{cases}
	\tril{ a \zeta^{-2i\kappa} e^{i\zeta^2/2} }_{\Bspace} & \arg(\zeta) = \pi/4 \\
	\triu{a^* \zeta^{2i\kappa}e^{-i\zeta^2 / 2} }_{\Bspace} & \arg(\zeta) = -\pi/4 \\   
	\triu{\frac{ a^*}{ 1+ a a^*}  \zeta^{2i\kappa}e^{-i\zeta^2 / 2} }_{\Bspace} & \arg(\zeta) = 3\pi/4 \\   
	\tril{\frac{a }{1+ a a^* }  \zeta^{-2i\kappa} e^{i\zeta^2 / 2} } & \arg(\zeta) = -3\pi/4,
	\end{cases}
	\end{equation}
	and $\kappa =  -\frac{1}{2\pi}\log(1+a a^*)$.
\end{enumerate}
\end{rhp}

\begin{prop}\label{prop: PCRHP solution}
The RHP for $\Psi_{PC}(\zeta; a)$ above has a unique, uniformly bounded, solution given by  
\begin{equation}\label{PCsoln}
    \Psi_{PC}(\zeta; a) = \left\{ \begin{array}{l@{\quad : \quad}l}
        P(\zeta; a) \tril{-a}_{\Bspace} e^{\frac{i}{4}\zeta^2\sig} \zeta^{-i\kappa \sig} & \arg(\zeta) \in (0, \pi/4) \\
        P(\zeta; a) \triu{\frac{-\bar{a}}{1+|a|^2}}_{\Bspace} e^{\frac{i}{4}\zeta^2\sig} \zeta^{-i\kappa \sig} & \arg(\zeta) \in(3\pi/4, \pi) \\
        P(\zeta; a) e^{\frac{i}{4}\zeta^{2}\sig}_{\Bspace} \zeta^{-i\kappa \sig} & |\arg(\zeta)| \in( \pi/4, 3\pi/4) \\
        P(\zeta; a) \tril{\frac{a}{1+|a|^2}}_{\Bspace} e^{\frac{i}{4}\zeta^2\sig} \zeta^{-i\kappa \sig} & \arg(\zeta) \in (-\pi, -3\pi/4) \\
        P(\zeta; a) \triu{\bar{a}} e^{\frac{i}{4}\zeta^2\sig} \zeta^{-i\kappa \sig} & \arg(\zeta) \in (-\pi/4, 0)
        \end{array} \right.
\end{equation}
where the function $P(\zeta,a)$, built out of the parabolic cylinder
functions $D_{\pm i\kappa}(\cdot)$, is
\begin{equation}\label{Parabolic}
    P(\zeta; a)  = \left\{ \begin{array}{l@{\ :\ }l}
                \PCsolnUp{\kappa}{\zeta}{\beta_{12}}{\beta_{21}}_{\Bspace}
                & \zeta \in \C^+ \\
                \PCsolnDown{\kappa}{\zeta}{\beta_{12}}{\beta_{21}}^{\Tspace}
                & \zeta \in \C^-
        \end{array}\right.
\end{equation}
and the constants $\kappa$, $\beta_{12}$, and $\beta_{21}$ are given
by the formulae:
\begin{align*}
    \kappa = -\frac{1}{2\pi} \log(1+a a^*)_{\Bspace}, \qquad \beta_{12} = \frac{\sqrt{2\pi} e^{i\pi/4}e^{-\pi \kappa/2}}{a \Gamma(-i\kappa)_{\Bspace}} , \quad \text{and}\quad \beta_{21} = \kappa / \beta_{12}. 
\end{align*}
\end{prop}
To verify that \eqref{PCsoln} gives the solution of the RHP for $\Psi_0$ one uses the properties of the parabolic cylinder functions, $D_{i\kappa}(\cdot)$, (cf \cite{AS64}) to show explicitly that the above formulae satisfies the jump and normalization conditions. As this is a somewhat standard model problem we refrain from doing so here. The details of the solutions derivation can be found in \cite{DZ94} and \cite{DM08}.

Later we will need estimates of the error introduced by the local model. On the boundary $\partial \U_0$ 
\begin{equation}\label{xi_0 boundary error}
	O^{-1}(z) A_0(z) = I + \bigo{\eps^{-1/2}}.
\end{equation}
This follows from \eqref{zeta_0 outside}, the large $\zeta$ expansion of $\Psi_{PC}$ (cf. RHP \ref{rhp: PC}), and the boundedness of the outer solution $O(z)$ on $\partial \U_0$. The remaining error introduced by the model lies in the approximation of the jumps along $\Gamma_0$, $\Gamma_3$ and their conjugates. By direct calculation, one shows that the largest contribution to this error is introduced by approximation \eqref{delta approx} and consequently along each of these rays:  
\begin{equation}\label{xi_0 ray error}
	V_Q V_{A_0}^{-1} = I + \bigo{\eps^{1/2} \log \eps}. 
\end{equation}

\subsubsection{Local model near $z = \xi_1$\label{sec: outside xi_1} }   	
At $\xi_1$ the first harmonic $\theta_1$ is stationary and this forces the lens opening contours $\Gamma_1, \Gamma_2, \Gamma_3 ,$ and their conjugates to return to the real axis at this point. Just as in the local model at $\xi_0$, the locally quadratic structure of the stationary harmonic motivates the definition of a locally invertible analytic change of variables; define $\zeta = \zeta_1(z)$, through the relation
\begin{equation}\label{zeta_1 def}
	\frac{1}{2} \zeta_1^2 := \frac{1}{\eps}\lp \theta_1(z) - \theta_1(\xi_1) \rp = \frac{\theta''(\xi_1)}{2\eps}(z-\xi_1)^2 + \bigo{(z-\xi_1)^3}. 
\end{equation}
We choose the set $\U_1$ to be a fixed but sufficiently small neighborhood of $\xi_1$ bounded away from $\U_0$, such that \eqref{zeta_1 def} is analytic and invertible inside, and shaped such that the image $\zeta_1(\U_1)$ is a disk in the $\zeta$-plane. Additionally, we use our freedom to choose the rays $\Gamma_{1},\ \Gamma_2$,  and their conjugates such that the images $\zeta_1(\Gamma_k \cap \U_1),\ k=1,2$ lie on the rays $\arg \zeta = \pi/4$ and $3\pi/4$ respectively.  Inside $\U_1$ we seek a local model of the form
\begin{equation}\label{A1 model}
	A_1(z) = A_1^{(1)}(z) \delta(z)^\sig
\end{equation}
so as to remove the jumps along the real axis exactly from the local problem. The effect of \eqref{A1 model} on the exact local jumps is to replace $V_Q$ by $\delta_-^\sig V_Q \delta_+^{-\sig}$; the local contours and exact jumps after introducing this $\delta$ factorization are given in Figure \ref{fig: xi_1 local jumps}. 

\begin{figure}[ht]
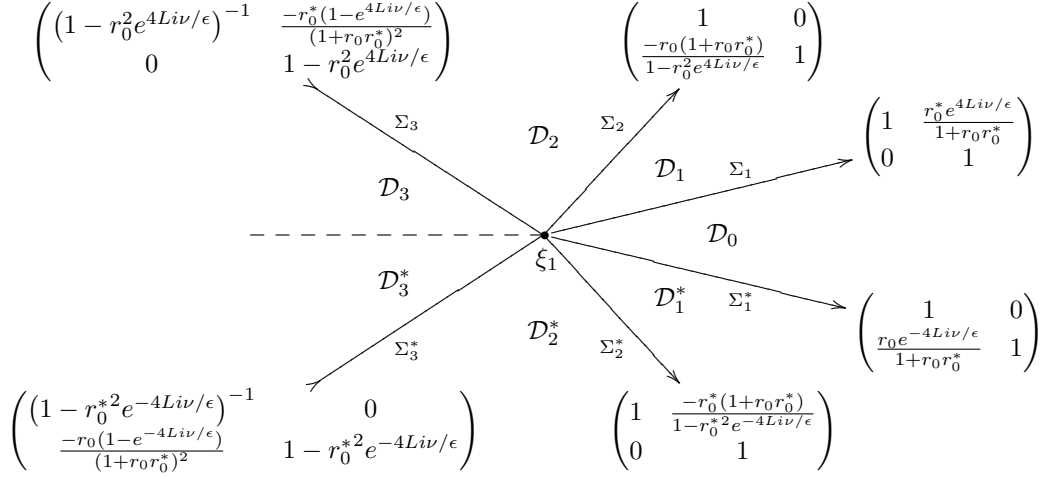

\centerline{
\begin{xymatrix}@-2.0pc{
    &
        {\begin{pmatrix}
        \lp 1 - r_0^2 e^{4Li\nu/\eps}\rp^{-1}
        & \frac{-r^*_0(1-e^{4Li\nu/\eps}) }{(1+r_0 r_0^*)^2} \\
        0 &  1 - r_0^2 e^{4Li\nu/\eps} 
        \end{pmatrix}}
    \ar@{>->}'[4,3]+0^{\Sigma_3}[0,6]^{\Sigma_2} & & & & & & {\tril{\frac{-r_0(1+r_0r_0^*)}{1-r_0^2e^{4Li\nu/\eps}}}} &  \\
    &  & & & \DD_2 & & & &  \ar@{<->}'[3,-4]_{\Sigma_1}[6,0]_{\Sigma_1^*}{\triu{ \frac{ r_0^* e^{4Li\nu / \eps}}{1+r_0 r_0^*} }} \\
    & & & & & & & \ar@{}[-1,-1]|(.4){\displaystyle \DD_1} &    \\
    & \ar@{}[0,1]^(.6){\displaystyle \DD_3} &  & & & &  & &  \\
    &\ar@{--}[0,3]+0 \ar@{{}{}{*}}[0,3]+0  & & & & & &  \DD_0 &  \\
    & \ar@{}[0,1]_(.6){\displaystyle \DD_3^*}& & & \ar@{}@<-2ex>[-1,0]^(.1){\displaystyle \xi_1} & & & & \\
    & & & & & & & \ar@{}[1,-1]|(.4){\displaystyle \DD_1^*} & \\
    & & & & \DD_2^* & & & &  {\tril{\frac{r_0 e^{-4Li \nu/\eps} }{1+r_0r_0^*} }} \\
    &
        {\begin{pmatrix} \lp 1 - {r_0^*}^2 e^{-4Li\nu/\eps}\rp^{-1} & 0 \\
        \frac{-r_0(1-e^{-4Li\nu/\eps}) }{(1+r_0 r_0^*)^2} &  1 - {r_0^*}^2 e^{-4Li\nu/\eps} 
        \end{pmatrix}}
    \ar@{>->}'[-4,3]+0_{\Sigma^*_3}[0,6]_{\Sigma_2^*} & & & & & & {\triu{\frac{-r_0^*(1+r_0 r_0^*)}{1- {r_0^*}^2 e^{-4Li\nu/\eps}} }} &
}\end{xymatrix} }
\caption{ The exact local jumps after conjugating $V_Q$ by $\delta(z)^\sig$ are  $\left[ \delta(z) e^{-i \theta_1(z)/\eps} \right]^{\ad \sig} \widetilde{V}_Q$ where $\widetilde{V}_Q$ are the jumps given above. Inside the shrinking disk $\DD$ (see \eqref{xi_1 disk}) we relabel the contours $\Sigma_k$ and define the subsets $\DD_k, \ k=1,2,3$ in $\C^+$ and their conjugates as shown.
\label{fig: xi_1 local jumps} 
}
\end{figure}

Comparing Figures \ref{fig: local jumps xi_0} and \ref{fig: xi_1 local jumps}, the jumps near $\xi_1$ are more complicated than those near $\xi_0$. The jump matrices on $\Gamma_1$, $\Gamma_2$ and their conjugates involve the full reflection coefficient. As such, at $\xi_1$ the full multi-harmonic expansion returns to the real axis---manifesting as the factor $1-r_0^2 e^{4Li\nu/ \eps}$ appearing in the jumps in $\C^+$ and the conjugate factor appearing in $\C^-$---and we have to deal with every harmonic simultaneously. This is markedly different than the local structure near $\xi_0$ where the jump matrices contained only the single harmonic $\theta_0$. However, there are also important similarities, the $\delta$ factors in the off-diagonal entries have a power law singularity of the form $(z-\xi_k)^{i\kappa}$ and along the rays $\Gamma_1,\ \Gamma_2$, and their conjugates the quadratic decay of a stationary harmonic makes the jumps asymptotically near identity at any fixed distance from $\xi_1$. These similarities lead one to believe that the parabolic cylinder model, RHP \ref{rhp: PC}, should be involved in the construction of the local approximation. The key to the construction of our model at $\xi_1$ is the separation of length scales between the onset of asymptotic growth/decay of the locally linear harmonics, $\theta_k,\ k \neq 1$, and the locally quadratic harmonic, $\theta_1$, which we record in the following elementary proposition.  

\begin{prop}\label{prop:  xi_1 harmonic est} 
For $z \in \U_1$, the following estimates hold:
\begin{enumerate}[ i.]
	\item For $|z - \xi_1| = \bigo{\eps^{1/2} }$, the functions $e^{\pm i \theta_1/ \eps}$ are bounded independent of $\eps$.
	\item For $|z - \xi_1| = \bigo{\eps}$, the functions $e^{\pm 4Li \nu/  \eps}$ are bounded independent of $\eps$.
	\item For $\imag z \gg \eps$  $\lp -\imag z \gg \eps \rp$ the function $e^{4Li\nu/\eps}$ $\lp e^{-4Li \nu / \eps} \rp$  is small beyond all orders.   
\end{enumerate}
\end{prop}

Proposition \ref{prop: xi_1 harmonic est} motivates the introduction of a shrinking disk inside the fixed sized disk $\U_1$. Define
\begin{equation}\label{xi_1 disk}
	\DD = \{ z\, : \, |z-\xi_1| \leq \sqrt{\eps} \}.
\end{equation} 
Inside $\DD$ relabel the jump contours by $\Sigma_k\ (\Sigma_k^*),\ k=1,2,3,$ ordered from the real axis to the right of $\xi_1$ around $\xi_1$ with positive (negative) orientation. Additionally, define the real contour $\Sigma_4 := \{ z \in \DD\, :\, z <\xi_1 \}$, oriented left to right.  Label by $\DD_0$ the subset of $\DD$ bounded by $\Sigma_1$ and $\Sigma_1^*$, and label by $\DD_k,\ 1 \leq k \leq 3,$ the other subsets of $\DD$ in $\C^+$ ordered counterclockwise; let $\DD_k^*$ denote the complex conjugate sectors. Finally, denote by $V_k\ (V_k^\dagger)$ the exact values of the jump  $\delta^\sig V_Q \delta^{-\sig}$ along $\Sigma_k\ (\Sigma_k^*)$, see Figure \ref{fig: xi_1 local jumps}. Our procedure will be to build a model from the parabolic cylinder functions as was introduced in the local model at $\xi_0$. However, the presence of the linearly harmonics requires that inside $\DD$ the jump matrices be preconditioned for approximation by the parabolic cylinder RHP. The price we will pay for this construction are near-identity jumps on $\partial \DD$ and $\Sigma_4$. The preconditioning procedure, which we describe below, is summarized in Figure \ref{fig: xi_1 model steps}. 

For $z \in \U_1 \backslash \DD$, the locally linear harmonics, those terms proportional to $e^{i\theta_0/ \eps}$ or $e^{\pm 4 i\nu/\eps}$ in the local jumps as expressed in Figure \ref{fig: xi_1 local jumps}, are small beyond all orders (Prop. \ref{prop: xi_1 harmonic est}), and they may be set to zero without introducing any appreciable error. Note that this approximation replaces the jumps on $\Gamma_0$ and its conjugate by identity. Eliminating the linear harmonics and employing \eqref{delta expansions} and \eqref{zeta_1 def} the remaining jumps are approximated as follows.  

We define the locally analytic and nonzero scalar function
\begin{equation}\label{xi_1 scalar conjugate}
h_1 = \left[  \lp \frac{\eps}{\theta''_1(\xi_1)} \rp^{i\kappa(\xi_1)} \delta_{1-}^{hol}(z) \delta_{1+}^{hol}(z) e^{-i\theta_1(\xi_1)/ \eps} \right]^{1/2}
\end{equation}
then for $z \in \U_1 \backslash \DD$ we replace the exact jump $\delta^\sig V_Q \delta^{-\sig}$ with $h_1^\sig V_{\xi_1} h_1^{-\sig}$, where,   
\begin{equation}\label{xi_1 annular approx}
V_{\xi_1} (z) =  \begin{cases}
		\tril{-r_0(\xi_1) \zeta_1(z)^{-2i\kappa(\xi_1)}  e^{i\zeta_1(z)^2/2 }}_{\Bspace} & z \in \Gamma_1  \\
		\triu{-r_0^*(\xi_1) \zeta_1(z)^{2i\kappa(\xi_1)} e^{-i\zeta_1(z)^2/ 2} }_{\Bspace} & z \in \Gamma_1^* \\
		\triu{\frac{-r_0^*(\xi_1)}{1+|r_0(\xi_1)|^2 } \zeta_1(z)^{-2i\kappa(\xi_1)} e^{-i\zeta_1(z)^2/ 2} }_{\Bspace} & z \in \Gamma_2  \\
		\tril{ \frac{-r_0(\xi_1)}{1+|r_0(\xi_1)|^2} \zeta_1(z)^{-2i\kappa(\xi_1)} e^{i\zeta_1(z)^2/ 2} } & z \in \Gamma_2^* .
\end{cases}
\end{equation}

\begin{figure}[htb]
\begin{center}
\includegraphics[width =.3\textwidth]{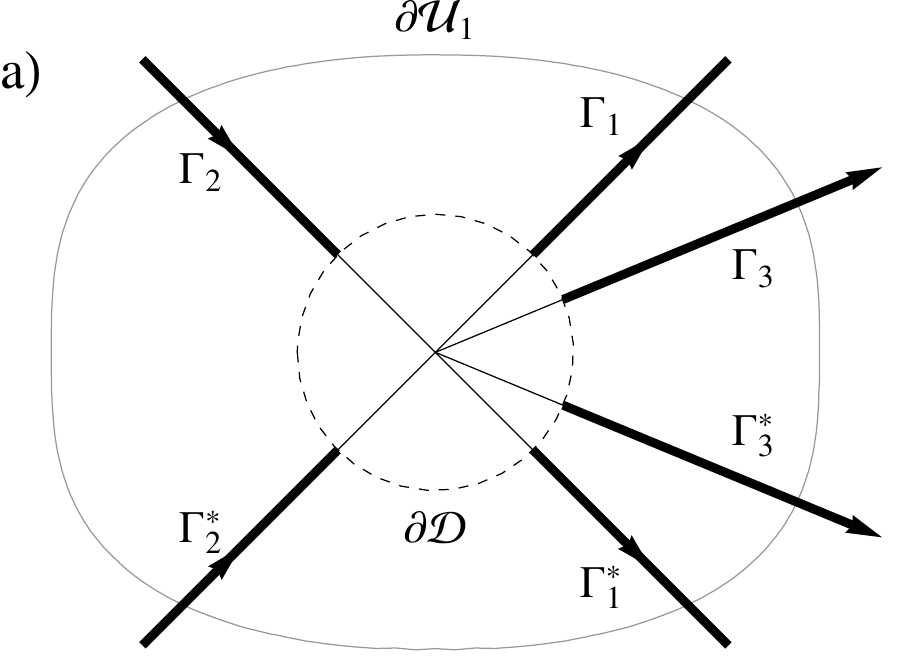}
\includegraphics[width =.3\textwidth]{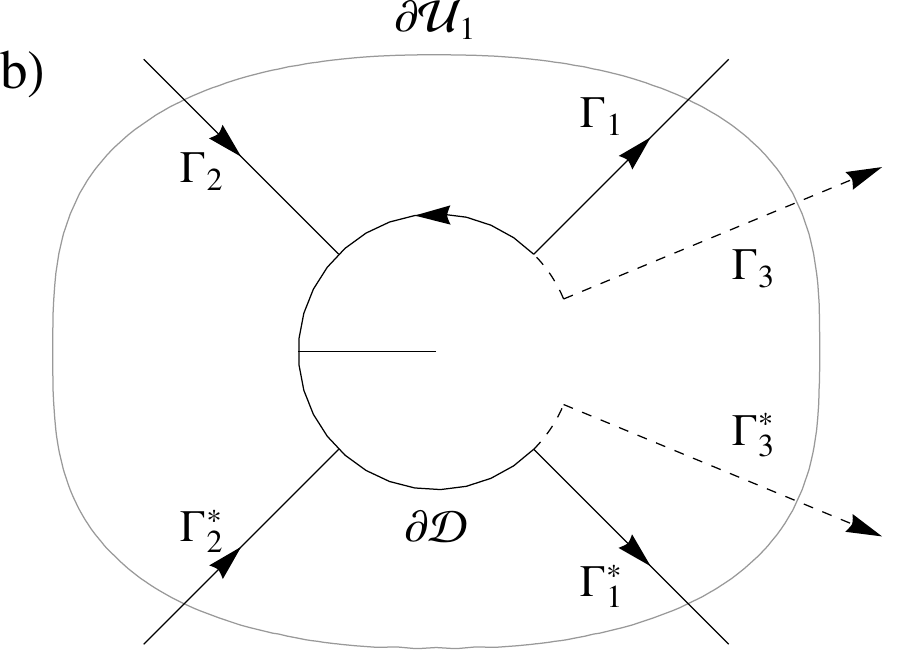}
\includegraphics[width =.3\textwidth]{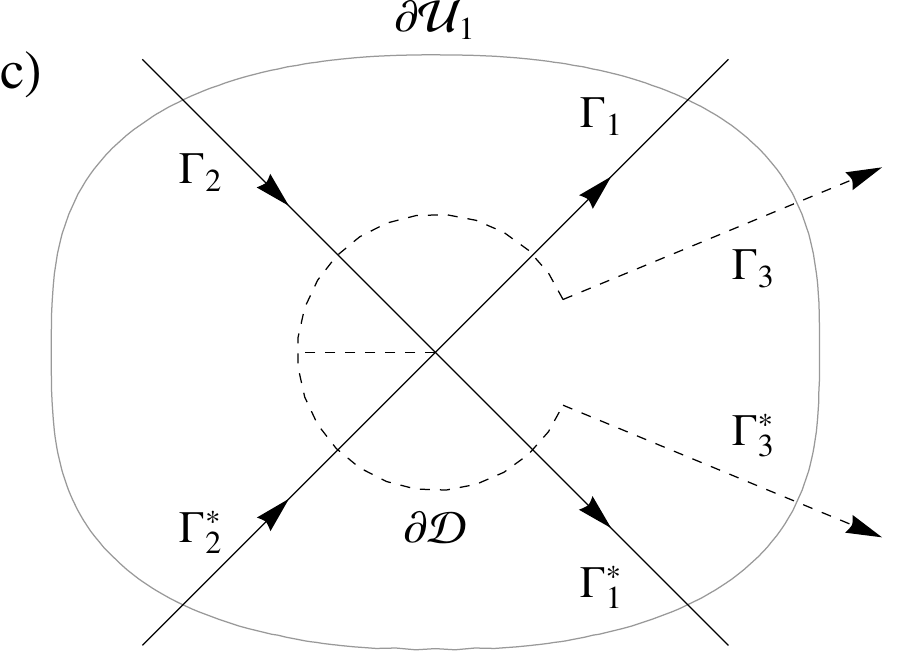}
\caption[Conditioning the local model at $\xi_1$]{The three steps of approximation in the construction of the model problem at $\xi_1$. $a)$ The jumps along the $\Gamma_k$ are replaced by the approximation \eqref{xi_1  annular approx} in the annular region $\U \backslash \DD$ (bold contours). $b)$ The exact jumps inside $\DD$ are all folded onto the real axis to the left of $\xi_1$ creating jumps on $\partial \DD$ and $\Sigma_4 = \DD\cap \{ z < \xi_1 \}$. $c)$ The annular approximation are unfolded inside $\DD$ from $\Sigma_4$. Dashed contours denote jumps which are asymptotically small after the folding and unfolding and thus ignored in the construction of the parametrix. 
\label{fig: xi_1 model steps}
}
\end{center}
\end{figure}

Inside $\DD$, the linear exponentials cannot be replaced by zero without introducing significant errors. Instead, we use the local self consistency of the jump matrices to fold away the exact jumps and then replace them with jumps matching the parabolic cylinder model. To that end we make the following change of variables folding all of the jumps inside $\DD$ onto $\Sigma_4$:  
\begin{equation}\label{folding factorization}
	A_1^{(2)} = A^{(1)}_1 F, \\
\end{equation}
where $F$ is the piecewise analytic function
\begin{equation*}
	F = \begin{cases}
		I  & z \in \DD_0 \\
		V_1^{-1}  & z \in \DD_1 \\
		V_2^{-1} V_1^{-1}  & z \in \DD_2 \\
		V_3 V_2^{-1} V_1^{-1}  & z \in \DD_3 \\
	\end{cases}
	\qquad
	F  = \begin{cases}
		I & z \in \U_1 \backslash \DD \\
		V_1^{\dagger}  & z \in \DD_1^* \\
		V_2^\dagger V_1^\dagger  & z \in \DD_2^* \\
		V_3^{-\dagger} V_2^\dagger V_1^\dagger  & z \in \DD_3^* .
	\end{cases}
\end{equation*}
The new unknown $A^{(2)}_1$ has had the jumps removed from $\Sigma_k,\ k=1,2,3,$ and their conjugates but gets new jumps on $\partial \DD$ and $\Sigma_4 $.  A simple calculation shows that the resulting jump along $\Sigma_4$ is $(1+|r_0(z)|^2)^{-\sig}$, while along $\partial \DD$ the induced jumps are given by ${A_{1-}^{(1)}}^{-1} A_{1+}^{(2)}$. This factorization is useful because it has pushed the contribution of the locally linear harmonics off of the contours $\Sigma_k$ and onto the disk boundary $\partial (\DD \backslash \DD_0)$ where it is exponentially small away from the real axis. Along the arcs $\partial \DD_3$ and $\partial \DD_3^*$, which approach the real axis, the induced jumps are independent of the linear harmonics: the jump on $\partial \DD_3$ is given by 
\begin{align*}
	F = V_3 V_2^{-1} V_1^{-1}  
	= \tbyt{ (1+r_0 r_0^*)^{-1} }{ \frac{-r_0^*}{(1+r_0 r_0*)^2} \delta^{2} e^{-i \theta_1/\eps} } 
	{r_0(1+r_0 r_0^*)\delta^{-2} e^{i \theta_1/\eps} }{1}, \text{ for } z \in \partial \DD_3.
\end{align*}
The jump on $\partial \DD_3^*$ is similar. Thus, the locally linear harmonics appear only in the jumps on $\partial (\DD_1 \cup \DD_2)$ and its conjugate where their contribution is exponentially small.  

With the exact jumps folded away, we now introduce a second factorization which unfolds onto each $\Gamma_k \cap \DD$ the approximation \eqref{xi_1 annular approx}. Define
\begin{align}\label{A3 def} 
	A_1^{(3)} &= A_1^{(2)} U, 
\end{align}
where $U$ is the piecewise analytic function
\begin{equation*}
	U = \begin{cases}
	 h_1^{\ad\sig} \tril{-r_0(\xi_1) \zeta_1(z)^{-2i\kappa(\xi_1)} e^{i\zeta_1(z)^2/ 2} }_{\Bspace} & z \in \DD_2 \\
	 h_1^{\ad\sig} \tril{-r_0(\xi_1) \zeta_1(z)^{-2i\kappa(\xi_1)} e^{i\zeta_1(z)^2/ 2} } \triu{ \frac{r_0^*(\xi_1)}{1+|r_0(\xi_1)|^2} \zeta_1(z)^{2i\kappa(\xi_1)} e^{-i\zeta_1(z)^2/ 2} }_{\Bspace} & z \in \DD_3 \\
     h_1^{\ad\sig} \triu{r_0(\xi_1)^* \zeta_1(z)^{2i\kappa(\xi_1)} & e^{-i\zeta_1(z)^2/ 2} }_{\Bspace}  & z \in \DD_2^* \\ 
	 h_1^{\ad\sig} \triu{r_0(\xi_1)^* \zeta_1(z)^{2i\kappa(\xi_1)} e^{-i\zeta_1(z)^2/ 2} } \tril{ \frac{-r_0(\xi_1)}{1+|r_0(\xi_1)|^2} \zeta_1(z)^{-2i\kappa(\xi_1)} e^{i\zeta_1(z)^2/ 2} }_{\Bspace} & z \in \DD_3^* \\
	I & \text{elsewhere}
	\end{cases}
\end{equation*}
By first folding away the exact jumps \eqref{folding factorization} and then unfolding the parabolic cylinder model jumps \eqref{A3 def}, the new problem has, by construction, the exact parabolic cylinder jumps along $(\Gamma_k \cup \Gamma_k^*) \cap \DD$, $k=1,2,3$. The sequence of factorizations induces jumps 
along the new contours $\partial \DD$ and $\Sigma_4$, which are all near identity. As an example, for $z \in  \partial \DD_2$,
\begin{equation}\label{xi_1 error sample}
\begin{split}
	FU &= \tril{\frac{r_0(1+r_0 r_0^*)}{1-r_0^2 e^{4Li\nu} } \delta^{-2} e^{i\theta_1/\eps} }
	\triu{\frac{-r_0^*}{1+r_0 r_0*} \delta^2 e^{-i\theta_0/\eps} } \times \\
	& \qquad \qquad 
	\tril{ -r_0(\xi_1)(1+r_0 r_0^*) (\delta_{1+}^{hol})^{-2} \left[ \sqrt{\frac{\eps}{\theta_1''(\xi_1)}}\zeta_1 \right]^{-2i\kappa(\xi_1)} e^{i\theta_1/\eps}} \\
	&= \lp \delta e^{\frac{-i}{2\eps} \theta_1} \rp^{\ad \sig} \tril{(1+r_0 r_0^*)\left[ r_0 -  r_0(\xi_1) \frac{\left[ \sqrt{\frac{\eps}{\theta_1''(\xi_1)}}\zeta_1\right]^{-2i\kappa(\xi_1)} }{(z-\xi_1)^{-2i\kappa}} \right]   } + \bigo{e^{-\frac{c}{\sqrt{\eps} }} } \\
	&= I+ \bigo{\sqrt{\eps}\log \eps} 
\end{split}
\end{equation}
The calculations in the other sectors are similar with the largest contributing error always being $\bigo{\sqrt{\eps}\log \eps}$ coming from the approximation of $(z-\xi_1)^{i\kappa}$ by $\lp \sqrt{\eps/ \theta_1''(\xi_1)} \zeta_1 \rp^{i\kappa(\xi_1)}$. 

Our local model is constructed by simply dropping the near identity jumps along $\DD$ and $\Sigma_4$. This leaves only the jumps on $\Gamma_1,\ \Gamma_2$ and their conjugates given by \eqref{xi_1 annular approx} which depend on $z$ only through $\zeta_1(z)$. Comparing $V_{\xi_1}$ with \eqref{Psi_0 jumps} we see that the functions are identical up to the substitution of the constant $-r_0(\xi_1)$ for $r_0(\xi_0)$. Thus, $A^{(3)}(\zeta_1(z))$ should solve RHP \ref{rhp: PC} up to replacement of the appropriate constants. To avoid repeating details, we simply state that our local model in $\U_1$ is completed by taking  
\begin{equation}\label{A3 solution}
	A^{(3)}_1(z) = h_1^\sig \, \Psi_{PC} \lp \zeta_1(z),-r_0(\xi_1) \rp \, h_1^{-\sig}
\end{equation}
where $\Psi_{PC}$ is the function defined in Prop. \ref{prop: PCRHP solution} built from the parabolic cylinder functions. 

\subsection{The error matrix, $E(z)$. Proof of Theorem \ref{thm: main}, part $1.$}
Here we prove that the parametrix $P(z)$ constructed in the previous subsection is a uniformly accurate estimate of the exact solution $Q(z)$ to RHP \ref{rhp: Q} which was derived from the original NLS RHP by explicit transformations. We prove this by considering the error matrix $E(z)$ defined as the ratio:
\begin{equation*}\label{error def}
	E(z) = Q(z) P^{-1}(z).
\end{equation*}  
Both $Q$ and the parametrix $P$ are piecewise analytic functions taking continuous boundary values on the contours that bound their respective domains of analyticity. As such, $E(z)$ also satisfies a Riemann-Hilbert problem with jump relation $E_+ = E_- V_E$, where 
\begin{equation}\label{error jump def}
V_E = P_- \lp V_Q V_P^{-1} \rp P_-^{-1}.
\end{equation}
Let $\Gamma_E$ denote the totality of contours on which $E$ has a nontrivial jump. 
Then $E$ satisfies the following Riemann-Hilbert problem:
\begin{figure}[htb]
\begin{center}
\includegraphics[width = .6\textwidth]{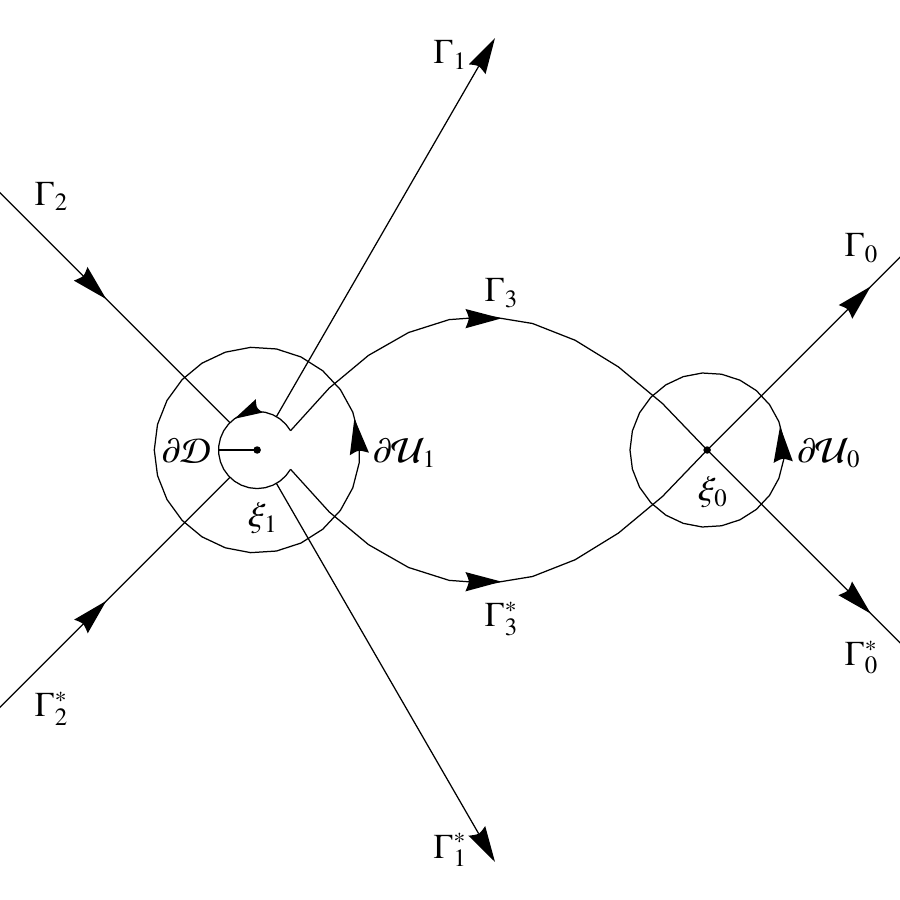}
\caption{The collection, $\Gamma_E$, of jump contours for the error matrix, E(z). \label{fig: error contours outside}
}
\end{center}
\end{figure}
\begin{rhp}[for the error matrix, E(z).]\label{rhp: outside error}
 Find a matrix $E(z)$ such that
\begin{enumerate}[1.]
\item $E(z)$ is analytic for $z \in \C \backslash \Gamma_E$.
\item As $z \rightarrow \infty$, $E(z) = I + \bigo{1/z}$
\item For $z \in \Gamma_E$, $E$ takes continuous boundary values satisfying $E_+ = E_- V_E$, where
	$V_E$ is defined by \eqref{error jump def},
\end{enumerate}
\end{rhp}
We now shift our perspective and think of $E$ not as defined by \eqref{error def}, but as the solution of RHP \ref{rhp: outside error} given the jump matrix $V_E$ which we can calculate explicitly from \eqref{Q jumps} and \eqref{outside parametrix}. The following lemma will allow us to establish a uniform asymptotic expansion for $E(z)$.

\begin{lem}\label{lem: outside error} For each $(x,t) \in K \subset \mathcal{S}_0$ compact and $\ell \in \N_0$ the jump matrix $V_E$ defined by \eqref{error jump def} satisfies,
\begin{equation}\label{outside error bound}
	\| z^\ell(V_E - I) \|_{L^p(\Gamma_E)} = \bigo{ \eps^{1/2} \log \eps},
\end{equation}
for each sufficiently small $\eps$ and $p=1,2, \text{ or }\infty$.
\end{lem}

\begin{proof} For each $(x,t)$ in a compact subset $K$ of $\mathcal{S}_0$, the factorizations defining the RHP for $Q$ can be successfully carried through and the distance between $\zeta_1$ and $\zeta_0$ is bounded below. For any allowable $(x,t)$, the parametrix $P$ is uniformly bonded in the plane, so it suffices to show that $\|(z^\ell(V_Q V_P^{-1} - I) \| =\bigo{ \eps^{1/2}\log \eps}$. The collection of contours $\Gamma_E$ decouples into three categories: the portions of $\Gamma_k,\ k=1,2,3$ and its conjugates lying outside the disks $\U_0$ and $\U_1$, these we denote $\Gamma_E^{out}$; the portions of each $\Gamma_k$ inside $\U_0$ and $\U_1 \backslash \DD$ together with the real segment $\Sigma_4$ we denote $\Gamma_E^{in}$; and finally the loop contours $\partial \U_0, \partial \U_1$, and $\partial \DD$ denoted by $\Gamma_E^{circ}$. We verify \eqref{outside error bound} independently on a representative contour of each subset, the derivation on the other contours being similar.  Consider first the unbounded set of contours $\Gamma_E^{out}$ as represented by $\Gamma_0^{out}$. In this region the parametrix is given by the outer model $O(z)$ which is analytic for $z$ off the real axis.  Thus,
\begin{align*}
	\| z^\ell(V_E - I) \|_{L^p(\Gamma_0^{out})} &\leq M \| z^\ell(V_Q -I) \|_{L^p(\Gamma_0^{out})} \\
	&\leq  M  \| z^\ell \cdot r_0 \delta^{-2}  e^{i\theta_0 / \eps} \|_{L^p(\Gamma_0^{out})} = \bigo{e^{- c /\eps} }.
\end{align*}
The last line follows from the boundedness of $r_0$ and $\delta$ and that, along $\Gamma_0^{out}$, $\imag \theta_0$ increases without bound from a least value $c >0$.The error in the exterior region is small beyond all orders and does not contribute to \eqref{outside error bound}.

On the remaining finite length contours proving that $\| V_Q V_P^{-1} - I \|_{L^\infty} = \bigo{ \eps^{1/2} \log \eps}$ implies the result for each norm and all moments $z^\ell$. For the contours in $\Gamma_E^{in}$ take $\Gamma_1^{in}$ as representative. For $z \in \DD$, $V_P = V_Q$ so there is no jump on the inside of $\DD$ we need only consider the segment within $\U_1 \backslash \DD$. Using Prop. \ref{prop: xi_1 harmonic est}, \eqref{Q jumps}, \eqref{xi_1 scalar conjugate}, and \eqref{xi_1 annular approx} we have
\begin{align*}
	&\| V_Q V_P^{-1} -I \|_{L^\infty(\Gamma_1^{in})} = \| R R_0^{-1} \delta^{-\sig} h_1^\sig V_{\xi_1}^{-1} h_1^{-\sig}\delta^\sig - I \|_{L^\infty(\Gamma_1^{in})} 
\end{align*}
Now let  $M := \max \left\{ \| h_1 \delta^{-1} \|_{L^\infty(\Gamma_1^{in})}, \| h_1^{-1} \delta \|_{L^\infty(\Gamma_1^{in})} \right\}$.Then,
\begin{align*}
	&\| V_Q V_P^{-1} -I \|_{L^\infty(\Gamma_1^{in})} \\
	\leq & M^2  \left\| \tril{ -r_0  e^{i\theta_1/ \eps} \sqrt{ \frac{\theta''_1(\xi_1)}{\eps} }^{-2i\kappa(\xi_1)} (z-\xi_1)^{-2i\kappa} } \tril{ r_0(\xi_1) e^{i\theta_1/ \eps} \zeta_1^{-2i\kappa} } - I \right\| 
	  + \bigo{e^{-c/ \eps}} \\
	 =& M^2\left\| \left[ -r_0 \sqrt{ \frac{\theta''_1(\xi_1)}{\eps} }^{-2i\kappa(\xi_1)} (z-\xi_1)^{-2i\kappa} - r_0(\xi_1)  \zeta_1^{-2i\kappa}\right] e^{i\theta_1/\eps} \right\|_{L^\infty(\Gamma_1^{in})} + \bigo{e^{-c/ \eps}} \\
	 =& \ \bigo{ \sqrt{\eps} \log \eps }.
\end{align*}
The final estimate above follows from \eqref{zeta_1 def} and the locally quadratic behavior of $\theta_1$ near $\xi_1$. 

Finally consider the disk boundaries $\partial \U_0,\ \partial \U_1$, and $\partial \DD$. The exact unknown $Q$ is analytic here so $V_Q=I$. On the positively oriented boundaries of the two fixed sized disk $V_P =  h_k^\sig \Psi_{PC}(\zeta_k(z); (-1)^k r_0(\xi_k) ) h_k^{-\sig}$, $k=0,1$. The $\eps^{-1/2}$ scaling in \eqref{zeta_0 outside} and \eqref{zeta_1 def}  and the large $\zeta$ asymptotics of $\Psi_{PC}$ then give
\begin{align*}
	\| V_P^{-1} - I \|_{L^\infty(\partial \U_k)} = \| h_k^\sig \left\{ I + \bigo{1/\zeta_k} \right\} h_k^{-\sig} - I \|_{L^\infty(\partial \U_k)}  \leq \bigo{\sqrt{\eps}}.
\end{align*}
The last contour to consider is the shrinking disk $\partial \DD$. Using \eqref{xi_1 scalar conjugate}-\eqref{A3 def} we have $ \| V_P^{-1} - I \|_{L^\infty(\partial \DD)} = \| \delta^{-\sig} (F U-I) \delta^{\sig} \|_{L^\infty(\partial \DD)}= \bigo{\sqrt{\eps} \log \eps}$ where the last equality was previously verified in \eqref{xi_1 error sample}.
\end{proof}

Lemma \ref{lem: outside error} establishes $E$ as a small-norm Riemann-Hilbert problem. As such, its solution is given by
\begin{equation*}
	E(z) = I + \frac{1}{2\pi i} \int_{\Gamma_E} \frac{ \mu(s) (V_E(s) - I) }{s-z} \dd s
\end{equation*}
where $\mu(s)$ is the unique solution of $(1 -  C_{V_E}) \mu = I$. Here $C_{V_E} f = C_- [ f (V_E - I)]$ where $C_-$ denotes the Cauchy projection operator. In particular, $E(z)$ has the large $z$ expansion  
\begin{equation}\label{E expansion outside}
	E(z) = I + \frac{E^{(1)} (x,t)}{z} + \ldots, \quad \text{where,} \quad \left| E^{(1)} (x,t)\right| = \bigo{\sqrt{\eps} \log \eps}.
\end{equation} 

\begin{rem} The small norm theory of RHPs as it pertains to this problem can be found in \cite{BC84} and \cite{BDT88}. In particular, we need $L^2$ bounds on the Cauchy operators over the contour $\Gamma_E$. For the $\eps$ independent analytic contours in $\Gamma_E$ this is a classical result. More care is needed on the $\eps$-dependent shrinking contour $\partial \DD$; the boundedness of the Cauchy operators on shrinking circles can be found in \cite{KMM03}.
\end{rem}

The proof of part $1$ of our main result follows immediately from \eqref{E expansion outside}. By inverting the series of explicit transformations $m \mapsto M \mapsto Q \mapsto E$ we arrive at a uniform asymptotic expansion of the solution $m$ of RHP \ref{rhp: square barrier}. Using \eqref{M outside}, \eqref{Q def outside}, and \eqref{outer model outside} we have 
\begin{equation*}
	\psi(x,t) =  \lim_{z \rightarrow \infty} 2iz\,  m_{12}(z;x,t) = \bigo{\sqrt{\eps} \log \eps }
\end{equation*}
for each fixed $(x,t) \in \mathcal{S}_0$.  

%
%

\section{Inverse problem inside the support, before breaking \label{sec: one band}}

\begin{figure}[htb]
\begin{center}
\includegraphics[width = .3\textwidth]{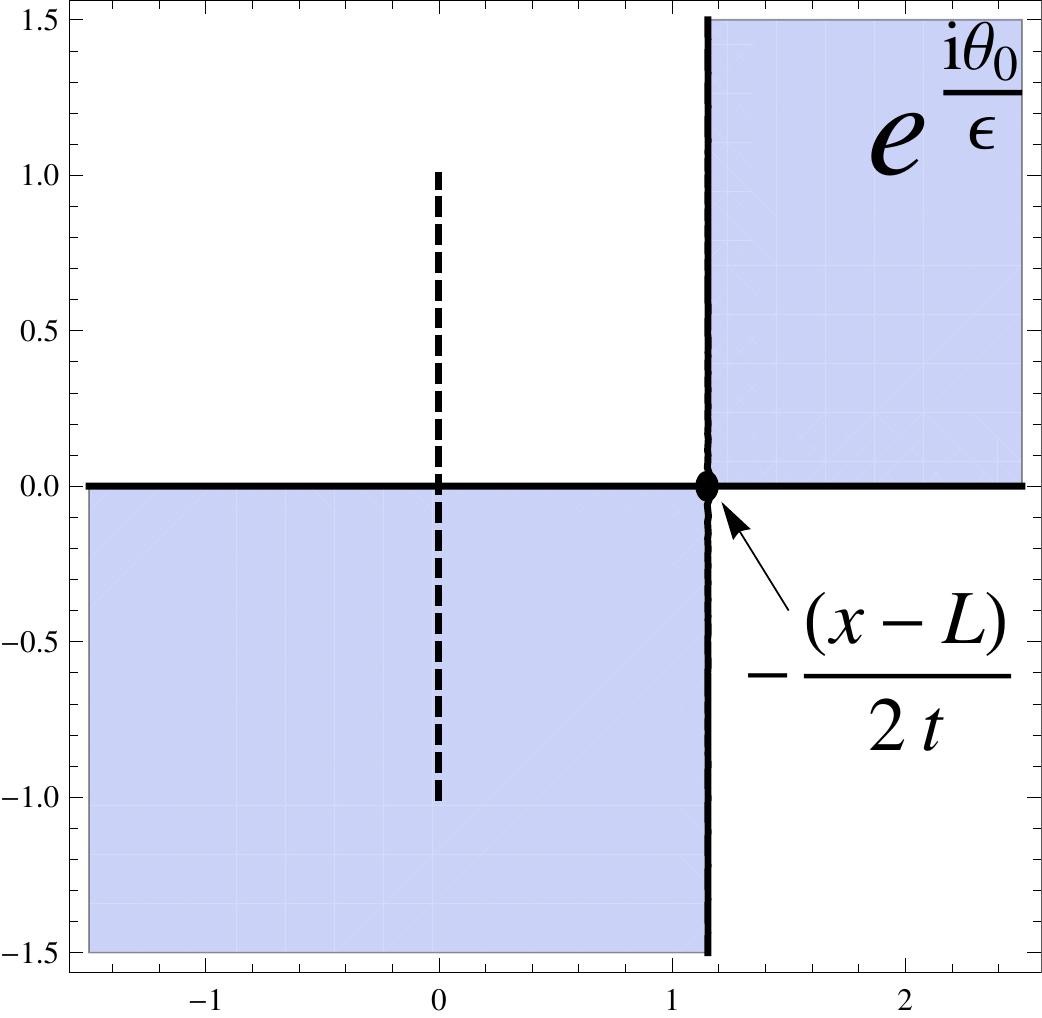}
\caption[First harmonics inside the initial support]{The regions of growth and decay of the zero harmonic $\exp\lp i \theta_0(z)/\eps \rp$ for $x \in [0, L)$ . Shaded/Unshaded regions represent regions of decay/growth. The dashed line represents the locus of pole accumulation for $m(z)$.
\label{fig: harmonics inside support}
}
\end{center}
\end{figure}
\subsection{Introducing a new pole removing factorization}
In this section we being to consider the inverse problem for those $x\in [0,L)$, that is, for $x$ inside the support of the initial data. The analysis in this regime is made more difficult by the fact that the poles of $m$ in $\C^+$ now lie in the region where $e^{i\theta_0/ \eps}$ is exponentially large, see Figure \ref{fig: harmonics inside support}. The exponential growth of $e^{i\theta_0/\eps}$ leads us to introduce a different pole removing scheme (see \eqref{M inside} below) for $x \in [0,L)$. As we will see, the new scheme introduces contours with exponentially large jumps which we control by introducing a so called $g$-function transformation which effectively ``preconditions" the problem for steepest-descent analysis. In this section we show that for each $x \in [0,L)$ there exist an order-one finite time $T_1(x)$ such that the problem is controlled by the introduction of a genus-zero $g$-function, i.e., branched on a single interval, for all $0< t< T_1(x),$ and that the corresponding solution $\psi(x,t)$ of NLS with initial data \eqref{square barrier} is asymptotically described by a slowly modulated plane wave. We consider the regime $x \in [0,L)$ and $t > T_1(x)$ in Section \ref{sec: two band}.
  
To begin our analysis, we introduce a new pole removing factorization which accounts for the exponential growth induced by the lowest order harmonic. Let $\Gamma_1$ be a semi-infinite contour in $\C^+$ leaving the real axis at a point $\xi_1$ and oriented toward infinity; denote by $\Omega_M$ the region consisting of everything to the right of $\Gamma_1$ in $\C^+$, see Figure \ref{fig: M contour inside}. For now all we demand of $\xi_1$ and $\Gamma_1$ is that the locus of pole accumulation in $\C^+$ is contained within $\Omega_M$:
\begin{equation}\label{pole removal condition}
	(0,iq] \subset \Omega_M.
\end{equation}	
In the course of our analysis additional conditions will arise which will serve to more precisely define $\xi_1$ and $\Gamma_1$.  
\begin{rem} For the inverse problem inside the support we recycle notation from the analysis outside the support. Contours and points associated with opening lenses in the RHPs are given the same names as before but new definitions, while the functions defining the various matrix factorizations, $R, \widehat{R}_0^\dagger$, etc., keep their previous definitions. 
\end{rem}   

\begin{figure}[htb]
	\begin{center}
		\includegraphics[width = .35\textwidth]{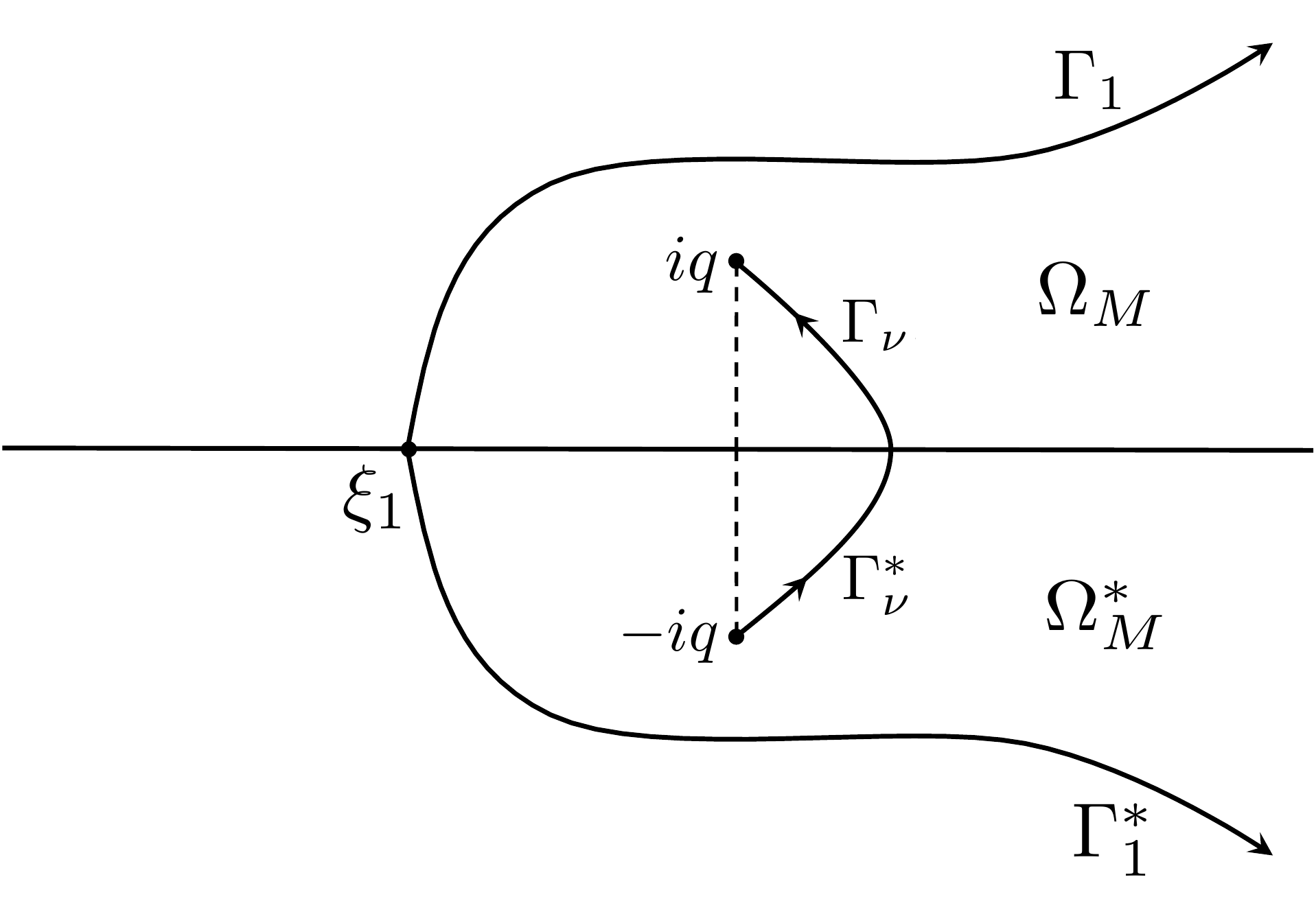}
		\caption{The contours $\Gamma_1$, $\Gamma_1^*$ and sets $\Omega_M$ and $\Omega_M^*$ defining the pole removing factorization for $x \in [0,L)$. The exact shape of the contours $\Gamma_1$ and $\Gamma_1^*$, and thus the set $\Omega_M$, $\Omega_M^*$, and the point $\xi_1$ are for now undefined. All we require is that $\Omega_M \cup \Omega_M^*$ contain the locus of pole accumulation (dashed line) and the branch $\Gamma_\nu \cup \Gamma_\nu^*$ of $\nu = \sqrt{z^2+q^2}$. 
		\label{fig: M contour inside}      
		}
	\end{center}
\end{figure}
Define
\begin{equation}\label{M inside}
	M = \begin{cases}
		m R^{-1} R_0 & z \in \Omega_M \\
		m R^\dagger R_0^{- \dagger} & z \in \Omega_M^* \\
		m & \text{elsewhere}.
		\end{cases}
\end{equation}
The matrices $R$ and $R_0$ are those defined previously, \eqref{R} and \eqref{R_0} respectively. Clearly, the new unknown $M$ has no poles; the combination $m R^{-1}$ is analytic in $\C^+$--this was the basis of our pole removal scheme for $x$ outside the support--while the extra factor $R_0$ has no poles. Additionally, it follows from \eqref{r-r_0} that the nonzero off-diagonal entry of  $R R_0^{-1}$ is at leading order proportional to $e^{i\theta_1 / \eps}$ and is thus unaffected by the exponential growth of $e^{i\theta_0 / \eps}$. 

Though the new unknown $M$ has no poles it acquires an extra jump in addition to those on the lens boundaries. The matrix $R_0$ \eqref{R_0} involved in the pole removing factorization inherits a branch cut from $\nu = \sqrt{z^2+q^2}$ which we are free to choose; we take the cut to be a finite length contour $\Gamma_\nu \cup \Gamma_\nu^*$ oriented from $-iq$ to $iq$. Note that as sets $\Gamma_\nu$ and $\Gamma_\nu^*$ are symmetric, but as contours with orientation they are antisymmetric. We defer $\Gamma_\nu$'s exact definition and assume for now only that it lies completely within $\Omega_M \cup \Omega_M^*$. The jumps acquired by $M$ along the branch cut are 
\begin{align}
	  {R_0}^{-1} {R_0}_+  &= \tril{ w e^{i\theta_0/ \eps} } \quad &z &\in \Gamma_\nu, \\
	  {R_0^\dagger}_- {R_0^{-\dagger}}_+ &=  \triu{ - w^* e^{-i\theta_0/ \eps} }, &z &\in \Gamma_\nu^*.
\end{align}
Here we have defined the scalar function  
\begin{equation}\label{w}
	w:= {r_0}_+ - {r_0}_- = \frac{2\nu_+}{iq},
\end{equation}
which we note may be analytically continued away from $\Gamma_\nu \cap \Gamma_\nu^*$. The RHP for $M$ then follows directly from \eqref{M inside}.
\begin{rhp}[for $M$:]\label{rhp: M inside support} 
Let $\Gamma_M = \R \cup \Gamma_1 \cup \Gamma_1^* \cup \Gamma_\nu \cup \Gamma_\nu^*$. Find a matrix $M(z)$ such that 
\begin{enumerate}
	\item $M$ is analytic in $\C \backslash \Gamma_M$.
	\item As $z \rightarrow \infty$, $M(z) = I + \bigo{1/z}$. 
	\item $M$ assumes continuous boundary values $M_+$ and $M_-$ on $\Gamma_M$ satisfying $M_+ = M_- V_M$ where,
	\begin{equation}\label{M jump inside}
	V_M = \begin{cases}
		R^\dagger R & z \in (-\infty, \xi_1) \\
		R_0^\dagger R_0 & z \in (\xi_1, \infty) \\
		\tril{ w e^{i\theta_0/\eps} }_{\Bspace} & z \in \Gamma_\nu \\
		\triu{-w^* e^{-i\theta_0/ \eps} } & z \in \Gamma_\nu^* \\
		R_0^{-1} R & z \in \Gamma_1 \\
		R^\dagger R_0^{-\dagger} & z \in \Gamma_1^* .
		\end{cases}
	\end{equation}	
\end{enumerate}
\end{rhp}

\subsection{Introducing the g-function}\label{sec: one band g}
The pole removing factorization introduces jumps on $\Gamma_\nu \cup \Gamma_\nu^*$ which are ill-suited to semi-classical asymptotics. The off-diagonal entries of the jump $V_M$ along $\Gamma_\nu$ and  $\Gamma_\nu^*$  are proportional to $e^{i\theta_0/\eps}$ and  $e^{-i\theta_0/\eps}$ respectively, which increase exponentially as $\eps \to 0$. To account for these jumps we introduce a scalar $g$-function, by making the change of variables
\begin{equation}\label{g trans}
	N =  M e^{-i g(z) \sig / \eps}.
\end{equation}  
The new unknown $N$ should solve a RHP of the type we have consider thus far, placing the following restrictions on $g$:
\begin{enumerate}[i.]
 \item $g(z) = \bigo{1/z}$ as $z \rightarrow \infty$.
 \item There exist a symmetric contour $\gamma \cup \gamma^*$ such that $g$ is analytic for $z \in \C \backslash (\gamma \cup \gamma^*)$ with continuous boundary values $g_+$ and $g_-$ on $\gamma \cup \gamma^*$.
  \item $g(z)^* =  g(z^*)$ for all $z \in \C \backslash (\gamma \cup \gamma*)$.
 
\end{enumerate}

Given such a $g$-function, the jump matrices for $N$ take the following form: 
\begin{align*}
	V_N &=  \left\{ \begin{array}{ll}
		 \lp \begin{array}{ll}
			\# e^{-i(g_+ - g_-)/\eps} & \# e^{i( g_+ + g _-)/ \eps} \\
			\# e^{-i( g_+ + g_-)/ \eps} & \# e^{i(g_+ - g_-)/\eps} 
		\end{array} \rp_{\Bspace} & z \in \gamma \cup \gamma^* \\
		\lp \begin{array}{ll} 
			\#  & \qquad\ \# e^{2i g / \eps}  \\
			\# e^{-2i g / \eps} & \qquad\ \# 
		\end{array} \rp & z \notin \gamma \cup \gamma^* 
	\end{array} \right.
\end{align*}
where the $\#$'s denote the entries of $V_M$.  The $g$-function affects the jumps by modifying the exponential phases in the jump matrices; away from the branch cuts of $g$ the transformation replaces the phases $\theta_k$ in the expansion of the reflection coefficient with   
\begin{equation}\label{modified phases}
	\varphi_k = \theta_k - 2g.
\end{equation}
Our procedure is to now construct $g$ so that the resulting RHP for $N$ is asymptotically stable, i.e. g should remove the exponentially growing components of $V_M$ so that, after additional lens openings, the new jump matrix $V_N$ implied by \eqref{g trans} has simple semi-classical limits with controllable errors. To construct $g$, we decompose $\gamma$, and by symmetry $\gamma^*$, into two interlacing sets of contours: bands and gaps. Each band is a maximally connected component of $\gamma$ satisfying \eqref{band cond}, and each gap is a maximal complimentary open interval satisfying \eqref{gap cond}; the union of all bands (gaps) in $\C^+$ we label $\gamma_b$ ($\gamma_g$) and label the conjugate bands (gaps) $\gamma_b^*$ ($\gamma_g^*$).  
\begin{subequations} \label{band/gap cond}
\begin{align}
	\label{band cond} &\textsc{Bands:} \quad \  \ \imag \lp  g_+ - g_- \rp = 0.  \\
	 \label{gap cond}  &\textsc{Gaps:} \quad    \left\{ \begin{array}{l}  
		\imag \lp \theta_{k(j)} - g_+ - g_- \rp  > 0 \text{  for } z \in \gamma, \\
		\imag \lp \theta_{k(j)} - g_+ - g_- \rp  < 0 \text{  for } z \in \gamma^*.
		\end{array} \right.
\end{align}
\end{subequations}
Note that in \eqref{gap cond} we have to select a phase $\theta_{k(j)}$ which may change from one interval to the next. 

As usual, in the cases where we can successfully construct the $g$-function, it associates with the Riemann-Hilbert problem a rational function $\mathcal{R}(x,t) = \sqrt{ \prod_{n=1}^{2N} (z-\alpha_n) }$ and the associated Riemann surface of genus $G=N-1$. The eventual outer model problem, and thus the leading order asymptotics of the solution of NLS, can be described in terms of the Riemann theta functions on the Riemann surface $\mathcal{R}(x,t)$.    

\subsection{Genus zero ansatz for $x \in [0,L)$ and small times}
Having no a priori way to determine the number of bands and gaps needed to construct the $g$-function, we proceed instead by ansatz. 
We begin by considering the simplest nontrivial case, that the $g$-function is analytic for $z \in \C \backslash(\Gamma_\nu \cup \Gamma_\nu*)$ and the entire contour consists of a single band interval. In this case, we seek $g$, as the solutions of the following scalar RHPs:
\begin{rhp}[genus zero $g$-function and its derivative:] \label{rhp: genus zero g}
Find a scalar function $g(z)$ such that 
\begin{equation*}
\begin{array}{rl@{\quad}|@{\quad}rl}
	1. & g \text{ analytic for } z \in \C \backslash (\Gamma_\nu \cup \Gamma_\nu^*).
		& 1. &  \rho \text{ analytic for } z \in \C \backslash (\Gamma_\nu \cup \Gamma_\nu^*)_{\Bspace}. \\
	2. & g(z) = \bigo{1/z} \text{ as } z \rightarrow \infty
		& 2. & \rho(z) = \bigo{1/z^2} \text{ as } z \rightarrow \infty_{\Bspace}. \\
	3. & g_+ +  g_- = \theta_0 + \eta, \quad z \in \Gamma_\nu \cup \Gamma_\nu^*.
		& 3. &\rho_+ +  \rho_- = \theta_0', \quad z \in \Gamma_\nu \cup \Gamma_\nu^*. \\
	4. & g =\bigo{ (z \pm iq)^{1/2}}  \text{ + locally} 
		& 4. &\rho =\bigo{ (z \pm iq)^{-1/2}}  \text{ + locally} \\
		& \text{ analytic function near } \pm iq. & & \text{ analytic function near } \pm iq.
\end{array}
\end{equation*}
\end{rhp}
\noindent
Here $\eta$ is a real constant to be determined. The RHP for the density $\rho$ follows from formally differentiating the RHP for $g$. We begin with the derivative $\rho$-problem because it does not depend on the unknown constant $\eta$. By integrating the solution of the RHP for $\rho$ we define a function $g$ which a posteriori solves the RHP for $g$ and justifies the formal differentiation.

\begin{rem}
In the typical construction of a $g$-function, at each band endpoint, $\alpha_j$, the local behavior is take to be $g = \bigo{(z-\alpha_j)^{3/2}}$. However, for square barrier initial data the asymptotic density of eigenvalues in the scattering problem \eqref{asymp density} is singular: $\rho_0(z) = \bigo{ (z \pm iq)^{-1/2} }$ unlike the usual case where eigenvalues vanish as a square root. The loss of regularity in the density motivates the $\bigo{z \pm iq)^{1/2}}$ behavior at the end points.   
\end{rem}

\begin{rem}
The single band ansatz is motivated by the observation that the initial data $\psi_0(x)$ is locally analytic for $x\in [0,L)$. As such the genus zero Whitham equations \eqref{modulation equations} are locally well-posed and so its reasonable to assume that for sufficiently small times the NLS solution, $\psi(x,t)$, is well approximated inside the support by a slowly modulating plane wave. In principle we could take the band to be a subset of $\Gamma_\nu \cup \Gamma_\nu^*$ with endpoints $\alpha(x,t)$ and $\alpha^*(x,t)$. However, in this case one can show that the resulting gap condition $\imag ( \theta_0 - g_+ - g_- ) > 0$ for $z\in \Gamma_\nu \backslash \gamma_b$ cannot be satisfied which forces $\alpha(x,t) = iq$.
\end{rem}

RHP \ref{rhp: genus zero g} for $\rho$ is easily solved using the Plamelj formula:
\begin{equation}\label{one band g'}
	\begin{split}
	\rho(z) &= \frac{1}{2\pi i \nu(z) } \int_{\Gamma_\nu \cup \Gamma_\nu^*} \theta_0'(\lambda) \nu_+(\lambda) \frac{ \dd \lambda}{\lambda - z} 
		= \frac{1}{2} \theta_0'(z) - \frac{2t z^2 + (x-L)z + t q^2 }{\nu(z)}  \\
		&= \deriv{}{z} \left\{ \frac{1}{2}\theta_0(z) - \nu(z) \left[ t z + (x-L)  \right] \right\}.
	\end{split}
\end{equation}  
Any choice of antiderivative will satisfy conditions $1, 3,$ and $4$ of RHP \ref{rhp: genus zero g} for the $g$-function; in order for $g$ to tend to zero at infinity, we define 
\begin{equation}\label{one band g}
	\begin{split}
	g(z) := \int_\infty^z \rho(\lambda) \dd \lambda  =  \frac{1}{2}\theta_0(z) - \nu(z) \left[ t z + (x-L) \right] + \frac{1}{2}t q^2,
	\end{split}
\end{equation}
where to be concrete the path of integration is taken such that it avoids the contour $\Gamma_\nu \cup \Gamma_\nu^*$ on which $\rho$ is branched. Following \eqref{modified phases}, $g$ has the effect of replacing the unmodified phases $\theta_0$ and $\theta_1$ with 
\begin{equation}\label{one cut phases}
	\varphi_0 = 2 \nu \left[  t z + (x-L) \right] - t q^2
	\quad \text{and} \quad
	\varphi_1 = 2 \nu \left[  t z + (x+L) \right] - t q^2, 
\end{equation}	
 
\subsubsection{Satisfying the band and gap conditions: contour selection}
The construction of our $g$-function is not complete; we have not yet defined the contour $\Gamma_{\nu} \cup \Gamma_{\nu}^*$ on which it is branched. This contour is selected by showing that the appropriate band and gap conditions can be satisfied. For the single band ansatz we have two such conditions:
\begin{subequations}\label{one cut band/gap cond}
\begin{align}
	\label{one cut band cond} 
		\text{\textsc{Band:}} \quad &\imag ( g_+ - g_- ) = \imag( \varphi_{0-}  ) = 0, \quad z \in \Gamma_\nu \\
	\label{one cut gap cond}
		\text{\textsc{Gap:}}  \quad  &\imag ( \theta_1 - g_+ - g_- ) = \imag (\varphi_1) >0, \quad z \in \Gamma_1
\end{align}
 \end{subequations}
We list only the conditions for $z \in \C^+$ as the reflection symmetry implies the correct condition in $\C^-$. The inequality in \eqref{one cut gap cond} is enforced so that the jumps introduced by \eqref{M inside} along $\Gamma_1$ and its conjugate are near identity away from the real axis. It is not a typical gap condition since $\Gamma_1$ is not the complement of a band, but the inequality is of the gap form and, as usual for gap inequalities, its failure can be a mechanism for the introduction of a new band in the $g$-function. 

The band and gap conditions in \eqref{one cut band/gap cond} amount to understanding the topology of the zero level curves of $\imag \varphi_k,\ k=0,1$, and the corresponding open regions of growth and decay of the associated exponentials $e^{i\varphi_k/ \eps}$. Note that the zero level sets of $\imag \varphi_k$ are independent of how we choose the branch cut of $\nu$, so using them to determine the branch is not circular. The following lemma characterizes the time evolving topology of the zero level sets of each modified phase. 
\begin{lem}\label{lem: one band zero level}
Let $L(t;b)$ denote the zero level curve of the function 
\begin{equation*}
	\imag \varphi_b := \imag \lp 2\nu \left[ t z + b \right] -t q^2 \rp
\end{equation*}
for $b \in \R \backslash \{ 0 \}$ and $\nu$ defined by \eqref{nu}. Additionally, define the breaking time
\begin{equation*}
	T_c(b) = \frac{|b|}{2\sqrt{2} q}.
\end{equation*}
Then the zero level $L(t;b)$ evolves as follows: 
\begin{enumerate}[i.]
	\item For $t=0$, $L(0;b) = \R \cup [-iq,\, iq]$.
	\item For $0 < t \leq T_c (b)$, $L(t;b)$ consists of the real axis and two simple non-intersecting arcs: a simple finite arc connecting $iq$ to $-iq$, and an infinite arc which asymptotically approaches the line $\re z = -b/2t$ for $z$ large. The finite and infinite arcs lie completely in the half-plane $ \re  bz  < 0 $ and cross the real axis at points $z_0$ and $z_1$ respectively such that $|z_0|  < |z_1|$.
	\item For $t  > T_c(b)$, $L(t;b)$ consist of the real axis and two semi-infinite contours each connecting one of the branch points $\pm iq$ to complex infinity in the same half-plane. In particular, $L(t;b)$ does not contain a finite path connecting $iq$ to $-iq$     
\end{enumerate}
\end{lem}

\begin{proof}
For $t = 0$, $\imag \varphi_b = \imag \nu$ since $b$ is real and nonzero, the topology of $L(0;b)$ follows immediately. For $t > 0$ the topology is determined from the following observations:
\begin{itemize}
	\item The conjugation symmetry $\varphi_b^*(z) = \varphi_b(z^*)$ guarantees that $\R \subset L(t;b)$ and allows us to consider only $z \in \C^+$. Any branch of $L(t;b)$ in $\C^+$ can terminate only at a critical point of $\varphi_b$ on the real line, at the branch point $iq$, or at $\infty$. 
	\item The square root singularity of $\varphi_b$ at $iq$ guarantees that exactly one branch of $L(t;b)$ terminates at $iq$.
	\item  The large $z$ expansion $\varphi_b = t z^2 + bz  + t q^2/2 +\bigo{1/z}$ shows that $L(t;b)$ has exactly one branch terminating at $\infty$ in $\C^+$ and the branch asymptotically approaches the line $\re z = -b/2t$ for large $z$.
	\item For $t>0$ no branch of $L(t;b)$ can cross the imaginary axis, since 
	\begin{equation*}
	\left| \imag \lp \varphi_b(iy) \rp \right| = \begin{cases}
		ty\sqrt{q^2-y^2} & 0 < y \leq q \\
		b\sqrt{y^2-q^2} & y > q.
	\end{cases} 		 
	\end{equation*}
	\item  $L(t;b)$ cannot contain a closed bounded loop if $\imag \varphi_b$ is harmonic inside the loop, since the maximum modulus principal would imply that $\varphi_b$ is then identically zero.
\end{itemize}
To complete the proof we need to determine the number of real critical points of $\varphi_b$. A simple calculation shows that $\varphi_b$ has exactly two critical points given by 
\begin{align*}
	z_0 = -\frac{b}{4t} \left[ 1 - \sqrt{1 - \frac{8t^2q^2}{b^2}} \right] \qquad
	z_1 = -\frac{b}{4t} \left[ 1 + \sqrt{1 - \frac{8t^2q^2}{b^2}} \right].
\end{align*}
For $0< t < T_c(b)$ these are two real points satisfying the given properties in part \emph{ii.} above. As $t$ increases and surpasses the breaking time $T_c(b)$ the two real critical points $z_0$ and $z_1$ coalesce and spit into complex conjugates. Thus, for $t > T_c(b)$ the branches of $L(t;b)$ in $\C^+$ do not have a real terminus. Moreover, for $t> T_c(b)$ the branches of $L(t;b)$ no longer pass through the critical points since this would necessarily imply the existence of a closed loop in $L(t;b)$ violating the last bulleted conditions above. In each case we now connect the termini of the branches of $L(t;b)$ in $\C^+$ without violating the above observations. The results for $t > 0$ follow immediately.
\end{proof} 

Clearly, the zero level sets of $\imag \varphi_0$ and $\imag \varphi_1$ are those described above with $b= x-L$ and $b=x+L$ respectively. To satisfy the band condition \eqref{one cut band cond} we define $\Gamma_{\nu} \cup \Gamma_{\nu}^*$ to be the branch of the zero level of $\imag \varphi_0$ connecting $-iq$ to $iq$. Lemma \ref{lem: one band zero level} guarantees this contour exists for $(x,t) \in \mathcal{S}_1$, where
\begin{equation}\label{S1}
	\begin{split}
		\mathcal{S}_1 := \left\{ (x,t)\, : \, 0 \leq \right. & x \left. <L,  \leq t < T_1(x) \right\}, \\
		T_1(x) &:= \frac{L - |x| }{ 2\sqrt{2} q } .
	\end{split}
\end{equation} 
For $(x,t) \notin \mathcal{S}_1$ the band condition \eqref{one cut band cond} cannot be satisfied. The upper limit $T_1(x)$ we call the \emph{first breaking time}, and in the remainder of this section we will always assume that $(x,t) \in \mathcal{S}_1$; we will address what happens when one moves beyond this first breaking time in Section \ref{sec: two band}. 

The last condition to be satisfied is the gap condition \eqref{one cut gap cond}. For $(x,t) \in \mathcal{S}_1$ the topology of the level curve $\imag \varphi_1=0$ falls always into case \emph{ii.} of Lemma \ref{lem: one band zero level}. The regions of growth and decay of the exponentials $e^{i\varphi_0 / \eps}$ and $e^{i \varphi_1 / \eps}$ follow from continuation of their large $z$ asymptotic behavior; plots of these regions for generic values of $(x,t) \in \mathcal{S}_1$  are given in Figure \ref{fig: one band harmonic growth/decay}. As the figure illustrates, the finite and infinite complex branches of $\imag \varphi_1=0$ do not intersect so the subset of $\C^+$ such that   $\imag \varphi_1 >0$ consist of the entire region to the right of the infinite complex branch of  the zero level set except for the closed set bounded by the finite branch of $\imag \varphi_1$ and the branch cut of $\nu$ along $\Gamma_\nu \cup \Gamma_\nu^*$. To satisfy \eqref{one cut gap cond} we define $\xi_1$ to be the real critical point of $\varphi_1$ through which the infinite branch of $\imag \varphi_1 = 0$ passes:
\begin{equation}\label{one band xi_1}
	\xi_1 = -\frac{x+L}{4t} \left[ 1 + \sqrt{1 - \frac{8t^2q^2}{(x+L)^2} }\, \right]
 \end{equation}
and take $\Gamma_1$ to be any contour in $\C^+$ terminating at $\xi_1$ laying completely in the set $\imag \varphi_1 > 0$, see Figure \ref{fig: one band Q contours}. 

\subsection{Removing the remaining oscialltions: $N \mapsto Q$}
\begin{figure}[htb]
\begin{center}
	\includegraphics[width = .3\textwidth]{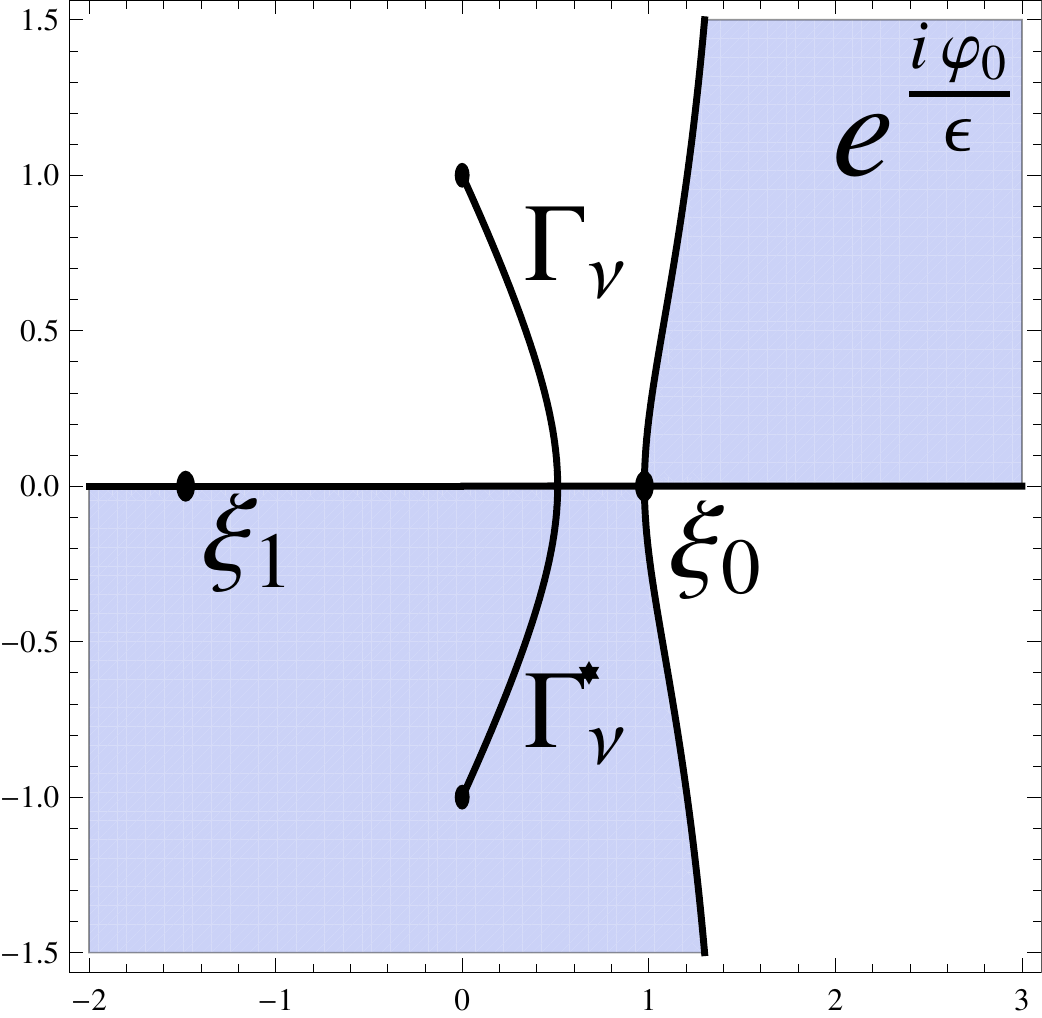}
	\qquad \qquad
	\includegraphics[width = .3\textwidth]{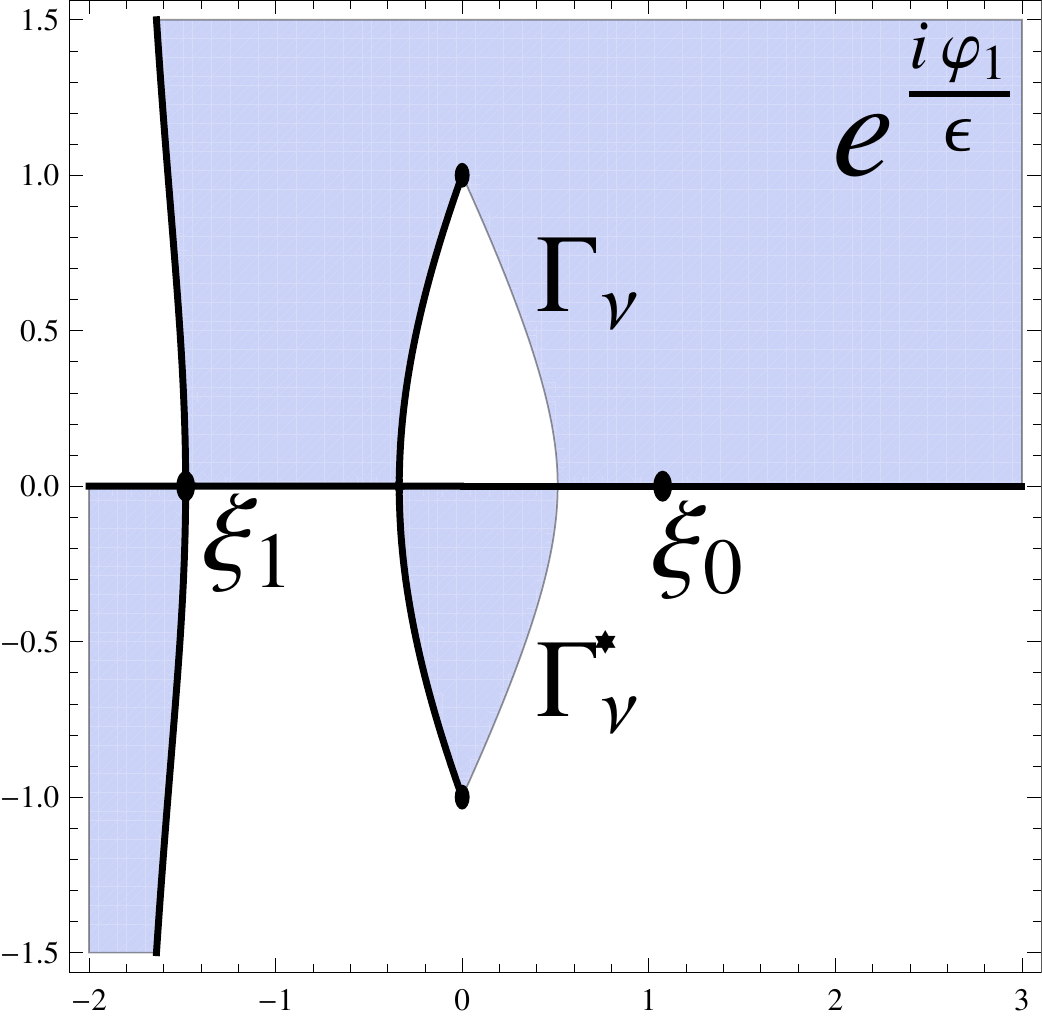}
	\caption{The modified harmonics $e^{i\varphi_k /\eps},\ k=0,1,$ each separate the plane into regions of growth (white) and decay (grey). The level set $\imag \varphi_0 = 0$ (bold lines) separating the regions of growth and decay of $e^{i \varphi_0/\eps}$ consists of three elements: the real line, a finite arc connecting $\pm iq$ and an unbounded arc which approaches the line $\imag z = -\frac{(x-L)}{2t}$ for large $z$. We take the branch of $\nu =\sqrt{z^2+q^2}$ along the finite branch of $\imag \varphi_0 = 0$ which we label $\Gamma_\nu$ in $\C^+$ and $\Gamma_\nu^*$ in $\C^-$ respectively. The level set $\imag \varphi_1=0$ consists of similar arcs with the unbounded arc now approaching $\imag z = -\frac{(x+L)}{2t}$ for large $z$. These three arcs and the branch cut $\Gamma_\nu \cup \Gamma_\nu^*$ separate the regions of growth and decay of $e^{i \varphi_1/\eps}$. Finally, the points $\xi_k,\, k=1,2$ defined by \eqref{one band xi_0} and \eqref{one band xi_1} are the unique points at which the unbounded arcs of $\imag\varphi_k =0$ cross the real line and from which contours are opened.	
	\label{fig: one band harmonic growth/decay}
	}
\end{center}
\end{figure}
Now that the construction of the one band $g$-function is complete the RHP for $N$ follows immediately from RHP \ref{rhp: M inside support} and \eqref{g trans}. 
\begin{rhp}[for $N(z)$:]\label{rhp: one band N} Find a $2 \times 2$ matrix $N$ such that
\begin{enumerate}[1.]
	\item $N$ is analytic in $\C \backslash \Gamma_M$. 
	\item As $z\rightarrow \infty,\ N(z)=I + \bigo{1/z}$. 
	\item $N$ assumes continuous boundary values $N_+$ and $N_-$ on $\Gamma_M$ satisfying $N_+ = N_-  V_N$ where,
	\begin{equation}\label{one band N jumps}
		V_N = \begin{cases}
			e^{ig \ad \sig/ \eps} \lp R^\dagger R \rp & z \in (-\infty, \xi_1) \\
			e^{ig \ad \sig/ \eps}( R_0^\dagger R_0) & z \in (\xi_1, \infty) \\
			\begin{pmatrix}
				{e^{-i(\varphi_{0-}+ t q^2)/\eps} }	&	{0} \\
				{we^{-itq^2/ \eps} }	&	{e^{-i(\varphi_{0+} + tq^2) / \eps} }
			\end{pmatrix}_{\Bspace} & z \in \Gamma_\nu \\
			\begin{pmatrix}
				{e^{-i(\varphi_{0-}+ t q^2)/\eps} }	&	{-w^*e^{itq^2/ \eps} } \\
				{0}	&	 {e^{-i(\varphi_{0+} + tq^2) / \eps} }
			\end{pmatrix}_{\Bspace} & z\in \Gamma_\nu^* \\
			e^{ig \ad \sig/ \eps} \lp R_0^{-1} R \rp & z \in \Gamma_1\\
			e^{ig \ad \sig/ \eps} \lp R^\dagger R_0^{-\dagger} \rp & z \in \Gamma_1^*
		\end{cases}
	\end{equation}
\end{enumerate}
\end{rhp}
The $g$-function transformation has removed the exponentially large jumps along $\Gamma_\nu \cup \Gamma_\nu^*$. However, the RHP for $N$ still has rapidly oscillatory jumps along the real axis which must be factored and moved onto appropriate complex contours. To open steepest descent lenses off the real axis we need to understand the regions of growth and decay of the, now modified,  first harmonics $e^{i\varphi_0 / \eps}$ and $e^{i\varphi_1/ \eps}$. The topology of these regions follows directly from Lemma \ref{lem: one band zero level}. The level sets $\imag \varphi_k=0,\  k=0.1$ each have one infinite, asymptotically vertical, branch which cross the real axis at points which we label $\xi_0$ and $\xi_1$ respectively; $\xi_1$ was given above by \eqref{one band xi_1}, following the proof of Lemma \ref{lem: one band zero level} a similar formula can be given for $\xi_0$:
\begin{equation}\label{one band xi_0} 
	\xi_0 = -\frac{x-L}{4t} \left[ 1 + \sqrt{1 - \frac{8t^2q^2}{(x-L)^2} } \right].
\end{equation}
These are stationary phase points of the associated harmonics $e^{i\varphi_k/ \eps}$. We denote them in this way in analogy with the stationary phase points of the same name for the unmodified phases $\theta_k,\ k=0,1$ which were described in Section \ref{sec: harmonic expansion} for $x$ outside the support of the initial data. Just as before, these points separate connected unbounded components of the complex plane in which the associated harmonic is everywhere either decaying or growing. These regions are plotted numerically in Figure \ref{fig: one band harmonic growth/decay}.

The oscillatory jumps on the real axis are essentially identical to those in RHP \ref{rhp: M outside} up to the conjugation by the $g$-function factor. To deform the oscillations in the current problem onto contours on which they are exponentially small we use same factorizations as before (cf. \eqref{left factorization}-\eqref{intermediate factorization}). Without repeating those details we make the following transformation:
\begin{equation}\label{one band Q}
\begin{split}
	Q(z) &= N e^{ig \ad \sig/ \eps} \, L_Q \\
	L_Q &= \begin{cases} 
		R_0^{-1} & z \in \Omega_0 \\
		\widehat{R}^{-\dagger} \lp a / a_0 \rp^{\sig} & z \in \Omega_2 \\
		\widehat{R}_0^{-\dagger} & z \in \Omega_3 \\
		\widehat{R}_0 & z \in \Omega_3^* \\
		\widehat{R}  \lp a^* / a^*_0 \rp^{-\sig}  & z \in \Omega_2^* \\
		R_0^\dagger & z \in \Omega_0^* \\
		I & \text{elsewhere}.
	\end{cases}	
\end{split}
\end{equation}
The regions $\Omega_k$ and the contours $\Gamma_k$ bounding them are defined as follows: we take $\Gamma_0$ to be a contour in $\C^+$ lying completely in the region $\imag \varphi_0 > 0$ and meeting the real axis at $\xi_0$; $\Gamma_2$ lies completely in the region $\imag \varphi_1 < 0$ and meets the real axis at $\xi_1$; and finally $\Gamma_3$ connects the stationary phase points $\xi_1$ and $\xi_0$ and encloses $\Gamma_\nu$ always staying in the region $\imag  \varphi_0 < 0$. The corresponding sets $\Omega_k,\ k=1,2,3$ are those enclosed by $\Gamma_k$ and the real axis, see Figure \ref{fig: one band Q contours}. 
\begin{figure}[htb]
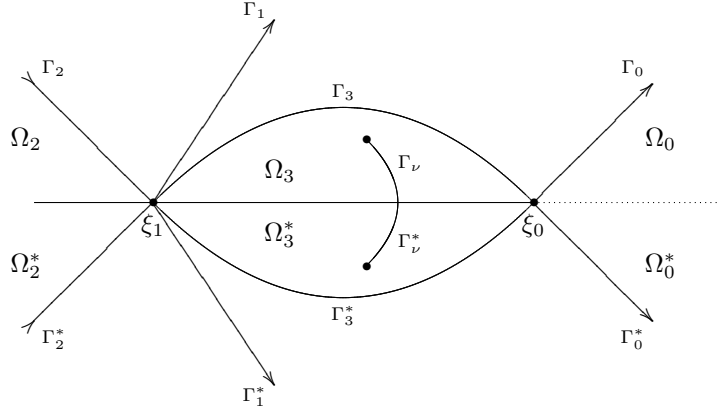

\centerline{
\begin{xymatrix}@!0{
& & & & \ar@{<-}[3,-2]+0_<{\Gamma_1}& & & & & & & \\
\ar@{>-}[2,2]+0^<{\Gamma_2} & & & & &  & & & & &\ar@{<-}[2,-2]+0_<{\Gamma_0} & \\
\Omega_2 & & & & \ar@{}[1,0]|{\displaystyle\Omega_3} & \ar@{{}{-}{*}}@/^1pc/@<2ex>[2,0]+0^(.25){\Gamma_\nu}  & & & &  &  \Omega_0 \\
\ar@{-}[0,8]+0  \ar@{{}{}{*}}[0,2]+0 \ar@{}@<-2ex> [0,4]|{\displaystyle \xi_1} & & \ar@{-}@/^3pc/[0,6]+0^(.5){\Gamma_3}     \ar@{-}@/_3pc/[0,6]+0_(.5){\Gamma_3^*} & & & & & & \ar@{-}@/^3pc/[0,-6]+0 \ar@{-}@/_3pc/[0,-6]+0 & & & \ar@{{}{.}{*}}[0,-3]+0 \ar@{}@<2ex>[0,-6]|{\displaystyle \xi_0} \\
\Omega_2^* & & & & \ar@{}[-1,0]|{\displaystyle\Omega_3^*} &  \ar@{{}{-}{*}}@/_1pc/@<-2ex>[-2,0]+0_(.25){\Gamma^*_\nu}  & & & & & \Omega_0^* \\
\ar@{>-}[-2,2]+0_<{\Gamma_2^*} & & & & & & & & & &\ar@{<-}[-2,-2]+0^<{\Gamma_0^*} &\\
& & & &\ar@{<-}[-3,-2]+0^<{\Gamma_1^*}& & & & & & &\\
}
\end{xymatrix} } 
\caption[Contours diagram for $Q(z)$]{ Schematic diagram of the contours $\Gamma_Q$ and regions $\Omega_k$ used to define the transformation $N \mapsto Q$. 
\label{fig: one band Q contours}
}
\end{figure}

To calculate the resulting jumps of $Q$ we need the factorizations \eqref{left factorization}, \eqref{between factorization}, \eqref{intermediate factorization}, and the following factorization. For $z$ along $\Gamma_\nu$, the jump of $V_N$ can be factored as follows:
\begin{multline*}
	\begin{pmatrix}  
		{e^{-\frac{i}{\eps}(\varphi_{0-}+ t q^2)}  }	&	{0} \\
		{we^{-\frac{i}{\eps} t q^2} }	&	{e^{-\frac{i}{\eps}(\varphi_{0+} + tq^2)}  }
	\end{pmatrix}	 \\  
	= \triu{w^{-1} e^{-\frac{i}{\eps} \varphi_{0-}}  } \offdiag{-w^{-1} e^{\frac{i}{\eps} t q^2} }{w e^{ -\frac{i}{\eps} t q^2} } 
	\triu{w^{-1} e^{-\frac{i}{\eps} \varphi_{0+}} }.
\end{multline*}	
Next, recalling the definition of $r_0(z)$, \eqref{r_0}, and $w$, \eqref{w}, we observe the following identity
\begin{equation}\label{R hat sing}
	\lp \frac{r_0^*}{1+r_0 r_0^*} \rp_+ = -\lp \frac{r_0^*}{1+r_0 r_0^*} \rp_-    
	= \frac{iq}{2\nu_+} = w^{-1}, 
\end{equation}	
which allows one to express the above factorization in terms of the boundary values of $\widehat{R}^\dagger$:
\begin{equation}\label{band factorization}
	\begin{split}
	   \begin{pmatrix}
		{e^{-\frac{i}{\eps}(\varphi_{0-}+ t q^2)}  }	&	{0} \\
		{we^{-\frac{i}{\eps} t q^2} }	&	{e^{-\frac{i}{\eps}(\varphi_{0+} + tq^2)}  }
	\end{pmatrix} & \\	 
	=  \lp e^{\frac{i}{\eps} g \ad \sig} \widehat{R}_0^{- \dagger} \rp_-  & \offdiag{-w^{-1} e^{\frac{i}{\eps} t q^2} }{w e^{ -\frac{i}{\eps} t q^2} } 
	\lp e^{\frac{i}{\eps} g \ad \sig} \widehat{R}_0^\dagger \rp_+.
\end{split}
\end{equation} 	
This allows us to open a single lens $\Gamma_3$ enclosing the band $\Gamma_\nu$ of the $g$-function instead of opening separate lenses off the real axis and again off $\Gamma_\nu$ which significantly simplifies the construction of the ensuing global parametrix. Let $\Gamma_Q^0 = \bigcup_{k=0}^3 \lp \Gamma_k \cup \gamma_k^* \rp$ and $\Gamma_Q = \Gamma_Q^0 \cup \Gamma_\nu \cup \Gamma_\nu^* \cup (-\infty, \xi_0]$. Using this factorization we arrive at: 
\begin{rhp}[ for $Q(z)$:]\label{rhp: one band Q}
Find a $2 \times 2$ matrix valued function $Q$ of $z$ such that:
\begin{enumerate}[1.]
	\item $Q$ is analytic for $z \in \C \backslash\Gamma_Q$. 
	\item $Q(z) = I + \bigo{1/z}$ as $z \rightarrow \infty$.
	\item $Q$ takes continuous boundary values $Q_+$ and $Q_-$ on $\Gamma_Q$ satisfying the jump relation $Q_+ = Q_-  V_Q$ where, 
	\begin{equation}\label{one band Q jumps}
		V_Q = \begin{cases}
		 	e^{i g \sig /\eps}\, V_Q^{(0)}\, e^{-i g \sig / \eps} & z \in \Gamma_Q^0 \\
			 \offdiag{-w^{-1} e^{\frac{i}{\eps} t q^2} }{w e^{ -\frac{i}{\eps} t q^2} }_{\Bspace} & z \in \Gamma_\nu \\
			 \offdiag{-w^* e^{\frac{i}{\eps} t q^2} }{{w^*}^{-1} e^{ -\frac{i}{\eps} t q^2} } & z \in \Gamma_\nu^* \\
			 (1+ | r_0 |^2)^{2\sig} & z \in (-\infty, \xi_1) \\
			 (1+| r_0 |^2)^\sig & z \in (\xi_1, \xi_0) 
		\end{cases}, \\
	\end{equation}	
	and
	\begin{equation}\label{one band V_Q^0 jumps} 
		V_Q^{(0)} = \begin{cases}
			R_0 & z\in \Gamma_0 \\
			R_0^{-1} R & z\in \Gamma_1 \\
			(a/a_0)^{-\sig}\widehat{R}^\dagger &  z\in \Gamma_2 \\
			\widehat{R}_0^\dagger & z \in \Gamma_3 \\
		\end{cases}
		\qquad
		V_Q^{(0)} = \begin{cases}
			R_0^\dagger & z  \in \Gamma_0^* \\
			R^\dagger R_0^{-\dagger} & z \in \Gamma_1^* \\
			\widehat{R}(a^*/a_0^*)^{-\sig} & z\in \Gamma_2^* \\
			\widehat{R}_0  & z \in \Gamma_3^* 
		\end{cases}
	\end{equation}
	\item $Q$ admits at worst square root singularities at $z = \pm iq$ satisfying the following bounds
	\begin{equation}\label{one band Q sing}
	\begin{split}
		Q(z) &= \bigo{ \begin{array}{cc} 1 & |z - iq|^{-1/2} \\  1 & |z - iq|^{-1/2} \end{array} }, \quad z \rightarrow iq,  \\
		Q(z) &= \bigo{ \begin{array}{cc} |z + iq|^{-1/2} & 1 \\  |z + iq|^{-1/2} & 1 \end{array} }, \quad z \rightarrow -iq.
	\end{split}
	\end{equation}
\end{enumerate}
\end{rhp}

\begin{rem}
The new singularity condition in RHP \ref{rhp: one band Q} is caused by a singularity in \eqref{one band Q} defining the transformation $N \mapsto Q$.   
In $\Omega_3$, which contains $z=iq$, $Q = N e^{i g \ad \sig / \eps} \widehat{R}_0^{-\dagger}$ and observing \eqref{R hat sing} the $(2,1)$-entry of $\widehat{R}_0^{-\dagger}$ has a square root singularity at $iq$.  We must admit the possibility that $Q$ inherits this singularity in its second column. As we encounter such transformations in the rest of this paper we will record the resulting singularity conditions without comment.
\end{rem}  

\subsection{Constructing a  global Parametrix: one band case}
The result of the many factorizations defining $Q$ is a RHP which is well conditioned for asymptotic approximation. Away from the points $\xi_0$ and $\xi_1$ where they return to the real axis the jumps \eqref{one band V_Q^0 jumps} along $\Gamma_Q^0$ are all near identity, both for large $z$ and fixed $z$ as $\eps \downarrow 0$. The remaining jumps defining $V_Q$ behave simply in the semiclassical limit. The jumps on the real axis are $\eps$-independent and the twist-like jumps on $\Gamma_\nu \cup \Gamma_\nu^*$ depend on $\eps$ through an oscillatory constant which we can deal with explicitly. We proceed, as before, to construct a global parametrix $P$ such that the ratio $QP^{-1}$ is uniformly near identity. Since the jumps along $\Gamma_Q^0$ are not uniformly close to identity near the stationary phase points $\xi_0$ and $\xi_1$, we seek our global parametrix in the form
\begin{equation}\label{one band parametrix}
	P(z) = \begin{cases}
		A_0(z) & z \in \U_0 \\
		A_1(z) & z \in \U_1 \\
		O(z) & \text{elsewhere}
	\end{cases}
\end{equation}
where each $\U_k$ is a fixed size neighborhood of $\xi_k$, $k=1,2,$ which we will determined in the construction of the approximation $A_k$. 
\subsubsection{The Outer Mode, $O(z)$.l}     
At any fixed distance from the stationary phase points the jumps defined by $V_Q^{(0)}$ in RHP \ref{rhp: one band Q} are exponentially small perturbations of identity so we ignore them. Doing so, we arrive at the following problem for the outer model. 
\begin{rhp}[ for $O(z)$:]\label{rhp: one band O}
Find a $2 \times 2$ matrix valued function $O$ satisfying the following properties:
\begin{enumerate}[1.]
	\item $O$ is an analytic function for $z \in \C \backslash \Gamma_O, \quad \Gamma_O =  \Gamma_\nu \cup \Gamma_\nu^* \cup (-\infty, \xi_0].$
	\item $O(z) = I + \bigo{1/z}$ as $z \rightarrow \infty$
	\item $O$ takes continuous boundary values $O_+$ and $O_-$ on $\Gamma_O$ away from the points $z= \pm iq, \ \xi_0$, and $\xi_1$. The boundary values satisfy the jump relation $O_+ = O_- V_O$ where
	\begin{equation}\label{O jumps}
		V_O =  \begin{cases}
		 	\offdiag{-w^{-1} e^{\frac{i}{\eps} t q^2} }{w e^{ -\frac{i}{\eps} t q^2} }_{\Bspace} & z \in \Gamma_\nu \\
			\offdiag{-w^* e^{\frac{i}{\eps} t q^2} }{{w^*}^{-1} e^{ -\frac{i}{\eps} t q^2} } & z \in \Gamma_\nu^* \\
			(1+ | r_0 |^2)^{2\sig} & z \in (-\infty, \xi_1) \\
			(1+| r_0 |^2)^\sig & z \in (\xi_1, \xi_0). 
		\end{cases}
	\end{equation}
	\item $O$ admits at worst square root singularities at $z = \pm iq$ satisfying the following bounds
	\begin{equation}\label{one band O sing}
	\begin{split}
		O(z) &= \bigo{ \begin{array}{cc} 1 & |z - iq|^{-1/2} \\  1 & |z - iq|^{-1/2} \end{array} }, \quad z \rightarrow iq,  \\
		O(z) &= \bigo{ \begin{array}{cc} |z + iq|^{-1/2} & 1 \\  |z + iq|^{-1/2} & 1 \end{array} }, \quad z \rightarrow -iq.
	\end{split}
	\end{equation}
\end{enumerate}
\end{rhp}
We construct a solution of this outer model by first introducing two scalar functions which reduces the problem to the class of constant jumps Riemann-Hilbert problems. The first of these functions: 
\begin{equation*}
	\delta(z) = \exp \lp \frac{1}{2\pi i} \int_{-\infty}^{\xi_1} \frac{\log \lp 1 + |r_0(s)|^2 \rp }{s-z} \dd s 
	+  \frac{1}{2\pi i} \int_{-\infty}^{\xi_0} \frac{\log \lp 1 + |r_0(s)|^2 \rp }{s-z} \dd s \rp,
\end{equation*}	
we have already encountered; it was introduced in RHP \ref{rhp: O outside} to remove the jumps on the real axis from the outer model problem. The second function we need is new, and is the solution of the following problem: 
\begin{rhp}[ for $s(z)$:]\label{rhp: one band s} 
	Find a scalar function $s$ such that 
		\begin{itemize}
			\item $s$ is analytic in $\C \backslash \lp \Gamma_\nu \cup \Gamma_\nu^* \rp$.
			\item $s(z) = s_\infty + \bigo{1/z}$ as $z \rightarrow \infty$, where $s_\infty$ is a constant, bounded for all $\eps$.				\item For $z \in \Gamma_\nu \cup \Gamma_\nu^*$ the boundary values $s_+$ and $s_-$ satisfy the relationship
				\begin{equation}\label{one band s jumps}
					s_+(z) s_-(z) = \begin{cases}
						w(z) \delta^{-2}(z) e^{-i t q^2/ \eps} & z  \in \Gamma_\nu \\
						w^*(z)^{-1} \delta^{-2}(z) e^{-i t q^2 / \eps} & z  \in \Gamma_\nu^*.
					\end{cases}
				\end{equation}
			\item At the endpoints of $\Gamma_\nu \cup \Gamma_\nu^*$:
				\begin{enumerate}[i.]
					\item $s(z) = \bigo{ (z-iq)^{1/4} }$ as $z \rightarrow iq$,
					\item $s(z) = \bigo{ (z+iq)^{-1/4} }$ as $z \rightarrow -iq$.
				\end{enumerate}
		\end{itemize}
\end{rhp}

Let us defer, temporarily, the construction of the solutions to RHP \ref{rhp: one band s}. If such a function $s$ can be found, we make the following transformation 
\begin{equation}\label{one band constant jump prob}
	O(z) = s_\infty^{-\sig} O^{(1)}(z) s(z)^\sig \delta(z)^\sig. 	
\end{equation}
The new unknown $O^{(1)}$ then satisfies the following simple, constant jump, Riemann-Hilbert problem:
\begin{rhp}[ for $O^{(1)}(z)$:]\label{rhp: one band constant jump}
Find a $2 \times 2$ matrix valued function $O^{(1)}(z)$ such that 
\begin{enumerate}[1.]
	\item $O^{(1)}$ is analytic in $\C \backslash \lp \Gamma_\nu \cup \Gamma_\nu^* \rp$. 
	\item $O^{(1)}(z) = I + \bigo{1/z}$ as $z \rightarrow \infty$.					
	\item Away from its endpoints $O^{(1)}$ takes continuous boundary values $O^{(1)}_+$ and $O^{(1)}_-$ on $\Gamma_\nu \cup \Gamma_\nu^*$ satisfying the jump realation
	\begin{equation*}
		O^{(1)}_+ (z) = O^{(1)}_- (z) \offdiag{-1}{1}
	\end{equation*}
	\item $O^{(1)}$ admits at worst $1/4$-root singularities at the endpoints $z = \pm iq$.
\end{enumerate}					
\end{rhp}						
This is a standard model problem which often appears in problems related to integrable PDE and random matrix theory. One builds the solution by diagonalizing the jump matrix and solving the resulting simple scalar problem. A moments work reveals 
\begin{equation*}
	O^{(1)}(z) = \begin{pmatrix}
		\frac{ a(z) + a(z)^{-1} }{2} &  - \frac{ a(z) - a(z)^{-1} }{2i} \\
		\frac{ a(z) - a(z)^{-1} }{2i} & \frac{ a(z) + a(z)^{-1} }{2}
		\end{pmatrix},
\end{equation*}
where $a$ is defined by,
\begin{equation}\label{a}
	a(z) = \lp \frac{z-iq}{z+iq} \rp^{1/4}
\end{equation}
and the root is understood to be branched along $\Gamma_\nu \cup \Gamma_\nu^*$ and normalized to approach unity at $\infty$. 

To complete the description of the outer model problem we need to build the solution of RHP \ref{rhp: one band s} which is the subject of the following proposition.
\begin{prop}\label{prop: one band s}
Riemann-Hilbert problem \ref{rhp: one band s} is solved by 
\begin{equation}\label{one band s}
	s(z) = a(z) e^{-\frac{i}{2}t q^2/\eps} \exp \lp \frac{\nu(z) } {2\pi i} \int_{\Gamma_\nu \cup \Gamma_\nu^*} \frac{ j(\lambda) }{ \nu_+(\lambda) } \frac{\dd \lambda}{\lambda - z} \rp.
\end{equation}
Here $a$ is as defined in \eqref{a}, and using \eqref{kappa} we define
\begin{equation*}
	j(z) = \begin{cases}
		\log \lp \frac{2(z+iq)}{q} \rp - 2( \chi(z,\xi_1) + \chi(z,\xi_0) )  & z \in \Gamma_\nu \\
		\log \lp \frac{q}{2(z-iq)} \rp -  2( \chi(z,\xi_1) + \chi(z,\xi_0) )   & z \in \Gamma_\nu^*.
		\end{cases}
\end{equation*}
where the logarithms are understood to be principally branched. As $z \rightarrow \infty$, $s(z) = s_\infty + \bigo{1/z}$ where 
the constant $s_\infty$ satisfies
\begin{equation}\label{one band s asymp}
	\begin{split}
		s_\infty^{-2} &= \exp \left[ \frac{i}{\eps} t q^2 - \frac{i\pi}{2} + i \omega(x,t) \right]
		\qquad \text{where,} \\
		\omega(x,t) &= -\frac{1}{ \pi} \lp \int_{-\infty}^{\xi_1} - \int_{\xi_0}^\infty \rp \frac{ \log(1+ |r_0(\lambda)|^2) }{ \nu(\lambda) } \dd \lambda.
\end{split}
\end{equation}
\end{prop}

\begin{rem}
The function $\omega$ defined by \eqref{one band s asymp} is real and thus represents a slow correction to the fast phase. The integrals can be evaluated in terms of special functions to give
\begin{equation*}
	\omega(x,t) = -\frac{1}{2 \pi} \left[ \Li_2 \lp r_0(\xi_0)^2 \rp + \Li_2 \lp r_0(\xi_1)^2 \rp   \right], 
\end{equation*}
where $\Li_2(z)$ is the dilogarithm and the $(x,t)$ dependence is contained in the location of the stationary phase points $\xi_0$ and $\xi_1$.
\end{rem}

\begin{proof}
The Plamelj formulae imply that the given function \eqref{one band s} satisfies the analyticity and jump conditions of RHP \ref{rhp: one band s}. Using classical results on the behavior of singular integrals along the line of integration one can show that the Cauchy integral and its boundary values are bounded except at the endpoints where at worst it behaves like an inverse square root \cite{Musk46}. However, this behavior is balanced by the presence of the factor $\nu(z)$ multiplying the Cauchy integral. Thus, \eqref{one band s} is everywhere bounded except at $z= \pm iq$ where its local behavior is identical to that of $a(z)$. Finally, by expanding \eqref{one band s} for large $z$ we get an asymptotic expansion for $s$ at infinity in powers of $z^{-1}$; the leading order term is constant: 
\begin{equation*}
	s_\infty = e^{-\frac{i}{2} t q^2/ \eps} \exp \lp -\frac{1}{2\pi i} \int_{\Gamma_\nu \cup \Gamma_\nu^*} \frac{ j(\lambda) }{ \nu_+(\lambda) }\dd \lambda \rp.
\end{equation*}
By using contour deformation arguments and exploiting the odd symmetry of $\nu(z)$ for $\{z: |z|>\xi_0, \ z\in \R \}$ and even symmetry for $\{z: |z|<\xi_0, \ z\in \R \}$ one shows that $s_\infty$ satisfies \eqref{one band s asymp}. 
\end{proof}

Note that $O(z)$ only depends on $\eps$ through the phase factors $e^{ i t q^2 / 2 \eps}$ appearing in $s(z)$ and $s_\infty$ and is thus bounded independent of $\eps$ at each point $z \in \C \backslash \{iq,\, -iq\}$. Near the endpoints $z = \pm iq$, $O(z)$ has square root singularities coming from the product $O^{(1)}(z) s(z)^\sig$ as described by \eqref{one band O sing}. As is expected, and will be proven in Section \ref{sec: one band error}, the outer model contributes the leading order behavior of the solution of NLS in the semiclassical limit; expanding $O(z)$ for large $z$ and taking the $(1,2)$-entry of the first moment we have
\begin{equation}\label{one band O asymp}
	2i \lim_{z \rightarrow \infty} O_{12}(z;\, x,t) = i q s_\infty^{-2}    = q e^{ \frac{i}{\eps} t q^2 + i \omega(x,t) }.
\end{equation}

\subsubsection{The local model problems}
The approximation of $Q(z)$ by the outer model is not uniform near $z = \xi_1$ and $z= \xi_0$ where the jump contours $\Gamma_Q^0$ (cf. \eqref{one band Q jumps}) return to the real axis. To construct a uniformly accurate parametrix we must introduce local models $A_0(z)$ and $A_1(z)$ defined on neighborhoods $\U_0$ and $\U_1$ of the points $\xi_0$ and $\xi_1$ respectively which account for the local structure of the jump matrices and which asymptotically match the outer solution on the boundaries $\partial \U_0$ and $\partial \U_1$. 

Comparing \eqref{one band Q jumps} to \eqref{Q jumps}, the local structure of $V_Q$ near $\xi_0$ and $\xi_1$ when $(x,t) \in \mathcal{S}_1$, the set in which the single band g-function stabilizes the problem, and the local structure of $V_Q$ for $x$ outside the initial support are essentially identical. In moving from outside the initial support into $\mathcal{S}_1$ the introduction of the $g$-function has replaced the harmonic phases $\theta_k$ with their modulated counterparts $\varphi_k = \theta_k -2g$ and the stationary phase points $\xi_0$ and $\xi_1$ have been suitably redefined in terms of the $\varphi_k$, everything else is identical. Thus, we can essentially construct the new local models at $\xi_0$ and $\xi_1$, mutatis mutandis, from the old models. Without repeating the details of their construction we define the new local models below and refer the reader to Sections \ref{sec: outside xi_0} and \ref{sec: outside xi_1} for the additional details.  
 
First, define the locally analytic and invertible change of variables $\zeta = \zeta_k(z)$, where, 
\begin{equation}\label{one band zeta_k}
	\frac{1}{2} \zeta_k^2 := \frac{ \varphi_k(z) - \varphi(\xi_k) } { \eps} = \frac{ \varphi_k''(\xi_k) }{2\eps} (z-\xi_k)^2 + \bigo{ (z-\xi_k)^3 }, \quad k=0,1.
\end{equation}  
Each $\zeta_k$ introduces a rescaled local coordinate at $\xi_k$ and we choose suitably small fixed size neighborhoods $\U_k$ of $\xi_k$ such that the $\zeta_k$ are analytic inside $\U_k$ and the images $\zeta = \zeta_k(\U_k)$ are disks in the $\zeta$-plane. Next, using Prop.~\ref{prop: delta expansions} and \eqref{delta expansions} we define the nonzero scaling functions:
\begin{equation}\label{one band h's}
	\begin{split}
	h_0 &= \left[ \lp \frac{\eps}{\varphi_0''(\xi_0)} \rp^{i\kappa(\xi_0)} \lp \delta_0^{\text{hol}}(z) \rp^2 e^{-i\varphi_0(\xi_0) / \eps} \right]^{1/2}, \\
	h_1 &= \left[ \lp \frac{\eps}{\varphi_1''(\xi_1)} \rp^{i\kappa(\xi_1)} \delta_{1+}^{\text{hol}}(z)  \delta_{1-}^{\text{hol}}(z) \ e^{-i\varphi_1(\xi_1) / \eps} \right]^{1/2}. 
	\end{split}
\end{equation}

With the above definitions in hand we define our local models as follows:
\begin{align}
	A_k(z) = \left[ s_\infty^{-\sig}\, O^{(1)}(z)\,  s(z)^\sig \right] \hat{A}_k(z)\ \delta(z)^\sig, \quad k=0,1,\label{one band local form} 
\end{align}
where the functions $\hat{A}_k$ are built from the solutions $\Psi_{PC}$ of the parabolic cylinder local model, RHP \ref{rhp: PC}: 
\begin{subequations}\label{one band local models}
	\begin{align}
	\hat{A}_0(z) &= h_0^\sig \Psi_{PC} ( \zeta_0(z),\, r_0(\xi_0) ) h_0^{-\sig}, \label{ one band xi_0 model}\\
	\hat{A}_1(z) &= h_1^\sig \Psi_{PC} ( \zeta_1(z),\, -r_0(\xi_1) ) h_1^{-\sig} \times 
		U^{-1} e^{ig(z)\sig / \eps} F^{-1} e^{-ig(z) \sig/ \eps}. \label{ one band xi_1 model} \end{align}	
\end{subequations}	
		
 \begin{rem} 
In addition to the modified phase functions $\varphi_k$ and the resulting redefinition of the stationary phase points $\xi_k$, the pre-factors right  multiplying each $\hat{A}_k$ in \eqref{one band local form} are a new element of the local models when compared to the local models defined in Sections \ref{sec: outside xi_0} and \ref{sec: outside xi_1}. It consists of the factors in the outer model \eqref{one band constant jump prob} which are locally analytic near $\xi_0$ and $\xi_1$; by including them as pre-factors, our estimates of the asymptotic error on the boundaries $\partial \U_k$ are identical to the previous estimates for $x$ outside the support of the initial data. 
\end{rem}

\subsection{Proof of Theorem \ref{thm: main},  part two}\label{sec: one band error}
We now consider the error matrix $E(z)$ defined as the ratio:
\begin{equation}\label{one band error def}
	E(z) = Q(z) P^{-1}(z).
\end{equation} 
As both $Q(z)$ and $P(z)$ are piecewise analytic functions whose components take continuous boundary values on the boundaries of their respective domains of definition, $E$ satisfies its own Riemann-Hilbert problem, which we give below.
\begin{rhp}[ for the error matrix, E(z).]\label{rhp: one band error}
Find a $2 \times 2$ matrix $E(z)$ such that
\begin{enumerate}[1.]
	\item $E(z)$ is bounded and analytic for $z \in \C \backslash \Gamma_E$, where $\Gamma_E$ is the system of contours depicted in Figure \ref{fig: error contours outside}.
	\item As $z \rightarrow \infty, \ E(z) = I + \bigo{1/z}$.
	\item $E$ takes continuous boundary values $E_+(z)$ and $E_-(z)$ for $z \in \Gamma_E$ satisfying the jump relation $E_+ = E_- V_E$ where
	\begin{equation}\label{one band E jumps}
		 V_E = P_- \lp V_Q V_P^{-1} \rp P_-^{-1}.
	\end{equation}
\end{enumerate}
\end{rhp}
All of the properties listed in the RHP for $E(z)$ follow immediately from \eqref{one band error def} except for the boundedness of $E(z)$ near the endpoints $z = \pm iq$. To show this, first observe that the jump matrices of the parametrix $P(z)$ exactly match those of $Q(z)$ along $\Gamma_{\nu} \cup \Gamma_\nu^*$ so that $V_E$, as defined by \eqref{one band E jumps}, are identity along $\Gamma_{\nu} \cup \Gamma_\nu^*$. At worst $E(z)$ has isolated singularities at the endpoints. However, the local growth restrictions \eqref{one band Q sing} and \eqref{one band O sing} dictate that at worst $E_{jk}(z) = \bigo{|z\pm iq|^{-1/2}}$ implying that the singularities are removable. Thus $E(z)$ is locally bounded near each endpoint. Globally boundedness then follows directly from conditions $2$ and $3$ defining RHP \ref{rhp: one band error}. 

The key estimate which allows us to prove the second part of our main result, Theorem \ref{thm: main}, is recorded in the following lemma which establishes $E$ as the solution of a small-norm Riemann-Hilbert problem.
\begin{lem}\label{lem: one band error} For $(x,t) \in K \subset \mathcal{S}_1$ compact and $\ell \in \N_0$ the jump matrix $V_E$ defined by \eqref{one band E jumps} satisfies,
\begin{equation}\label{one band error bound}
	\| z^\ell(V_E - I) \|_{L^p(\Gamma_E)} = \bigo{ \eps^{1/2} \log \eps},
\end{equation}
for each sufficiently small $\eps$ and $p=1,2, \text{ or }\infty$.
\end{lem}
\begin{proof} The estimates necessary to establish this lemma are identical to those used to prove Lemma \ref{lem: outside error} after introducing the modified phases \eqref{one cut phases} which replace the $\theta_k$ upon the introduction of the one band $g$-function \eqref{one band g} which remains valid for any $(x,t) \in \mathcal{S}_1$. The reader is referred to Lemma \ref{lem: outside error} for the details of the estimates.
\end{proof}  
The small norm estimate in Lemma \ref{lem: one band error} admits the representation of $E(z)$ as 
\begin{equation*}
	E(z) = I + \frac{1}{2\pi i} \int_{\Gamma_E} \frac{ \mu(s) (V_E(s) - I) }{s-z} \dd s
\end{equation*}
where $\mu(s)$ is the unique solution of $(1 -  C_{V_E}) \mu = I$. Here $C_{V_E} f = C_- [ f (V_E - I)]$ where $C_-$ denotes the Cauchy projection operator. In particular, $E(z)$ has the large $z$ expansion  
\begin{equation}\label{E expansion one band}
	E(z) = I + \frac{E^{(1)} (x,t)}{z} + \ldots, \quad \text{where,} \quad \left| E^{(1)} (x,t)\right| = \bigo{\sqrt{\eps} \log \eps}.
\end{equation} 

Unfolding the sequence of transformations: $m \mapsto M \mapsto N \mapsto Q \mapsto E$, we have the semi-classical asymptotic expansion of $m$ and, using \eqref{nls recovery}, we find the leading order behavior of the solution $q(x,t)$ of \eqref{nls} with initial data \eqref{square barrier} for each $(x,t) \in \mathcal{S}_1$.  Using \eqref{M inside}, \eqref{g trans}, \eqref{one band Q}, \eqref{one band parametrix}, and \eqref{one band O asymp} we find
\begin{equation*}
	\psi(x,t)  = q e^{\frac{i}{\eps} \left( q^2 t + \eps \omega(x,t) \right) } + \bigo{\sqrt{\eps} \log \eps }.
\end{equation*}

\subsection{Correction to the phase, evidence of a singular perturbation \label{sec: omega correction}} 
We take a moment here to remark on the phase correction $\omega(x,t)$ defined by \eqref{one band s asymp} that appears in 
\begin{equation*}
	\psi_{asy}(x,t) = q e^{\frac{i}{\eps} \lp q^2 t + \eps \omega(x,t) \rp },
\end{equation*}
the leading order asymptotic behavior of the solution for $(x,t) \in \mathcal{S}_1$. We can think of this phase in several ways. Thinking of the asymptotic solution as a slowly modulating plane wave, $\omega$ constitutes a slow correction to the singular phase. With respect to the fast evolution this is a constant phase which the Whitham theory cannot recover. Though the phase is slow it still affects $\bigo{1}$ change in the solution, which the Riemann-Hilbert machinery naturally recovers from the scattering data. This is a nice result, but can we understand this phase in a more direct way? Altering our perspective, we can treat $\omega$ as the phase correction in going from the geometric to physical optics approximation of the solution in the WKB theory. This leads to an interesting observation. Writing the solution of \eqref{nls} in the form $\psi(x,t) = A(x,t) e^{i S(x,t)/\eps}$ we get the following system of equations equivalent to focusing NLS
\begin{equation}\label{WKB system}
	\begin{split}
	\rho_t + \lp \rho S_x \rp_x &= 0, \\
	S_t + \frac{1}{2} S_x^2 - \rho &= \frac{\eps^2}{2} \left[ \frac{\rho_{xx}}{2\rho} - \lp \frac{\rho_x}{2\rho} \rp^2 \right],
	\end{split}
\end{equation}
 where $\rho(x,t) := A(x,t)^2$. The asymptotic solution $\psi_{asy}$ is a correction to the simple plane wave solution $\rho = q^2$ and $S = q^2 t$. Inserting a regular expansion ansatz: $\rho = q^2 + \eps \eta + \bigo{\eps^{2}}$ and $S = q^2 t + \eps \omega + \bigo{\eps^2}$, into \eqref{WKB system} and differentiating in $t$ to eliminate $\eta$ we are led to 
\begin{equation*}
 	\omega_{tt} + q^2 \omega_{xx} = \bigo{\eps}.
\end{equation*}
Assuming bounded derivatives in the expansion, $\omega$ should solve this rescaled Laplace equation, but, as defined by \eqref{one band s asymp}, $\omega$ does not satisfy the Laplace equation. To see this, we observe that due to the self-similar dependence of $\xi_1$ and $\xi_0$ on $(x,t)$ as defined by \eqref{one band xi_1} and \eqref{one band xi_0} we can write $\omega$ in the form $\omega(x,t) = F \lp \frac{x+L}{t} \rp + F \lp \frac{x-L}{t} \rp - \omega_0$ where
\begin{equation*}
	\begin{split}
	F(\zeta) &= -\frac{1}{\pi} \int_{-\infty}^{\lambda(\zeta)} \frac{\log(1+ |r_0(\lambda)|^2)}{\nu (\lambda)} \dd \lambda, 
	\qquad \lambda(\zeta) = -\frac{\zeta}{4}\left[1+\sqrt{1- 8q^2 \zeta^{-2}} \right] \\
	\omega_0 &=  -\frac{1}{\pi} \int_{-\infty}^\infty \frac{\log(1+ |r_0(\lambda)|^2)}{\nu (\lambda)} \dd \lambda.
	\end{split}
\end{equation*}
Letting $\zeta_\pm = (x \pm L)/t$ and seeking a solution of this general form, we get
\begin{equation*}
	 \omega_{tt}+ q^2 \omega_{xx} = 
	 \frac{1}{t^2} \left. \frac{d}{d\zeta} \left[ (\zeta^2+q^2) F' \right] \right|_{\zeta= \zeta_+} 
	 +  \frac{1}{t^2} \left. \frac{d}{d\zeta} \left[ (\zeta^2+q^2) F' \right] \right|_{\zeta= \zeta_-} = 0
\end{equation*}
Since $\zeta_+$ and $\zeta_-$ are independent this is only possible if $\frac{d}{d\zeta} \left[ (\zeta^2+q^2) F' \right] \equiv 0$
or $F(\zeta) = c_1 + c_2 \arctan\lp \zeta/q \rp$. As $\omega$ is clearly not of this form, this suggest that the initial assumption that the solution admits regular phase and amplitude expansions with bounded derivatives is false. From the results of \cite{DiFM05, Adrian06} we know that the discontinuities exhibit Gibbs phenomena in the small-time limit; these arbitrarily small wavelength oscillations are dispersed but not destroyed by the evolution. Further, though we know from the semiclassical Riemann-Hilbert analysis that the amplitude correction, $\eta$, is small, $\eta = \bigo{\sqrt{\eps} \log \eps}$, it is quite possible that $\eta$ exhibits rapid oscillations with significant derivatives. All of this suggest that the discontinuities generate interesting corrections to the phase and amplitude. In particular the phase correction, $\omega$, contributes at first order and is detected by the inverse-scattering machinery. We believe that both the phase and amplitude corrections for discontinuous initial data merit further study. Our focus here is to describe the regularization of the discontinuities by the semiclassical evolution, so we set these questions aside for now, but plan to address them in future work.


\section{Inverse analysis beyond the first breaking time \label{sec: two band} }
In the previous section we showed that for each $x \in [0,L)$ the inverse analysis was controlled before the breaking time $T_1(x)$, defined by \eqref{S1}, by the introduction of a $g$-function which a priori we assumed to contain a single band interval. The one band ansatz then implied the system of band and gap inequalities \eqref{one cut band/gap cond}. In this section we now consider the inverse problem for times $t > T_1(x)$. The breaking time $T_1(x)$ is characterized by the failure of the band condition \eqref{one cut band cond}, as implied by Lemma \ref{lem: one band zero level}; the lemma also guarantees that the gap condition \eqref{one cut gap cond} continues to be satisfied for times beyond the the first breaking time. The failure of the band condition for a generic choice of $x \in [0,L)$ as $t$ increases beyond the first breaking time $T_1(x)$ is depicted in Figure \ref{fig: one band levels beyond breaking}. 
\begin{figure}[thb]
	\begin{center}
		\includegraphics[width = .3\textwidth]{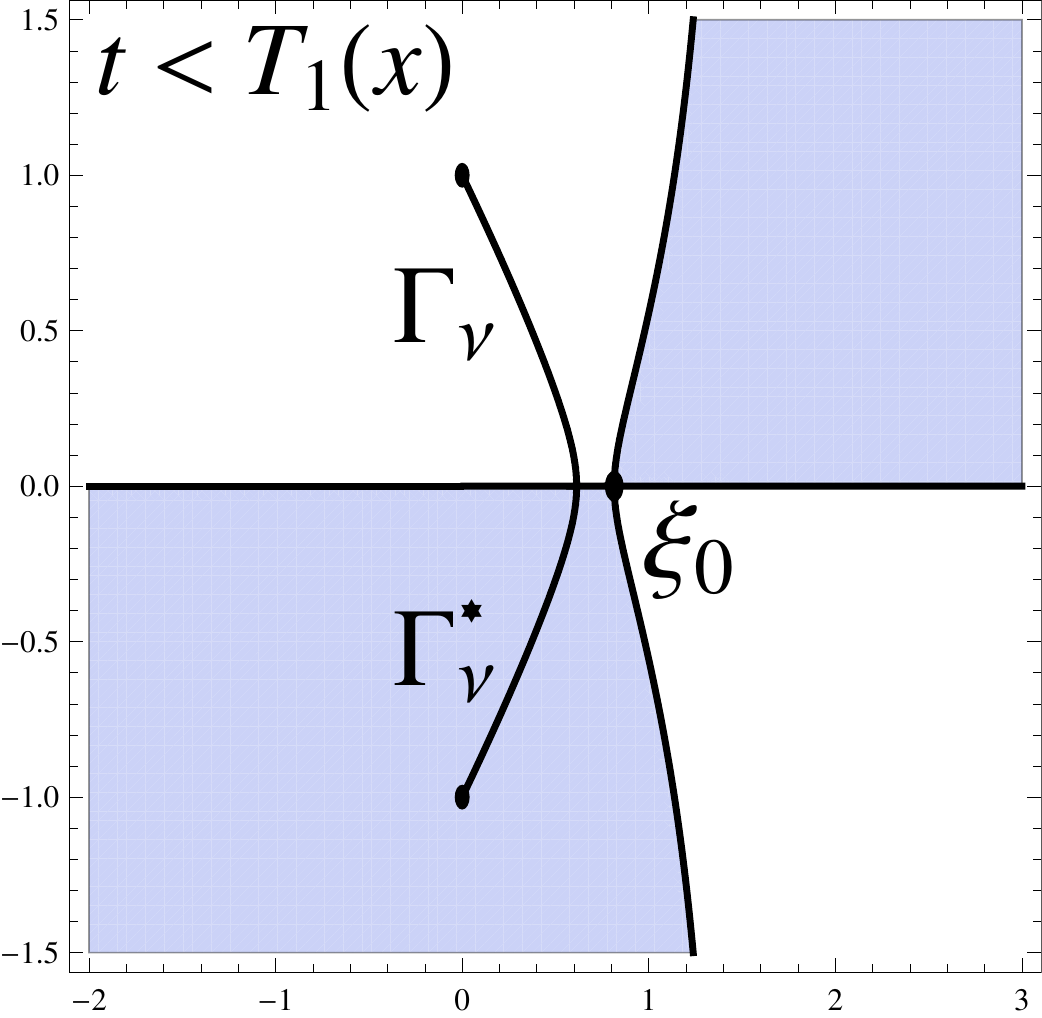}
		\quad
		\includegraphics[width = .3\textwidth]{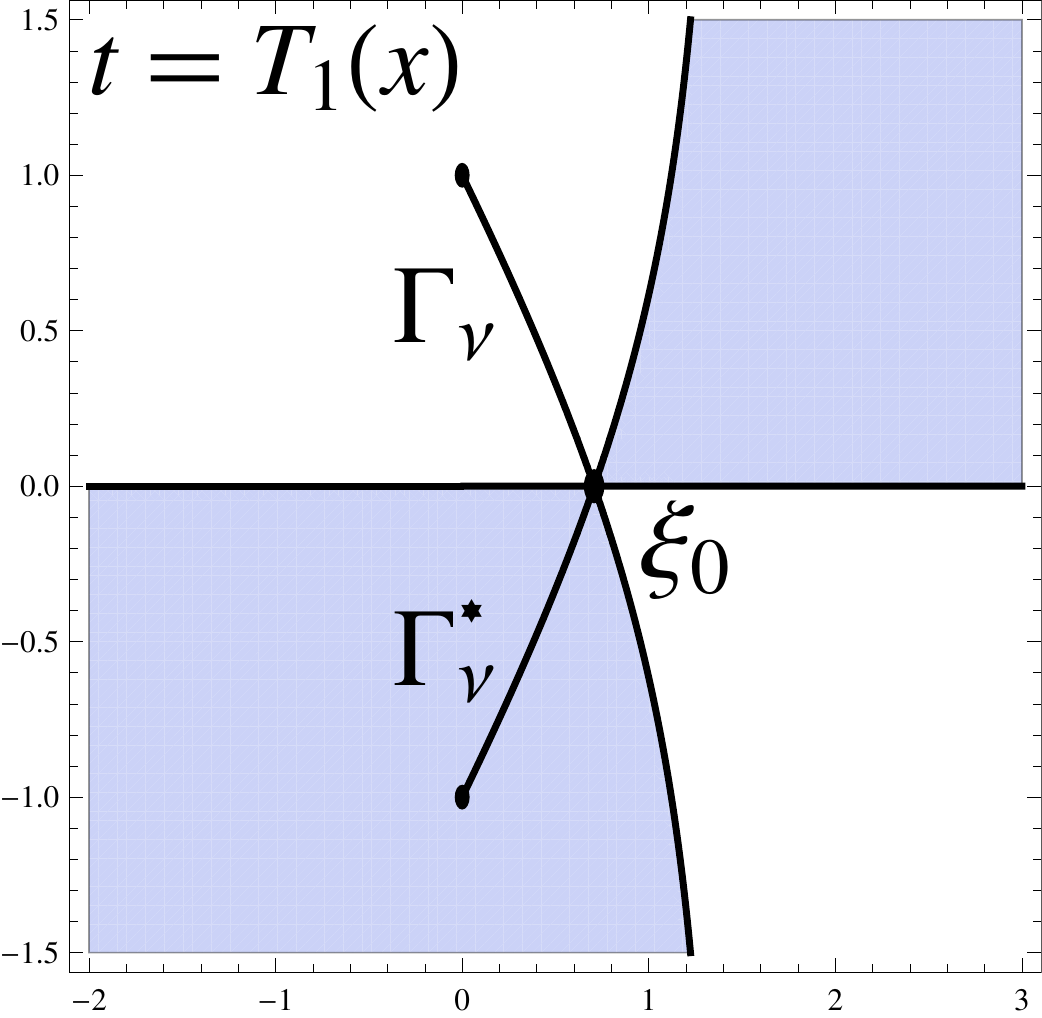}
		\quad
		\includegraphics[width = .3\textwidth]{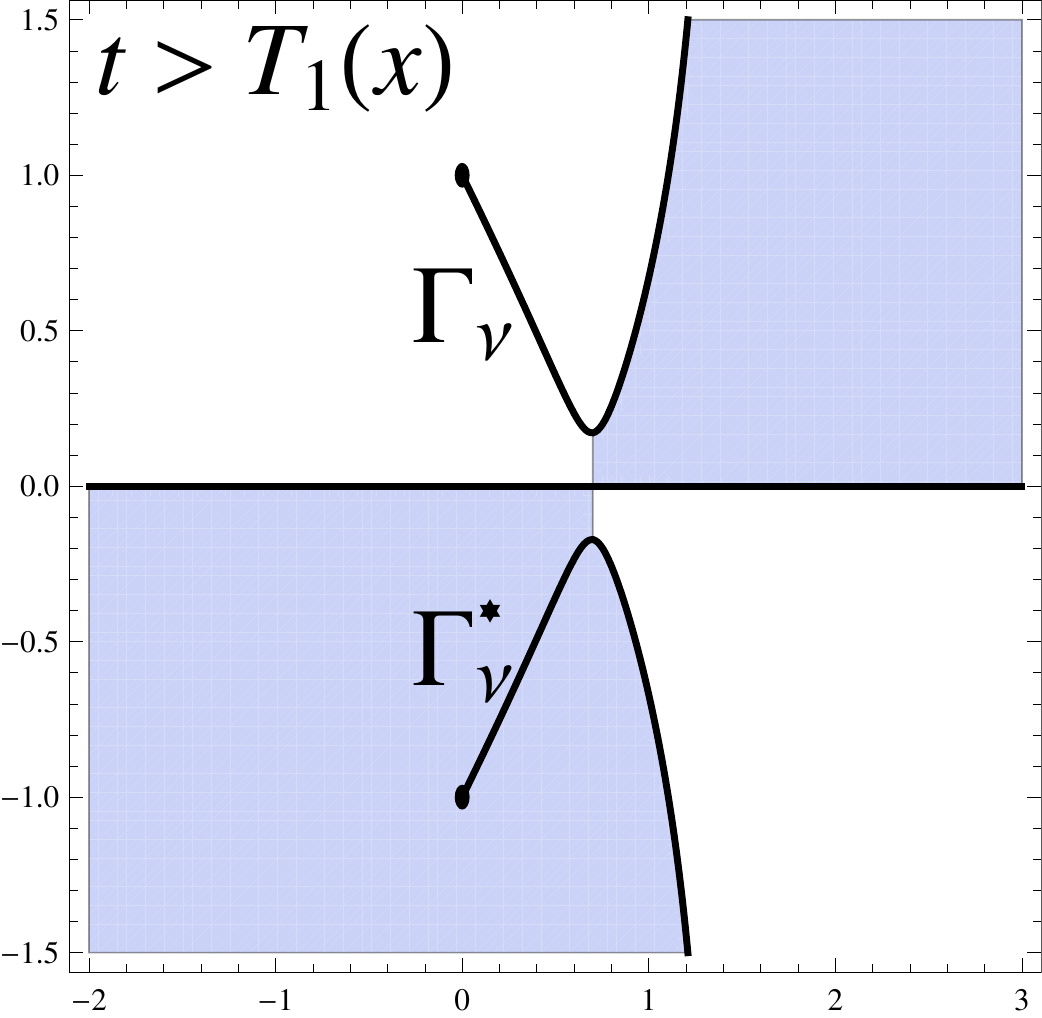}
		\caption{The zero level and sign structure of the modified phase $\imag \varphi_0$ introduced by the one band $g$-function as $t$ increases beyond the first breaking time, $T_1(x)$. Shaded regions represent $\imag \varphi_0>0$ and white regions $\imag \varphi_0 < 0$.   
		\label{fig: one band levels beyond breaking}
		}
	\end{center}
\end{figure}
In particular, note how the band condition \eqref{one cut band cond} fails for $t > T_1(x)$: a new gap opens across the real axis. To move the analysis of the inverse problem beyond the first breaking time we return to RHP \ref{rhp: M inside support} for the piecewise holomorphic unknown $M$. In the course of the analysis we will redefine the stationary points $\xi_k$, contours $\Gamma_k$, and introduce a new $g$-function which accounts for the newly opening gap.

\subsection{Introducing a new $g$-function, the genus one ansatz} 
As before we introduce our $g$-function by making the change of variables
\begin{equation}\label{two band N}
	N(z) = M(z) e^{-i g(z) \sig / \eps}.
\end{equation}
as always we demand that $g$ satisfy the generic decay and symmetry properties listed below \eqref{g trans}. Motivated by the appearance of a new gap as discussed above, we seek a new $g$-function supported on two symmetric bands $\gamma_b$ and $\gamma_b^*$ separated by a central gap contour $\gamma_g \cup \gamma_g^*$ such that $\gamma_b \cup \gamma_g = \Gamma_\nu$. We further suppose that the outer endpoints of the bands are fixed at $z = \pm iq$ but allow the inner endpoints, labelled $\alpha$ and $\alpha^*$, to evolve. Using this ansatz, we seek $g$ as the solution of the following scalar Riemann-Hilbert problem:
\begin{rhp}[for the genus one $g$-function]  \label{rhp: two band g}
Find a scalar function $g(z)$ such that
\begin{enumerate}[1.]
	\item $g$ is an analytic function of $z$ for $z \in \C \backslash (\Gamma_\nu \cup \Gamma_\nu^*)$.
	
	\item $g(z) = \bigo{1/z}$ as $z \rightarrow \infty$.
	
	\item $g$ takes continuous boundary values $g_+$ and $g_-$ on $\Gamma_\nu \cup \Gamma_\nu^*$ away from the endpoints which satisfy the jump relations:
	\begin{equation}\label{two band g jumps}
		\left\{ \begin{array} {l@{\quad}l}
			g_+(z) + g_-(z) = \theta_0 + \eta & z \in \gamma_b \cup \gamma_b^* \\
			g_+(z)  - g_-(z) = \Omega & z \in \gamma_g \cup \gamma_g^*
		\end{array}\right.
	\end{equation}
	where $\eta$ and $\Omega$ are real constants.
	
	\item As $z \rightarrow p$, where $p$ is a endpoint of the bands and gaps, $g(z)$ behaves as follows:
		\begin{itemize}
			\item $g(z) = \bigo{ (z - p)^{1/2}}$ for $p = iq$ or $p = -iq$. 
			\item $g(z) = \bigo{ (z-p)^{3/2}}$ for $p = \alpha$ or $p=\alpha^*$.
		\end{itemize}
\end{enumerate}
\end{rhp}
If we can construct such a function we define the modified phases as before
\begin{equation}\label{two band varphi def}
	\varphi_k(z) = \theta_k(z) - 2g(z),
\end{equation}
but now in terms of the solution $g$ of RHP \ref{rhp: two band g}. As discussed in Section \ref{sec: one band g}, in order for the $g$-function to lead to a semiclassically stable limiting problem, we require that it satisfy a system of band and gap relations of the form \eqref{band/gap cond}. It follows from transformation \eqref{two band N} and RHP \ref{rhp: two band g} that the necessary band and gap conditions take the form:
 \begin{subequations}\label{two band band/gap cond}
 \begin{align}
	\label{two cut band cond}  &\textsc{Bands:} \quad \ \ \imag \lp  { \varphi_0 }_- \rp = 0 \text{ for } z \in \gamma_b,  \\
	\label{two cut gap cond 1} &\textsc{Gap 1:}  \quad \ \ \imag \lp \varphi_0 \rp  > 0 \text{ for } z \in \gamma_g, \\
	\label{two cut gap cond 2} &\textsc{Gap 2:}  \quad \ \ \imag \lp \varphi_1 \rp > 0 \text{ for } z \in \Gamma_1.
\end{align}
\end{subequations}

\begin{rem} The above conditions imply, using the reflection symmetry of the $\varphi_k$, the corresponding relations in the lower half-plane. 
\end{rem}

Conditions \eqref{two cut band cond} and \eqref{two cut gap cond 1} ensure that the $g$-function does not introduce large jumps along $\Gamma_\nu \cup \Gamma_\nu^*$ while the last condition \eqref{two cut gap cond 2} ensures that the jumps along the contours $\Gamma_1 \cup \Gamma_1^*$ introduced by the pole removing transformation \eqref{M inside} remain small after introducing $g$.

\subsubsection{Constructing the genus one $g$-function} 
We begin by seeking a $g$-function of the form 
\begin{equation}\label{two band rho def}
	g(z) = \int_{\infty}^z \rho(\lambda) \dd \lambda
\end{equation}
for an unknown density $\rho$ where the path of integration is any simple contour which does not intersect $\Gamma_\nu \cup \Gamma_\nu^*$. The density $\rho$ must necessarily be symmetric: $\rho(\lambda^*)^* = \rho(\lambda)$; and solve the problem resulting from formally differentiating RHP \ref{rhp: two band g} .
\begin{rhp}[for the genus two density $\rho(z)$:]\label{rhp: two band density}
Find a scalar function $\rho$ such that:
\begin{enumerate}[1.]
	\item $\rho$ is an analytic function of $z$ for $z \in \C \backslash (\gamma_b \cup \gamma_b^*)$.
	\item $\rho(z) = \bigo{1/z^2}$ as $z \rightarrow \infty$.
	\item On $\gamma_b \cup \gamma_b^*$$\rho$ takes continuous boundary values $\rho_+$ and $\rho_-$ which satisfy the jump condition
	\begin{equation}\label{two band rho jump}
		\rho_+(z) + \rho_-(z) = \theta'_0(z), \quad z \in \gamma_b \cup \gamma_b^*
	\end{equation}
	\item As $z \rightarrow p$, where $p$ is an endpoint of a band or gap interval,  $\rho(z)$ behaves as follows:
	\begin{equation}\label{two band rho endpoints}
		\begin{array}{@{\bullet\ }l@{\ \ \text{for }}l}
			\rho(z) = \bigo{ (z - p)^{-1/2}} & p = \pm iq, \\
			\rho(z) = \bigo{ (z-p)^{1/2}} & p = \alpha,\ \alpha^*.
		\end{array}
	\end{equation}
\end{enumerate}
\end{rhp}

The normalization condition and path restriction guarantee that $g$ is a well defined function and that $g(z) = \bigo{1/z}$ as $z \to \infty$. The advantage of this representation of $g$ is that the unknown constants $\eta$ and $\Omega$ are expressible as particular integrals of the density $\rho$. Suppose $\rho$ is a solution of RHP \ref{rhp: two band density}, then for any $z \in \gamma_g$ and allowable paths for evaluating $g_+(z)$ and $g_-(z)$ it follows by contour deformation that
\begin{equation}\label{two band jump gaps}
	g_+(z) - g_-(z) = \int_{C_b} \rho(\lambda) \dd \lambda = \Omega, \qquad z \in \gamma_g \cup \gamma_g^*, 
\end{equation}
where $C_b$ is any positively oriented simple closed contour in $\C^+$ which encloses $\gamma_b$. Alternatively, the allowable paths could be deformed to the conjugate loop, $C_b^*$ (enclosing $\gamma_b^*$ with negative orientation), averaging gives the more symmetric formula:
\begin{equation}\label{omega def}
	\Omega = \frac{1}{2}\lp  \int_{C_b}\rho(\lambda) \dd \lambda + \int_{C_b^*}\rho(\lambda) \dd \lambda \rp,
\end{equation}
which implies, using the symmetry of $\rho$, that $\Omega$ is a real constant. 
\begin{rem}
Here and elsewhere we make use of the following elementary fact: if $f(z)$ satisfies the symmetry $f(\lambda^*)^* = f(\lambda)$ and $C$ is an oriented contour such that $C^* = C$ then
\begin{align*}
\lp  \int_C f(\lambda) \dd \lambda \rp^* =  \int_C f(\lambda)^* \dd \lambda^* =  \int_C^* f(\lambda^*)^* \dd \lambda  =  \int_C f(\lambda) \dd \lambda.
\end{align*}
Similarly, if instead $C^* = -C$ then:
\begin{align*}
\lp  \int_C f(\lambda) \dd \lambda \rp^* = - \int_C f(\lambda) \dd \lambda . 
\end{align*}
\end{rem}

Next we consider the jumps of $g$ along the bands and the corresponding constant $\eta$. We have, 
\begin{align}\label{two band jump bands}
	g_+(z) + g_-(z) = \theta_0(z) + \eta, \qquad z \in \gamma_b
\end{align}
where
\begin{align}\label{two band eta}	
	\eta = - \theta_0(iq) + 2 \int_{\infty}^{iq} \rho(\lambda) \dd \lambda.
\end{align}
Repeating the computation for $z \in \gamma_b^*$, we have $g_+(z) + g_-(z) = \theta_0(z) + \hat \eta$, where,
\begin{align*}
	\hat \eta = - \theta_0(-iq) + 2 \int_\infty^{-iq} \rho(\lambda) \dd \lambda.
\end{align*}

Observing symmetries, $\eta$ and $\hat \eta$ are complex conjugate: $\hat \eta = \eta^*$. However, a necessary component of RHP \ref{rhp: two band g} requires that $\eta$ be a real constant. By contour deformation, we can express the difference as a single integral over the gap:
\begin{equation*}
	\eta^*- \eta = \int_{\gamma_g}  \theta_0'(\lambda) - 2 \rho(\lambda)  \dd \lambda
\end{equation*}	
For generic $\alpha$ this integral will be nonzero. The vanishing of this integral constitutes a single real condition---symmetry implies the integral is imaginary---on the moving endpoint $\alpha$. This condition has another important implication: in terms of the modified phase $\varphi_0$, the vanishing of the above condition becomes
\begin{equation}\label{endpoint gap cond}
	\int_{\gamma_g}  \theta_0'(\lambda) - 2 \rho(\lambda)  \dd \lambda = \int_{\gamma_g} \dd \varphi_0 = 0.
\end{equation}
This implies that the endpoints $\alpha$ and $\alpha^*$ lie on the same level set of $\imag \varphi_0$, a necessity in order to satisfy the band condition \eqref{two cut band cond}.

We now turn to directly constructing the density $\rho$. Motivated by the endpoint behavior and the additive jump condition we introduce
\begin{equation*}
	S(z) = \lp \frac{ z - \alpha } { z - iq } \rp^{1/2} \lp \frac{ z - \alpha^* } { z + iq } \rp^{1/2}
\end{equation*}
understood to be branched along the bands $\gamma_b \cup \gamma_b^*$ and normalized to approach unity as $z \rightarrow \infty$. Using $S(z)$ and the Plemelj-Sokhotsky formula, it immediately follows that: 
\begin{align*}\label{two band rho}
	\rho(z) = \frac{S(z)}{2 \pi i} \int_{\gamma_b^{} \cup \gamma_b^*} \frac{ \theta_0'(\lambda)}{S_+(\lambda) } \frac{ \dd \lambda }{ \lambda - z},
\end{align*} 
which can be evaluated by residues yielding the explicit formula
\begin{equation}\label{two band rho explicit}
	\rho(z) = \frac{1}{2} \theta_0'(z) - S(z) \left[ t( 2z+ \alpha + \alpha^* ) + (x-L) \right].
\end{equation} 
Clearly, $\rho$ satisfies the analyticity and endpoint conditions of RHP \ref{rhp: two band density} for any generic choice of $\alpha$. However, for generic $\alpha$, $\rho(z) = \bigo{1/z}$ as $z \to \infty$. The quadratic decay at infinity is necessary for $g$ to be analytic at infinity and is essential to deriving formulas \eqref{omega def} and \eqref{two band eta} for the jump constants. In order for $\rho = \bigo{1/z^2}$ for  large $z$, $\alpha$ must satisfy the following moment condition: 
 \begin{align}\label{two band moment cond}
 	-\frac{1}{2\pi i} \int\limits_{\gamma_b^{} \cup \gamma_b^*} \frac{ \theta_0'(\lambda)}{S_+(\lambda) } \dd \lambda  =
	t \left[ q^2 + \frac{1}{4} \lp 3\alpha^2 + 2\alpha \alpha^* + 3{\alpha^*}^2 \right) \right] +  \frac{(x-L)}{2} (\alpha + \alpha^*) = 0,
\end{align}
where the first equality follows from evaluating the integral by residues. 

\subsubsection{Determining the motion of the $\alpha(x,t)$: self-similar solution of the Whitham system} The moment and gap conditions, \eqref{two band moment cond} and \eqref{endpoint gap cond}  respectively, constitute two real conditions on the moving endpoints $\alpha$ and $\alpha^*$ for each $(x,t)$ in the genus one region. For each $t > 0$ these conditions are equivalent to the intersection of the zero level sets of the functions
\begin{subequations}\label{two band endpoint functions}
\begin{align}
		F_M(\alpha, x, t) &:= \frac{t}{4} \lp 3\alpha^2 + 2\alpha \alpha^* + 3{\alpha^*}^2 + 4 q^2 \rp + \frac{(x-L)}{2} (\alpha + \alpha^*), \label{two band moment function} \\
		F_G(\alpha, x, t) &:= \int_{\alpha^*}^\alpha S(\lambda) \left[ t(2\lambda +\alpha + \alpha^*) + (x - L) \right] \dd \lambda.  \label{two band gap function}
\end{align}
\end{subequations}
Observing that these functions are linear---as opposed to affine---functions of $x-L$ and $t$, it follows that for all $t>0$ their zero level sets depend only on the self-similar variable
\begin{equation}\label{mu def}
 	\mu:= -\frac{x-L}{2t}
\end{equation}
 which encodes all of the explicit $(x,t)$ dependence of each function. Geometrically, $\mu$ is an angular variable parameterizing characteristic lines in $(x,t)$-space passing through $x=L$. The exterior region, $\mathcal{S}_0$, with no  $g$-function,  and the genus zero region, $\mathcal{S}_1$, are naturally described in terms of $\mu$ as $\mu < 0$ and $\mu > \sqrt{2}q$ respectively. The remainder of this section is devoted to the proof of the following proposition, which states roughly that for each $\mu$ in the remaining interval, $\mu \in (0, \sqrt{2}q)$, the moment and gap functions \eqref{two band endpoint functions} uniquely determine $\alpha(x,t) := \mathcal{A}(\mu)$, such that the limiting values $\alpha(0)$ and $\alpha( \sqrt{2} q )$ ``match" the exterior and genus one cases.
\begin{prop}\label{prop: two band endpoints} 
There exist a function $\mathcal{A}: (0, \sqrt{2}q) \rightarrow \C^+$ such that by taking $\alpha = \mathcal{A}(\mu)$, where $\mu$ is the self-similar variable \eqref{mu def}, the moment and gap equations $F_M( \mathcal{A}(\mu'), x', t')=0$ and $F_G(\mathcal{A}(\mu'), x', t') = 0$, with $F_M$ and $F_G$ given by \eqref{two band endpoint functions}, are satisfied for each (x',t') such that $t'>0$ and $\mu' \in (0,\sqrt{2} q)$. Additionally, $\alpha = \mathcal{A}(\mu)$ assumes the following limiting values at the boundaries of its definition. 
\begin{equation*}
	\lim_{\mu \rightarrow 0} \mathcal{A} =  iq \quad \text{and} \quad \lim_{\mu \rightarrow \sqrt{2}q} \mathcal{A} = \frac{q}{\sqrt{2}}
\end{equation*}   
\end{prop}

\begin{proof}
Both $F_M$ and $F_G$ are analytic functions of $\alpha$ and $\alpha^*$ so the Jacobian necessary to invoke the implicit function theorem to find $\alpha$ and $\alpha^*$ as functions of $x$ and $t$ is
\begin{equation}
	\mathcal{J}(\alpha, \alpha^*) := \det \left[ \begin{array}{cc} 
		\pd{F_M}{\alpha} & \pd{F_M}{\alpha^*}_\Bspace \\
	 	\pd{F_G}{\alpha} & \pd{F_G}{\alpha^*}
	\end{array} \right].
\end{equation}
A straightforward computation using \eqref{two band endpoint functions} shows that 
\begin{equation}\label{gap derivatives}
\begin{split}
	\pd{F_G}{\alpha} &= - \pd{F_M}{\alpha} \int_{\alpha^*}^\alpha \frac{\lambda - \alpha^*}{\mathcal{R}(\lambda)} \dd \lambda \\
	\pd{F_G}{\alpha^*} &= - \pd{F_M}{\alpha^*} \int_{\alpha^*}^\alpha \frac{\lambda - \alpha}{\mathcal{R}(\lambda)} \dd \lambda 
\end{split}
\end{equation}
where 
\begin{equation*}
\mathcal{R}(\lambda) := \sqrt{ (\lambda - iq) (\lambda - \alpha) (\lambda + iq) ( \lambda - \alpha^*)},
\end{equation*}
is cut along the band intervals $\gamma_b$ and $\gamma_b^*$ and normalized such that $\mathcal{R} \sim \lambda^2$ for $\lambda \rightarrow \infty$. From this it follows directly that
\begin{equation}\label{alpha Jacobian}
	\mathcal{J}(\alpha, \alpha^*) = (\alpha - \alpha^*) \pd{F_M}{\alpha} \pd{F_M}{\alpha^*} \int_{\alpha^*}^\alpha \frac{ \dd \lambda}{\mathcal{R}(\lambda)} = 2i \frac{(\alpha - \alpha^*)}{| \alpha + iq |} \pd{F_M}{\alpha} \pd{F_M}{\alpha^*} K(m),
\end{equation} 
where $K(m)$ is the complete elliptic integral of the first kind with parameter
\begin{equation}\label{elliptic parameter}
	m = 1 - \frac{|\alpha - iq|^2}{|\alpha+iq|^2}.
\end{equation}
Clearly, $m \in [0,1]$ for $\alpha \in \overline{\C^+}$, vanishes only for $\alpha \in \R$ and achieves unity only for $\alpha = iq$; thus the Jacobian is finite and nonzero for each $\alpha \in \C^+ \backslash \{ iq \}$. Just like the Jacobian,  the gap equation, $F_G$, can be expressed in terms of complete elliptic integrals (cf. chpater 22 of \cite{WW62}). Following this procedure, and employing a Landen transformation \cite{Byrd54} to simplify the form, the gap and moment equations $F_M = 0$ and $F_G = 0$ are equivalent to the pair of equations 
\begin{equation}\label{alpha elliptic equations}
	\begin{split}
		\frac{1}{4} \lp 3\alpha^2 + 2\alpha\alpha^* + 3{\alpha^*}^2 \rp + q^2 - \mu (\alpha + \alpha^*) &= 0, \\
		|\alpha - iq|^2 K(m) -\re\left[(\alpha-iq)(\alpha+iq)\right] E(m) &= 0.
	\end{split}
\end{equation}	
where $K$ and $E$ are the complete elliptic integrals of the first and second kind respectively. These equations together with \eqref{elliptic parameter} implicitly define $\alpha$ as a function of $\mu$ for each $\alpha \in \C^+ \backslash \{ i q\}$. Using these equations we can express $\alpha$ in the convenient form 
\begin{equation}\label{alpha implicit}
	\alpha = q \left( \sqrt{4A(m) - (1+mA(m))^2} + imA(m)  \right),
\end{equation}
where we have defined
\begin{equation*}
	A(m) = \frac{(2-m)E(m) - 2(1-m)K(m)}{m^2E(m)}.
\end{equation*}
 For each $m \in [0,1]$, \eqref{alpha implicit} is well defined and has the following limiting behavior 
 \begin{align*}
 	\alpha &= \frac{q}{\sqrt{2}} + \frac{3iq}{8}m + \bigo{m^2}  &m \rightarrow 0^+, \\
	\alpha &= iq - 2\sqrt{1-m} +  \bigo{(1-m)\log(1-m)} \quad &m \rightarrow 1^-.
\end{align*}
At these limiting values, we have $F_M(q/\sqrt{2},x,t) =  2tq ( q- \mu/ \sqrt{2})$ so $\mu = \sqrt{2}q$ at $m=0$, and $F_G(iq, x, t) = -4iqt \mu$, so $\mu = 0$ at $m = 1$. This completes the proof.
\end{proof}

Before moving to other considerations, we will show that the endpoints $\alpha$ and $iq$ of the $g$-function describe a special solution for the Riemann invariants of a slowly modulating one phase wave solution of NLS. The level curves $F_M = 0$ and $F_G=0$ define $\alpha$ as a smooth function of both $x$ and $t$ which allows one to define the characteristic speed     
\begin{equation}\label{speed def}
	c(\alpha, iq) := -\pd{\alpha}{t} \Bigg{/} \pd{\alpha}{x}.
\end{equation}
Implicitly differentiating \eqref{gap derivatives} it follows that 
\begin{equation*}
	c(\alpha, iq) = - \frac{ \pd{F_M}{t} \int_{\alpha^*}^\alpha \frac{\lambda - \alpha}{\mathcal{R}(\lambda)} d \lambda + \pd{F_G}{t} }{ \pd{F_M}{x}\int_{\alpha^*}^\alpha \frac{\lambda - \alpha}{\mathcal{R}(\lambda)} d \lambda + \pd{F_G}{x} }.
\end{equation*}
The speed $c$ depends on $x$ and $t$ only through the endpoints $iq$ and $\alpha$. Using the same procedure used to simplify the integrals in the gap equation in the above proof, we can express the characteristic speed as
\begin{equation}\label{Riemann invariant speed}
	c(\alpha, iq) = -\frac{1}{2}(\alpha + \alpha^*) - \frac{(\alpha - \alpha^*)(\alpha - iq)K(m) }{(\alpha -iq) K(m) + (iq- \alpha^*)E(m)},
\end{equation}
where $m$ is the elliptic parameter defined by \eqref{elliptic parameter}.	 Using \eqref{speed def} and \eqref{Riemann invariant speed} we have
\begin{equation}\label{one phase Whitham evolution}
	\pd{\alpha}{t} + c(\alpha, iq) \pd{\alpha}{x} = 0,
\end{equation}
which is precisely the Whitham equation for the evolution of one Riemann invariant of a slowly modulating one phase wave solution of NLS \cite{AI94}. The generic one phase modulation is described by two conjugate pairs of Riemann invariants labelled $\lambda_i$; each satisfying an evolution equation $\partial_t \lambda_i + c_i(\vec{\lambda}) \partial_x \lambda_i = 0,\ i=1,\ldots,4$. Motivated by the self-similar solution for $\alpha$ described by Prop. \ref{prop: two band endpoints}, if we seek solutions of the Whitham system in the form $\lambda_i(x,t) := \lambda_i(\mu)$, with $\mu$ as defined by \eqref{mu def}, the evolution becomes $\partial_\mu \lambda_i \lp c_i(\vec{\lambda}) - \mu \rp = 0$.  So each invariant $\lambda_i$ is either constant, just as the endpoint $z = iq$ in our construction, or they satisfy $c_i(\vec{\lambda}) = \mu$. Self-similar solutions of this type were studied in \cite{Kam97}; there the author shows that setting $c_i(\vec{\lambda}) = \mu$ with $c_i$ given by \eqref{Riemann invariant speed} one arrives at the implicit system of equations \eqref{alpha elliptic equations} which describe the evolution of the endpoint $z = \alpha$. Thus the evolution $\lambda_1 = iq$, $\lambda_2 = \alpha = \mathcal{A}(\mu)$, is preciesely a self-similar solution of the one-phase Whitham system.	
	
\subsubsection{Topological structure of the zero level sets}
With $\alpha$ selected to satisfy the endpoint equations \eqref{two band endpoint functions}, the construction of the $g$-function will be complete if we can select the band and gap contours to satisfy the conditions set out in \eqref{two band band/gap cond}. Those $(x,t)$ for which these conditions can be satisfied comprise the region of validity of the genus one $g$-function ansatz. Satisfying the band and gap conditions amounts to understanding the topology of the zero level sets of $\imag \varphi_0$ and $\imag \varphi_1$. Our procedure to determine the structure of the zero level sets mimics the proof of Lemma \ref{lem: one band zero level}. Let $L_k(x,t) := \{ z \, : \, \imag \varphi_k(z) = 0 \}$, the essential facts needed to determine the topology of these sets are contained in the following observations
\begin{itemize}
	\item The conjugation symmetry $\varphi_k^*(z) = \varphi_k(z^*)$ guarantees that $\R \subset L_k(x,t)$ and allows us to consider only $z \in \C^+$.
	\item $L_k(x,t)$ cannot contain a closed bounded loop if $\imag \varphi_k$ is harmonic inside the loop as this would imply $\varphi_k$ is identically zero inside the loop.
	\item Any branch of $L_k(x,t)$ in $\C^+$ can terminate only at a critical point of $\varphi_k$: the branch points $iq$ and $\alpha$, $\infty$, or any other zero of the corresponding density in $\overline{\C^+}$. 
\end{itemize}	
	
We begin with $L_0(x,t)$, using \eqref{two band varphi def}, \eqref{two band rho def}, and \eqref{two band rho explicit}, we have the convenient representation
\begin{align*}
	\varphi_0 &=  \varphi_0(p) + \int_{p}^z 4t S(\lambda) (\lambda - \xi_0) \dd \lambda  \\
\end{align*}
where $p$ is any fixed point in the plane and the path of integration is still understood to avoid $\Gamma_\nu \cup \Gamma_\nu^*$. Additionally, we have defined the $\mu$-dependent real value
\begin{equation}\label{two band xi_0}
	\xi_0 = \mu - \frac{1}{2}(\alpha+\alpha^*). 
\end{equation}
Simple estimates using \eqref{alpha elliptic equations} and \eqref{alpha implicit} imply that $\xi_0$ is always positive.  Two facts are immediately obvious from this representation. First, for $t>0$, $L_0$ depends only on the self-similar parameter $\mu$, and second, $z = \xi_0$ is the only critical point of $\varphi_0$ not listed above. Using \eqref{two band jump gaps} and \eqref{two band jump bands}, $\varphi_0$ has the following local expansions near each critical point, where $c$ is a nonzero constant which depends on the expansion point:
\begin{itemize}
	\item As $z \rightarrow iq, \ \varphi_0(z) \sim -\eta + c(z-iq)^{1/2}$, so one branch of $L_0$ terminates at $iq$.
	\item As $z \rightarrow \alpha, \ (\varphi_0)_\pm (z) \sim -\eta \mp \Omega + c(z-\alpha)^{3/2}$, so three branches of $L_0$ emerge from $\alpha$ separated by angles of $2\pi/3$.
	\item As $z \rightarrow \xi_0, \ \varphi_0(z) \sim \varphi_0(\xi_0) + c(z-\xi_0)^2$, so one branch of $L_0$ in $\C^+$ terminates at $\xi_0$.
	\item As $z \rightarrow \infty, \ \varphi_0(z) \sim -\theta_0(z)$, so in addition to the real line, one branch of $L_0$ leaves $\C^+$ by approaching $\infty$ along a trajectory asymptotically approaching $\re z  = \mu$.
\end{itemize}
Since the constant terms in each expansion are real, each critical point lies in $L_0$, and since no trajectory leaving $\alpha$ may form a homoclinic orbit the only possibility is that the three trajectories leaving $\alpha$ terminate at $iq$, $\infty$, and $\xi_0$ respectively. It only remains to determine the topology of those connections. Consider the degenerate case $\mu = \mu_c = \sqrt{2}q$; here the branch point $\alpha$ is real and the two bands degenerately form a single continuous path connecting $\pm iq$. Recalling that $\mu_c$ parameterizes the first breaking time $T_1(x)$ separating the genus zero and genus one regions, it is not surprising that, comparing \eqref{one band g'} and  
\begin{figure}[htb]
\begin{center}
		\includegraphics[width=.4\textwidth]{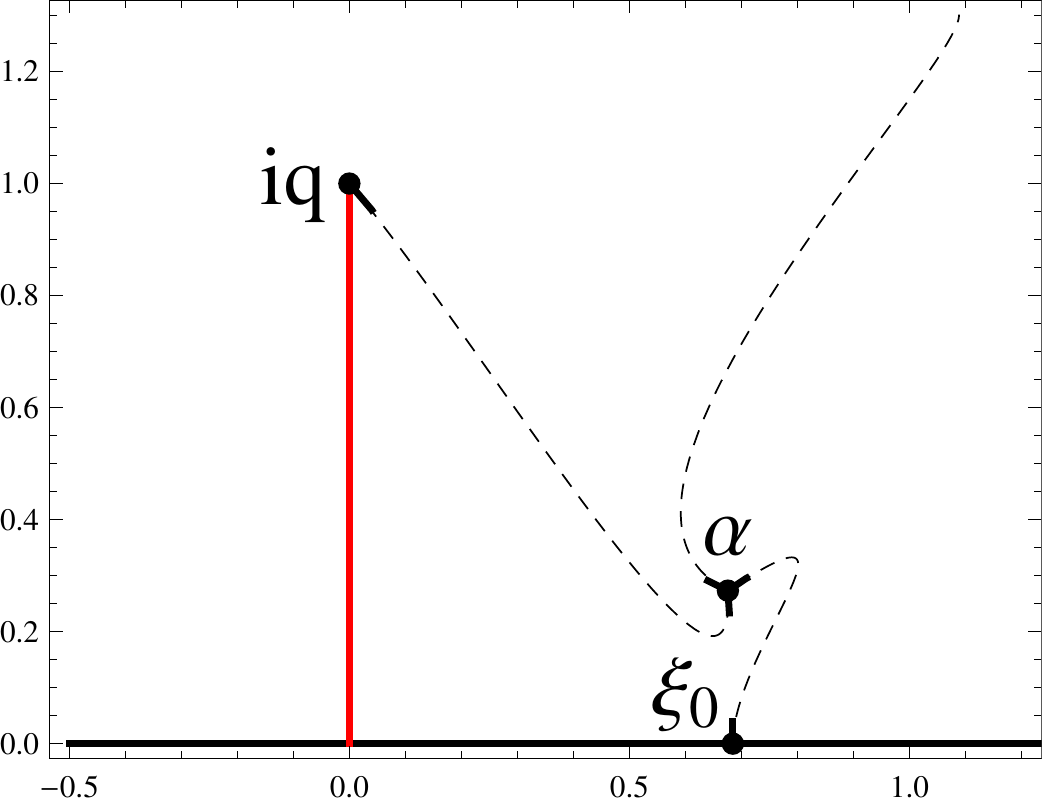} 
		\qquad
		\includegraphics[width=.4\textwidth]{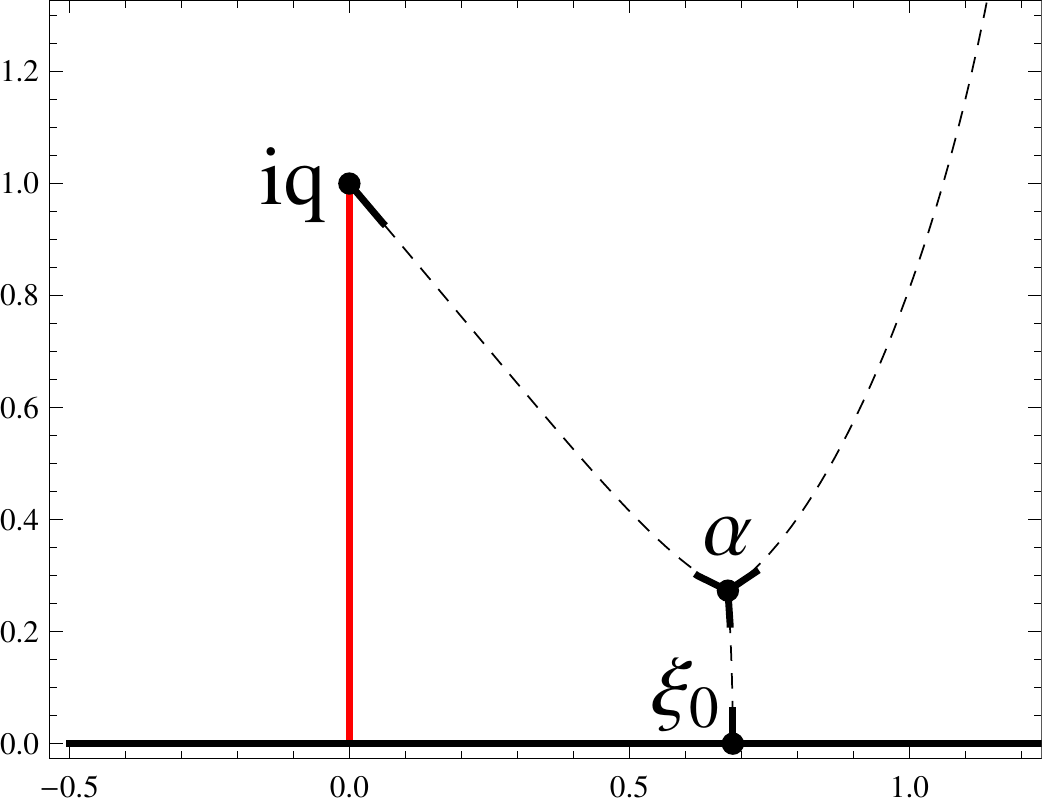}
\caption{The connection problem for the level set $\imag \varphi_0$. The left figure shows an allowable "twisted" configuration of the branches (dashed lines) connecting each of the critical points in $\C^+$, while the right figure shows the numerical computed branches for the same choice of $\alpha$. Each possible configuration is topologically equivalent and separates $\C^+$ into two connected components on which $\imag \varphi_0$ is single signed.
\label{fig: two band connection problem}
}    
\end{center}
\end{figure}	
\eqref{two band rho explicit}, the degenerate genus one phase $\varphi_0$ is identically equal to its genus zero counterpart along the caustic $\mu_c$. Just beyond the caustic, that is for $\mu = \mu_c - \hbar, \ \hbar \ll 1$, we can view, for any fixed $z$ bounded away from $\alpha(\mu_c)$, the genus one phase $\varphi_0(z)$ as a small, continuous, perturbation of the genus zero phase. In the genus zero case, we showed in the course of proving Lemma \ref{lem: one band zero level} that $\imag \varphi_0(i y)$ is bounded away from zero for $y$ in any closed subset of $(0,q)$. This is robust to small perturbations, and it follows that no trajectory of $L_0$ passes through the open imaginary interval between the real axis and $iq$ for $\mu$ sufficiently close to $\mu_c$. As the trajectories leaving $\alpha$ cannot intersect each other, this completely determines the topology of $L_0(x, t)$ for $\mu$ near $\mu_c$: the trajectories leaving $\alpha$ connect, in counterclockwise order, to $iq$, $\infty$, and $\xi_0$, see Figure \ref{fig: two band connection problem}. As $\mu$ decreases to zero, $L_0(x,t)$ deforms continuously and thus the topology of the level set is preserved. The level set $L_0(t)$ separates $\C^+$ into two connected regions, in each region the sign of $\imag \varphi_0$ is determined by continuation from its limiting behavior for large $z$. The band condition \eqref{two cut band cond} and the first gap condition \eqref{two cut gap cond 1} are satisfied by defining $\gamma_b$ to be the unique branch of the level set $\imag\varphi_0 = 0$ connecting $\alpha$ to $iq$ and taking $\gamma_g$ to be any simple contour connecting $\alpha$ to $\xi_0$ such that it lies everywhere to the right of $L_0(t)$, that is, in the region of $\C^+$  where $\imag \varphi_0 >0$. For convenience we choose $\gamma_g$ such that it approaches $\alpha$ and $\xi_0$ along steepest descent paths of $\varphi_0$, such a choice is depicted in Figure \ref{fig: two band varphi0}. 
\begin{figure}[htb]
\begin{center}
	\includegraphics[width = .4\textwidth]{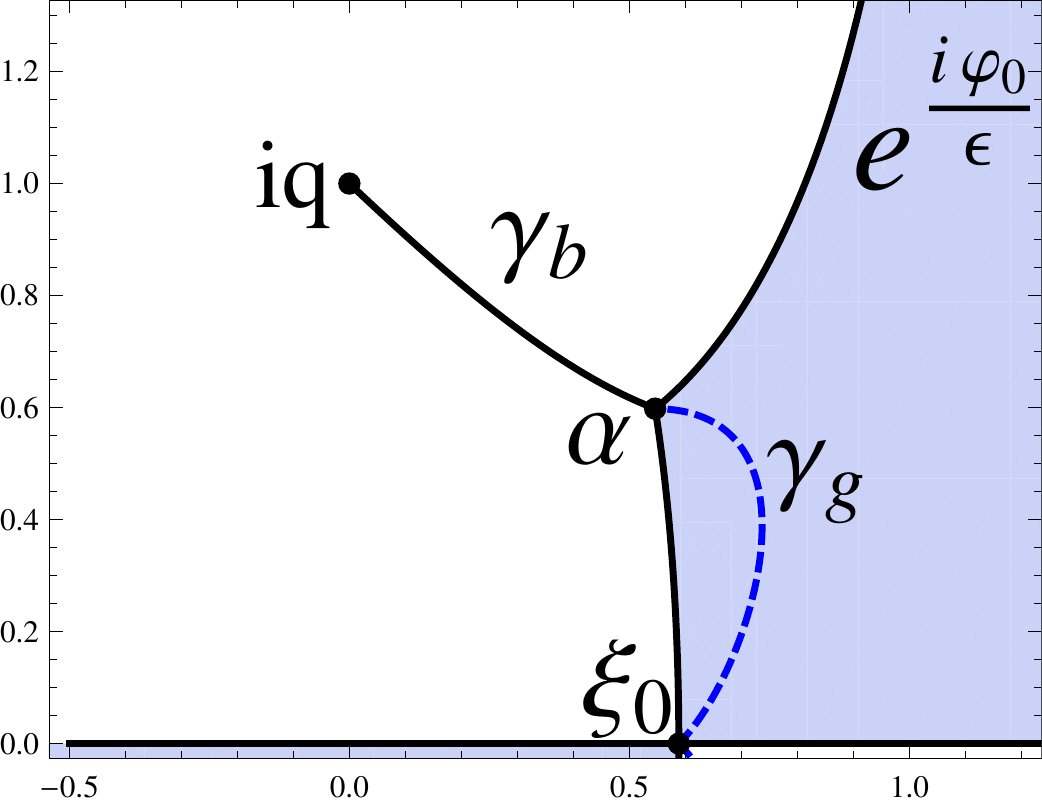}
\caption{Numerical computation of $\imag \varphi_0$ in $\C^+$ for a generic choice of $\mu \in (0,\sqrt{2}q)$. Solid lines represent the zero level set while shaded and unshaded region where $\imag \varphi_0$ is positive and negative respectively. The dashed line represents an arbitrary choice of the gap contour $\gamma_g$ which we must choose to lie everywhere in the region $\imag \varphi_0 > 0$. 
\label{fig: two band varphi0}
}
\end{center}
\end{figure}
	
It remains to show that condition \eqref{two cut gap cond 2} can be satisfied which turns our attention to $\varphi_1$ and its associated imaginary zero level set $L_1(x,t)$. By definition, $\varphi_1 = \varphi_0 + 4L\nu(z)$ but it also has the convenient integral representation
\begin{equation}\label{phase 1 integral rep}
	\begin{split}
		\varphi_1(z) &= \varphi_1(p) + \int_p^z \rho_1(\lambda) \dd{\lambda} \\
		\rho_1(\lambda) &= 4t S(\lambda)(\lambda - \xi_0) + 4L\frac{\lambda}{\nu(\lambda)}
	\end{split}
\end{equation}
for any fixed base point $p$ in $\C$; as always the path of integration is understood to avoid $\Gamma_\nu \cup \Gamma_{\nu}^*$ along which $\nu$ is branched. The two representations lead to the following simple observations:
\begin{itemize}
	\item For each $z \in \C^+\backslash \{ iq \}$ such that $\imag (\varphi_0) \geq 0,\ \imag(\varphi_1) \geq 4L\imag \nu(z) >0$. So $L_1$ is bounded away from $L_0$ except at $z=iq$. In particular, no branch of $L_1$ connects to $\alpha$.	
	\item As $z \rightarrow iq,\  \varphi_1(z) \sim -\eta + c(z-iq)^{1/2}$, so one branch of $L_1$ terminates at $iq$.
	\item As $z \rightarrow \infty,\ \varphi(z) = 2tz^2+2(x+L)z + \bigo{1/z}$, so in addition to the real line, one trajectory of $L_1$ leaves $\C^+$ by approaching infinity along a curve asymptotic to the line $\re z = - (x+L)/2t$ 
	\item Near any zero, $p_k$ of $\rho_1$, $\varphi_1 \sim \varphi_1(p_k) + c(z-p_k)^2$. If $\varphi_1(p_k)$ is real, four trajectories of $L_1$ leave $p_k$ separated by angles of $\pi/2$.
\end{itemize}   
The last points merits further investigation. Clearly, $ \rho_1(0^-) < 0$ and $\lim_{\lambda \rightarrow -\infty} \rho_1(\lambda) = -\infty$. However, for any fixed $\lambda_0< 0$, $\lim_{t \rightarrow 0} \rho_1(\lambda_0) = 4L\lambda_0/\nu(\lambda_0) > 0$. Thus, for each sufficiently smalll $t$, $\rho_1$ has at  least two real, negative zeros. In fact, these two are the only zeros of $\rho_1$ in the complex plane; estimates of the quartic equation underlying $\rho_1(\lambda)= 0$ show that the quartic always has two positive real zeros which lie on the other sheet of the Riemann surface associated to $\rho_1$. Now fix $\mu \in [0,\sqrt{2}q]$, thus fixing the values of $\alpha$ and $\xi_0$. For $t$ sufficiently small such that the real zeros exist label them $\xi_1$ and $z_1$, ordered $\xi_1 < z_1 < 0$. As $t$ increases the two zeros monotonically approach each other eventually reaching a breaking time $T_2$ at which the two zeros coalesce. For times $t > T_2$ the zeros become complex and $\varphi_1$ has no real critical points. Note that because we solve for $T_2$ along lines $\mu= \text{const.}$, $T_2 = T_2(x) > T_1(x)$ for each $x \in [0,L)$ (cf. \eqref{S1}). This breaking time $T_2(x)$ is uniquely characterized as follows:
\begin{equation}\label{second breaking}
	T_2(x) \text{ is the unique time such that } \left\{ \begin {array}{r} \rho_1 = 0_{\Bspace} \\ \deriv{\rho_1}{z} = 0 \end{array} \right. \text{ has a simultaneous solution.}
\end{equation} 
\begin{figure}[htb]
\begin{center}
	\includegraphics[width=.3\textwidth]{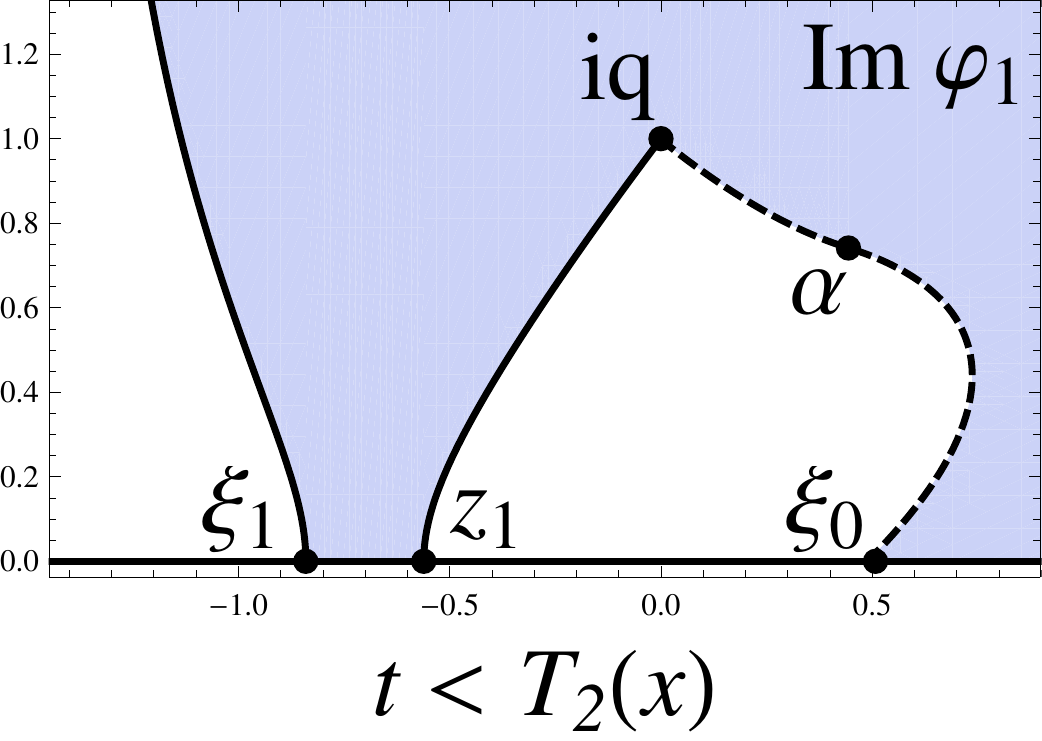}
	\quad 
	\includegraphics[width=.3\textwidth]{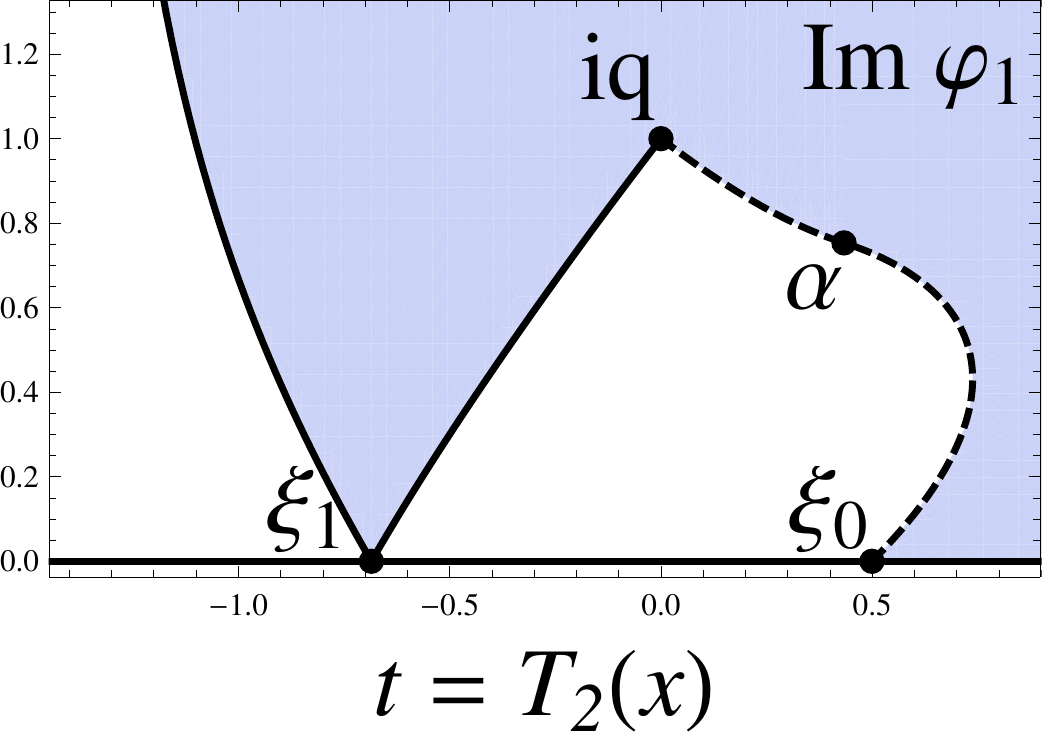}
	\quad 
	\includegraphics[width=.3\textwidth]{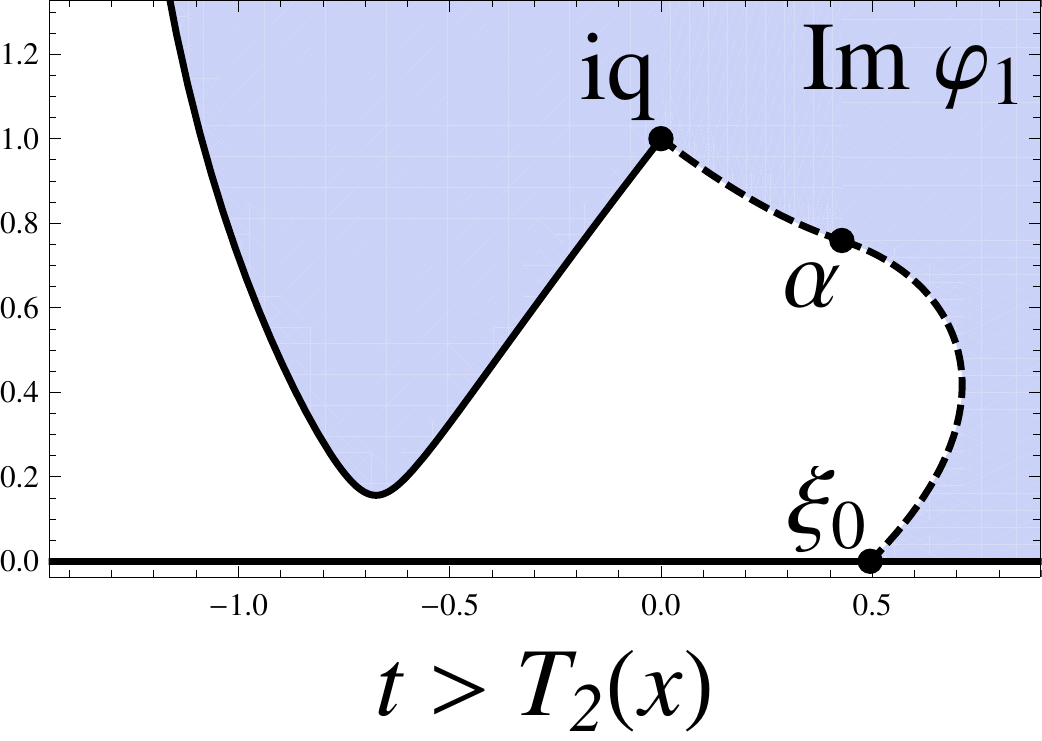}
\caption{Numerical calculation of $\imag \varphi_1$ as $t$ increases to and beyond $T_2(x)$. Solid lines indicate where $\imag \varphi_1=0$ and dashed lines the branch cut of $\varphi_1$ Note that once $t > T_2(x)$, the region $\imag \varphi_1 >0$ (shaded region) no longer reaches the negative real axis making it impossible to satisfy \eqref{two cut gap cond 2}.
\label{fig: two band varphi1}
}
\end{center}
\end{figure}	 
For each $t>T_1(x)$, the connection problem is completely determined by the local structure at each critical point of $\varphi_1$. For $T_1(x) < t < T_2(x)$ the level set $L_1(x,t)$ in $\overline{\C^+}$ consists of the real line, a branch connecting $\xi_1$ to infinity and a second branch connecting $z_1$ to $iq$. As $t$ increases beyond the second breaking time $T_2(x)$ the two branches of $L_1$ in $\C^+$ meet on the real line and then move into the complex plane so that, for $t>T_2(x)$, $L_1$ has a single branch in $\C^+$ connecting $iq$ to infinity, see Figure \ref{fig: two band varphi1}.  

The second breaking time $t=T_2(x)$ limits the validity of the genus-one ansatz for the $g$-function. To satisfy the final condition \eqref{two cut band cond} the contour $\Gamma_1$ must be chosen so that it meets the real axis at a point $p$ such that the region enclosed between $\Gamma_1$ and the real interval $(p,\infty)$ completely encloses the locus of pole accumulation $i(0,q)$, and must lie everywhere in the region $\imag \varphi_1 >0$. For $T_1(x)< t <T_2(x)$, the condition is satisified by taking $\Gamma_1$ such that it leaves the real axis at $\xi_1$, the leftmost real critical point of $\varphi_1$ lying everywhere in the region $\imag \varphi_1 >0$. However, for $t>T_2(x)$ the set $\imag \varphi_0 > 0$ pulls away from the negative real axis, so that given any choice of $\Gamma_1$ a new gap interval emerges across which condition \eqref{two cut gap cond 2} fails. For $t > T_2(x)$ the situation should be rectified by replacing the genus-one ansatz with a genus-two $g$-function with a new band of support across this newly opening gap. However, our focus in this paper is the semiclassical regularization of the square barrier so we leave the higher genus transitions as a line of investigation to be pursued later. For now we state the following proposition which summarizes the above discussion.
\begin{prop}\label{prop: genus one region}
Let  
\begin{equation}\label{S2}
	\mathcal{S}_2 = \left\{ (x,t)\, :\, x \in (0,L), \  t \in \lp T_1(x),\, T_2(x) \rp \right\}.
\end{equation} 
The genus-one ansatz, consisting of the $g$-function defined by \eqref{two band rho def} and \eqref{two band rho explicit}, the band contour $\gamma_b$ connecting $iq$ and $\alpha$, and the gap contours $\gamma_g$ and $\Gamma_1$ exist such that the system of band and gap inequalities \eqref{two band band/gap cond} are satisfied for each $(x,t) \in \mathcal{S}_2$.
\end{prop}   

\subsection{Removing the remaining oscillations: $N \mapsto Q$.}  
The completed definition of the $g$-function results in the following Riemann-Hilbert problem for $N(z)$ defined by \eqref{two band N}:
\begin{rhp}[ for N(z).]\label{rhp: two band N}
Find a $2 \times 2$ matrix $N(z)$ such that
\begin{enumerate}[1.]
	\item $N(z)$ is analytic for $z \in \C \backslash \Gamma_M$.
	\item $N(z) = I + \bigo{1/z}$ as $z \rightarrow \infty$.
	\item For $z \in \Gamma_M$ $N$ assumes continuous boundary values, $N_+$ and $N_-$, for $z \in \Gamma_M$ which satisfy $N_+ = N_- V_N$ where,
	\begin{equation}\label{two band N jumps}
	V_N = \begin{cases}
		e^{i g \ad \sig/\eps} \lp R^\dagger R 	\rp & z \in (-\infty, \xi_1) \\
		e^{i g \ad \sig/\eps} \lp R_0^\dagger R_0 	\rp & z \in (\xi_1, \infty) \\
	 	e^{i g \ad \sig/\eps} \lp R_0^{-1}R \rp & z \in \Gamma_1 \\
		e^{i g \ad \sig/\eps} \lp R^\dagger R_0^{-\dagger} \rp & z \in \Gamma_1^* \\
	         \tbyt{ e^{-i(\varphi_{0-} + \eta)/\eps} }{0}
	         {w e^{-i\eta/\eps}}{e^{-i(\varphi_{0-} + \eta)/\eps} }_{\Bspace}	& z \in \gamma_b \\
		\tbyt{ e^{-i(\varphi_{0-} + \eta)/\eps} } {-w^* e^{i\eta/\eps}}
		{0}{e^{-i(\varphi_{0-} + \eta)/\eps} }_{\Bspace}	& z \in \gamma_b^* \\
		\tbyt{ e^{-i\Omega/\eps}}{0} {w e^{i (\varphi_{0+} + \Omega)/\eps}}
		{e^{i\Omega/\eps}}_{\Bspace}	& z \in \gamma_g \\
		\tbyt{ e^{-i\Omega/\eps}}  {-w^* e^{-i (\varphi_{0+} + \Omega)/\eps}}
		{0} {e^{i\Omega/\eps}}	& z \in \gamma_g^*.
	\end{cases}
	\end{equation}
\end{enumerate}
\end{rhp}
\noindent
The jumps along $\R,\ \Gamma_1,$ and $\Gamma_1^*$ follow directly from definition \eqref{two band N}, while the jumps along the bands and gaps follow from \eqref{two band g jumps} and \eqref{two band varphi def}. The important observation to make is that, as the constants $\eta$ and $\Omega$ are real, the introduction of the genus-one $g$-function has removed the exponential growth from the jumps along $\Gamma_\nu \cup \Gamma_\nu^*$ introduced by the pole removing factorization, replacing it with oscillations in the bands and exponential decay in the gaps. 

To arrive at a RHP which is asymptotically stable in the semiclassical limit we have to now introduce factorizations which deform the remaining oscillatory jumps onto contours on which they decay exponentially to identity. These jumps, which lie on the real axis and the bands $\gamma_b$ and $\gamma_b^*$, are of the same character as those in RHP \ref{rhp: M outside} and RHP \ref{rhp: one band N} considered previously in the exterior and genus zero cases, respectively. Without repeating the details (c.f. sections \ref{sec: left factorization}-\ref{sec: middle factorization}), the necessary factorizations are identical to those introduced in the preceding cases. To define our transformation we introduce contours $\Gamma_k$ and regions $\Omega_k, \ k=0,2,$ or $3$ in $\C^+$ as follows. Take $\Gamma_0$ to be a semi-infinite ray leaving $\gamma_g$ at a fixed point bounded away from $\alpha$ and $\xi_0$ lying everywhere in the region $\varphi_0 >0$ and oriented toward infinity; $\Gamma_2$ is a semi-infinite ray lying everywhere in the region $\imag \varphi_1 < 0$ oriented toward $\xi_1$ where it meets the real axis; $\Gamma_3$ is a finite contour  consisting of two pieces, the first oriented from $\xi_1$ to $\alpha$ passing over the band $\gamma_b$, and the second oriented from $\alpha$ to $\xi_0$ passing under $\gamma_b$: all of $\Gamma_3$ lying in the region $\imag \varphi_0 < 0$. We note as well that the contour $\Gamma_0$ originates at a point $p$ along the gap contour $\gamma_g$ and naturally splits $\gamma_g$ into two pieces: $\gamma_g^{\text{up}}$ oriented from $p$ to $iq$ and $\gamma_g^\text{down}$ oriented from $\xi_0$ to $p$. As usual we denote the conjugate contours as ${\gamma_g^\text{up}}^*$ and  $\gamma_g^{\text{down}*}$.   The corresponding sets $\Omega_k$ are those enclosed by their respective $\Gamma_k$ and the real axis, see Figure \ref{fig: two band Q contours}. With these definitions in hand, the following transformation defines a new unknown $Q(z)$ whose jumps are either near identity or otherwise well approximated by explicit factors for which a parametrix can be constructed. 

Define 
\begin{equation}\label{two band Q}
\begin{split}
	Q(z) &= N e^{ig \ad \sig/ \eps} \, L_Q \\
	L_Q &= \begin{cases} 
		R_0^{-1} & z \in \Omega_0 \\
		\widehat{R}^{-\dagger} \lp a / a_0 \rp^{\sig} & z \in \Omega_2 \\
		\widehat{R}_0^{-\dagger} & z \in \Omega_3 \\
		\widehat{R}_0 & z \in \Omega_3^* \\
		\widehat{R}  \lp a^* / a^*_0 \rp^{-\sig}  & z \in \Omega_2^* \\
		R_0^\dagger & z \in \Omega_0^* \\
		I & \text{elsewhere}.
	\end{cases}	
\end{split}
\end{equation}
Let $\Gamma^0_Q = \bigcup_{k=0}^3 \lp \Gamma_k \cup \Gamma_k^* \rp$ and $\Gamma_Q = (-\infty, \xi_0) \cup \Gamma_\nu \cup \Gamma_\nu^* \cup \Gamma_Q^0$. RHP \ref{rhp: two band N}, the definition of $Q$, and the factorization formulas: \eqref{two factor}, \eqref{left factorization}, \eqref{intermediate factorization}, and \eqref{band factorization} result in the following RHP for the new unknown.

\begin{rhp}[for $Q(z)$:]\label{rhp: two band Q}
Find a $2 \times 2$ matrix $Q(z)$ such that:
\begin{enumerate}[1.]
	\item $Q$ is analytic for $z \in \C \backslash\Gamma_Q$. 
	\item $Q(z) = I + \bigo{1/z}$ as $z \rightarrow \infty$.
	\item $Q$ takes continuous boundary values $Q_+$ and $Q_-$ on $\Gamma_Q$ satisfying the jump relation $Q_+ = Q_-  V_Q$ where $V_Q =$ 
	\begin{equation}\label{two band Q jumps}
		 \begin{cases}
		 	e^{i g \sig /\eps}\, V_Q^{(0)}\, e^{-i g \sig / \eps}_{\Bspace} & z \in \Gamma^0_Q \\
			\offdiag{-w^{-1} e^{i\eta/\eps} }{w e^{ -i\eta /\eps} }_{\Bspace} & z \in \gamma_b \\
			\offdiag{-w^* e^{i\eta/ \eps} }{ {w^*}^{-1} e^{ -i \eta / \eps} }& z \in \gamma_b^* \\
			(1+ | r_0 |^2)^{2\sig} & z \in (-\infty, \xi_1) \\
			 (1+| r_0 |^2)^\sig & z \in (\xi_1, \xi_0) 
		\end{cases}
		\quad
		\begin{cases}
			\tbyt{e^{-i\Omega/\eps}}{0}{w e^{i(\varphi_{0+}+\Omega)/\eps}}{e^{i\Omega/\eps}}_{\Bspace} & z \in  \gamma_g^{\text{up}} \\
			\tbyt{e^{-i\Omega/\eps}}{0}{r_{0_+} e^{i(\varphi_{0+}+\Omega)/\eps}}{e^{i\Omega/\eps}}_{\Bspace} & z \in  \gamma_g^{\text{down}} \\
			\tbyt{e^{-i\Omega/\eps}} {-w^* e^{-i(\varphi_{0+}+\Omega)/\eps}} {0} {e^{i\Omega/\eps}}_{\Bspace} & z \in  \gamma_g^{\text{up}*} \\			
			\tbyt{e^{-i\Omega/\eps}}{-r^*_{0_+} e^{-i(\varphi_{0+}+\Omega)/\eps}} {0} {e^{i\Omega/\eps}}_{\Bspace} & z \in  \gamma_g^{\text{down}*} \\			\end{cases},
	\end{equation}
	and $V_Q^{(0)}$ was previously given by  \eqref{V_Q^0 jumps}.
	\item $Q$ admits at worst square root singularities at $z = \pm iq$ satisfying the following bounds
	\begin{equation}\label{two band Q sing}
	\begin{split}
		Q(z) &= \bigo{ \begin{array}{cc} 1 & |z - iq|^{-1/2} \\  1 & |z - iq|^{-1/2} \end{array} }, \quad z \rightarrow iq,  \\
		Q(z) &= \bigo{ \begin{array}{cc} |z + iq|^{-1/2} & 1 \\  |z + iq|^{-1/2} & 1 \end{array} }, \quad z \rightarrow -iq.
	\end{split}
	\end{equation}
\end{enumerate}
\end{rhp}
\begin{figure}[thb]
\begin{center}
	\includegraphics[width=.8\linewidth]{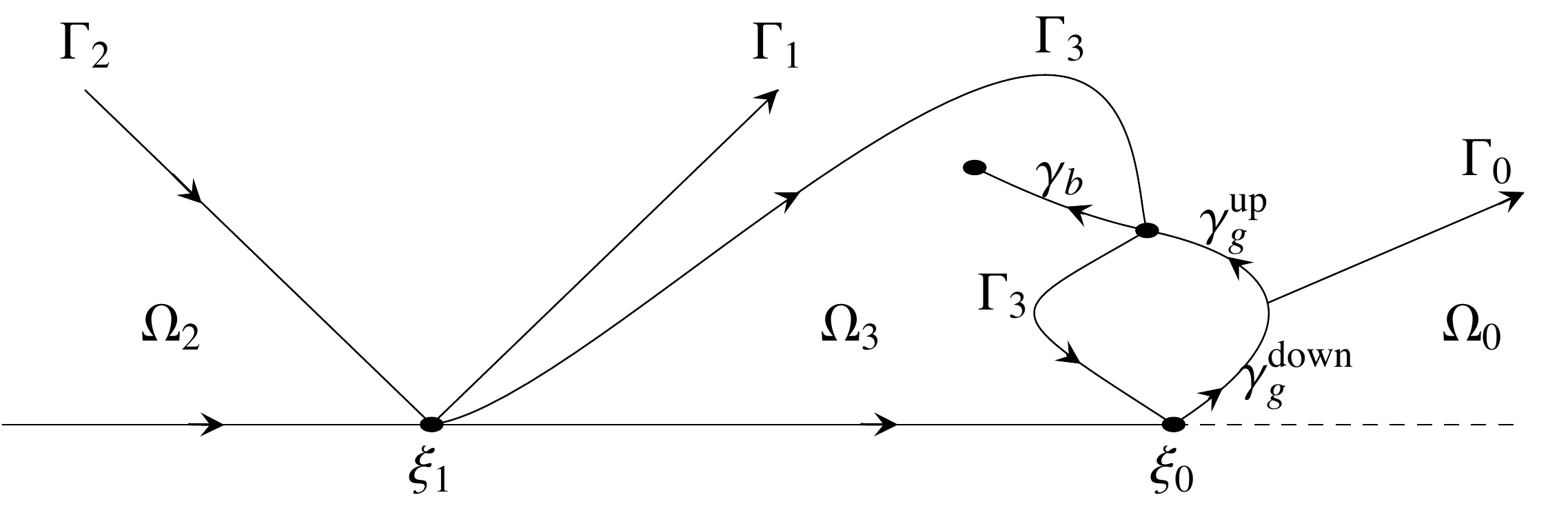}
	\caption{Schematic diagram of the contours $\Gamma_k$ and regions $\Omega_k$ in $\C^+$ used to define the transformation $N \mapsto Q$. Their counterparts $\Gamma_k^*$ and $\Omega_k^*$  in $\C^-$ are defined by conjugation symmetry.
	\label{fig: two band Q contours}
	}
\end{center}
\end{figure} 

\subsection{Constructing a global parametrix in the genus one case} 
The RHP for $Q(z)$ that results from the several factorizations above, while defined on an elaborate set of contours, in now well conditioned to semi-classical approximations. The jumps along the lens opening contours $\Gamma_Q^0$ all decay exponentially to identity both as $z \rightarrow \infty$ and as $\eps \rightarrow 0^+$ for each fixed $z \in \Gamma_Q^0$ away from the stationary phase points $\xi_0$ and $\xi_1$ and the branch points $\alpha$ and $\alpha^*$. The remaining jumps along the real axis and the band and gap contours all have simple asymptotic behaviors in the semiclassical limit. Along the real axis the remaining jumps are independent of $\eps$ and in the bands and gaps, the terms in each jump which are not exponentially vanishing are comprised of products of $\eps$-independent factors with constant, $\eps$-dependent, complex phases. 

We now begin the construction of a global parametrix $P$ with the goal that the error resulting from this approximation $E = QP^{-1}$ is uniformly near identity. The parametrix is necessarily piecewise constructed due to the non-uniformity of the outer approximation near the branch and stationary phase points. We seek $P$ in the form:
\begin{equation}\label{two band parametrix}
	P = \begin{cases}
		A_0	& z \in \U_0 \\
		A_1  & z \in \U_1 \\
		A_\alpha & z \in \U_\alpha \\
		A_{\alpha^*} & z \in \U_{\alpha^*} \\
		O & z\ \text{elsewhere}
	\end{cases}
\end{equation}
where $\U_0$, $\U_1$, $\U_\alpha$ and $\U_{\alpha^*}$ are sufficiently small, fixed sized, neighborhoods of $\xi_0$,  $\xi_1$, $\alpha$, and $\alpha^*$ respectively. The local problems are all essentially solved: as in the previous cases, the local models near $\xi_0$ and $\xi_1$ are described by parabolic cylinder functions while the new models introduced near the interior endpoints $\alpha$ and $\alpha^*$ are described, as we will show, by the standard Airy parametrix. The new outer model in the genus one case is completely different than the outer model constructed in the previous section for the genus zero problem. 

\subsubsection{The Genus One Outer Model, $O(z)$} We arrive at the outer model problem by simply dropping those terms of $V_Q$ which are exponentially vanishing for each fixed $z$ as $\eps \rightarrow 0^+$. This results in replacing all of the jumps along $\Gamma_Q^0$ with identity and dropping the off-diagonal entries of $V_Q$ along the gap contours $\gamma_g \cup \gamma_g^*$. The resulting outer model problem is given by:
\begin{rhp}[for the genus one outer model, $O(z)$:]\label{rhp: two band outer} 
Find a $2 \times 2$ matrix valued function $O(z)$ satisfying the following properties:
\begin{enumerate}[1.]
	\item $O$ is analytic for $z \in \C \backslash \Gamma_O, \quad \Gamma_O:= \Gamma_\nu \cup \Gamma_\nu^* \cup (-\infty, \xi_0]$.
	 \item $O(z) = I + \bigo{1/z}$ as $z \rightarrow \infty$.
	 \item $O$ takes continuous boundary values $O_+$ and $O_-$ on $\Gamma_O$ away from its endpoints. The boundary values satisfy the jump relation $O_+ = O_- V_O$ where
	 \begin{equation}\label{two band O jumps}
	 	V_O(z) = \begin{cases}
			\offdiag{-w(z)^{-1}e^{i\eta/\eps} }{w(z) e^{-i \eta/ \eps} }_{\Bspace} & z \in \gamma_b \\
			\offdiag{-w^*(z)e^{i\eta/\eps} }{{w^*(z)}^{-1}e^{-i\eta/\eps} }_{\Bspace} & z \in \gamma_b^* \\
			\diag{e^{-i\Omega/\eps}}{e^{i\Omega/\eps}}_{\Bspace} & z \in \gamma_g \cup \gamma_g^* \\
			(1+|r_0(z)|^2 ) ^{2\sig} & z \in (-\infty, \xi_1) \\
			(1+|r_0(z)|^2 ) ^{\sig} & z \in (\xi_1, \xi_0).
			\end{cases}
	\end{equation}
	\item $O$ is everywhere bounded except near the endpoints  $z = \pm iq,\ \alpha$, and $\alpha^*$ where it satisfies the following bounds
	\begin{equation}\label{two band O sing}
	\begin{array}{rll}
		O(z) &= \bigo{ \begin{array}{cc} 1 & |z - iq|^{-1/2} \\  1 & |z - iq|^{-1/2} \end{array} }, &  z \rightarrow iq,  \\
		O(z) &= \bigo{ \begin{array}{cc} |z + iq|^{-1/2} & 1 \\  |z + iq|^{-1/2} & 1 \end{array} }, & z \rightarrow -iq, \\
		O(z) &= \bigo{ \begin{array}{cc} |z-p|^{-1/4} & |z-p|^{-1/4} \\  |z-p|^{-1/4} & |z-p|^{-1/4} \end{array} }, & z \rightarrow p, \ p =\alpha \text{ or } \alpha^*.
	\end{array}
	\end{equation}
\end{enumerate}
\end{rhp}
\begin{rem} The replacement of the jumps in $\Gamma_Q^0$ with identity results in the outer solution, $O(z)$, being necessarily singular near each of the endpoints $\alpha,\ \alpha^*,\ \xi_0,$ and $\xi_1$ (in addition to the given singularities at $\pm iq$). The choice of quarter root singularities at the gap endpoints and bounded singularities at the stationary phase points is a choice we make to match the local Airy and parabolic cylinder local models which we will introduce in neighborhoods of each pair of points respectively.  
\end{rem}	 

The solution of this problem is given in terms of elliptic theta functions associated with the Riemann surface naturally associated to the rational function 
\begin{equation}\label{two band R}
	\mathcal{R}(z) = \sqrt{(z-iq)(z-\alpha)(z-\alpha^*)(z+iq)}.
\end{equation}
To be concrete, we will always understand $\mathcal{R}$ to be branched along $\gamma_b \cup \gamma_b^*$ and normalized such that $\mathcal{R}(z) \sim z^2$ for large $z$. However, to arrive at this solution we must first reduce the outer model to the canonical constant jump form by introducing two scalar functions which remove the $z$-dependence from the jump matrices. The first of these is the function $\delta (z)$ defined by \eqref{delta}. This function, as in the previous sections, we use to remove the jumps of $O(z)$ along the real axis. To remove the $z$-dependence in the bands we introduce the following scalar Riemann Hilbert problem:  
\begin{rhp}[ for $s(z)$:] \label{rhp: two band s}
Find a scalar function $s(z)$ with the following properties:
\begin{enumerate}[1.]
\item $s$ is analytic for $z \in \C \backslash (\Gamma_\nu \cup \Gamma_\nu^*)$.
\item $s(z)$ takes continuous boundary values, $s_+$ and $s_-$, on $\Gamma_\nu \cup \Gamma_\nu^*$ satisfying
\begin{equation}\label{two band s jumps}
\begin{array}{r@{\ = \ }l@{\quad:\quad}l}
	s_+(z) s_-(z) &  i w(z) \delta(z)^{-2} e^{-i \eta / \eps} & z \in \gamma_{b \Bspace}, \\
	s_+(z) s_-(z) &  i w^*(z)^{-1} \delta(z)^{-2} e^{-i \eta / \eps} & z \in \gamma_{b \Bspace}^*, \\
	s_+(z) \big{/}s_-(z) & e^{-i\Omega/\eps} & z \in \gamma_g \cup \gamma_g^*. 
\end{array}  
\end{equation}
\item As $z \rightarrow \infty$, $s(z) =  e^{ip(z)}  \lp 1 + \bigo{1/z} \rp$, for some linear function $p(z)$.
\item $s(z)$ is bounded and nonzero on any compact set not containing the endpoints $z=\pm iq$ where 
\begin{align*}
	s(z) &= \bigo{ (z-iq)^{1/4} }, \quad \text{ as } z \rightarrow iq, \\
	s(z) &= \bigo{ (z+iq)^{-1/4} }, \quad \text{as } z \rightarrow -iq.
\end{align*}
\end{enumerate}	
\end{rhp}

This problem is very closely related to RHP \ref{rhp: one band s} which was part of the constant jump reduction in the genus zero outer model.
\begin{prop}\label{prop: two band s}
The solution of RHP \ref{rhp: two band s} is given by $s(z) = s_0(z) s_1(z)$ where
\begin{equation}\label{two band s}
	\begin{split}
	s_0(z) &= a(z) \exp \lp \frac{i\pi}{4}  + \frac{\mathcal{R}(z)}{2\pi i} \int\limits_{\Gamma_\nu \cup \Gamma_\nu^*} \frac{ j(\lambda) }{\mathcal{R}_+(\lambda)} \frac{\dd \lambda}{\lambda - z} \rp, \\
	s_1(z) &= \exp \lp - \frac{i\eta}{2\eps} +\frac{\mathcal{R}(z)}{2\pi i \eps} \int\limits_{\gamma_g \cup \gamma_g^*} \frac{-i\Omega }{\mathcal{R}(\lambda)} \frac{\dd \lambda}{\lambda - z}  \rp.
\end{split}
\end{equation}
Here $a(z)$ is given by \eqref{a} and
\begin{equation*}
	j(\lambda) = \begin{cases}
		\log \lp \frac{2(z+iq)}{q} \rp  - 2\lp \chi(z,\xi_1) + \chi(z,\xi_0) \rp & z \in \gamma_b \\ 
		 -\frac{i\pi}{2}   & z \in \gamma_g \cup \gamma_g^* \\
		\log \lp \frac{q}{2(z-iq)} \rp  - 2\lp \chi(z,\xi_1) + \chi(z,\xi_0) \rp & z \in \gamma_b
	\end{cases}
\end{equation*}		
where $\chi(z,a)$ is defined by \eqref{kappa}. Moreover, the polynomial $p(z)$ characterizing the essential singularity of $s$ at infinity is given by
\begin{equation}\label{p poly}
	p(z)= p_0(z) + \frac{1}{\eps}p_1(z) \\ 
\end{equation}
where
\begin{equation*}
	\begin{split}
		p_0(z) &= \frac{1}{2\pi} \int\limits_{\Gamma_\nu \cup \Gamma_\nu^*} (z-\lambda-\re \alpha) \frac{ j(\lambda) }{\mathcal{R}_+(\lambda)}  \dd{\lambda} +   \frac{\pi}{4}, \\
		p_1(z) &=\frac{1}{2\pi }  \int\limits_{\gamma_g \cup \gamma_g^*} (z-\lambda-\re \alpha) \frac{ \Omega }{\mathcal{R}(\lambda)}  \dd{\lambda} - \frac{\eta}{2}.
	\end{split}
\end{equation*}
\end{prop}

\begin{proof} That \eqref{two band s} satisfies the jump condition \eqref{two band s jumps} follows from the Plemelj formulae and the boundary behavior of $\mathcal{R}(z)$ and $a(z)$ along their respective branch cuts. The asymptotic behavior at infinity and the reality of $p_0$ and $p_1$ follow from expanding the Cauchy integral defining $s(z)$ and the symmetries $j(\lambda^*)^* = - j(\lambda)$ and $(\Gamma_\nu \cup \Gamma_\nu^*)^* =  (\Gamma_\nu \cup \Gamma_\nu^*)^{-1}$ (with respect to orientation of the path). Finally, classical estimates from the theory of singular integrals guarantee that the Cauchy integrals defining $s(z)$ grow at worst as an inverse square root at each endpoint \cite{Musk46}, behavior exactly balanced by the rational pre-factor $\mathcal{R}(z)$. Thus $s(z)$ is bounded and nonzero at each finite $z$ except the points $\pm iq$ where it inherits the singular behavior of $a(z)$.
\end{proof} 

Using the functions $s(z)$ and $\delta(z)$ we seek the outer model in the form
\begin{equation}\label{elliptic model def}
	O(z) = \beta(z) \,  O^{(1)}(z) s(z)^\sig \delta(z)^\sig,
\end{equation}
where $\beta$ is the scalar function 
\begin{equation}\label{beta}
	\beta(z) = \lp \frac{z-iq}{z-\alpha} \rp^{1/4} \lp \frac{ z- \alpha^*}{z+iq} \rp^{1/4}
\end{equation}
cut along the bands $\gamma_b \cup \gamma_b^*$ and normalized to approach unity for large $z$. The resulting RHP for the new unknown $O^{(1)}(z)$ is analytic away from the bands where it has constant column permutation jumps and square root singularities at half of the branch points. 
\begin{rhp}[ for $O^{(1)}(z)$:]\label{rhp: O1}
Find a $2 \times 2$ matrix valued function $O^{(1)}(z)$ with the following properties:
\begin{enumerate}[1.]
	\item $O^{(1)}$ is analytic for $z \in \C \backslash (\gamma_b \cup \gamma_b^*)$.
	\item As $z \rightarrow \infty$, $O^{(1)}(z) = \left[ I + \bigo{1/z} \right] e^{-i p(z) \sig}$ where $p(z)$ is the linear function \eqref{p poly}.
	\item $O^{(1)}(z)$ assumes continuous boundary values, $O^{(1)}_+(z)$ and $O^{(1)}_-(z)$, on $\gamma_b$ and $\gamma_b^*$ which satisfy the jump relation
	\begin{equation}\label{O1 jumps}
		O^{(1)}_+(z) = O^{(1)}_-(z) \offdiag{1}{1}, \qquad z \in \gamma_b \cup \gamma_b^*.
	\end{equation}
	\item $O^{(1)}(z)$ is bounded on any compact set not containing the points $z=iq$ or $z= \alpha^*$ where it admits at worst square-root singularities.
\end{enumerate}
\end{rhp}

Riemann-Hilbert problems like $O^{(1)}$ are well known in the literature of integrable systems and random matrices \cite{DIZ97, DKMVZ99a, KMM03, TVZ04}. We provide the details of its construction here for completeness. The construction begins by `lifting' it onto the genus one Riemann surface associated with $\mathcal{R}(z) = \sqrt{ (z-iq)(z-\alpha)(z-\bar\alpha)(z+iq)}$. We denote this surface by $\Sigma$ and label its two sheets $\Sigma_1$ and $\Sigma_2$ arbitrarily. For any non-branch point $z \in \C$ we let $P^k(z),\, k=1,2,$ denote its pre-image on the corresponding sheet. The problem is then lifted onto $\Sigma$ by seeking $O^{(1)}(z)$ in the form  
\begin{equation}\label{lifting}
	O^{(1)}(z) = \begin{bmatrix} \  \vec{v}\,(P^1(z))\  ,\  \vec{v}\,(P^2(z))\  \end{bmatrix},
\end{equation}
where $\vec{v}(P)$ is a single vector-valued function defined for $P \in \Sigma$. The jump relation \eqref{O1 jumps} implies that the new unknown $\vec{v}(P)$ is holomorphic away from the branch points and the two pre-images of infinity: $\infty_1$ and $\infty_2$. Near the branch points $z=\alpha$ and $z=-iq$ the bounded behavior of $O^{(1)}(z)$ implies that any singularity of $\vec{v}(P)$ is removable, while the square-root singularities of $O^{(1)}(z)$ at $z=iq$ and $z=\alpha^*$, due to the double-ramification of the branch points, become poles of the function $\vec{v}(P)$. These properties together with the singular behavior at each infinity completely specify the function $\vec{v}(P)$:

\textbf{Problem for $\vec{v}(P)$:} Find a vector-valued function $v:\Sigma \rightarrow \C^2$  satisfying the following properties:
\begin{itemize}

\item $\vec{v}$ is meromorphic on $\Sigma \backslash \{\infty_1, \infty_2 \}$, and if $(v_k)$ denotes the divisor of the component $v_k,\ k=1,2$, over $\Sigma \backslash \{\infty_1, \infty_2 \}$ then,
\begin{equation}
	(v_k) + P_{\alpha^*} + P_{iq} \geq 0.
\end{equation}

\item $\vec{v}$ is essentially singular at each infinity and
\begin{align}
	\label{vsing}
	\begin{split}
	\vec{v}(P) e^{ip(z)}  &= 1 + \bigo{1/z}, \quad z \rightarrow \infty_1, \\
	\vec{v}(P) e^{-ip(z)}  &= 1 + \bigo{1/z}, \quad z \rightarrow \infty_2. 
	\end{split}
\end{align}
\end{itemize}

Finding the function $\vec{v}(P)$ necessary to solve the RHP \ref{rhp: O1} is a classical problem in the study of Riemann surfaces whose solution can be constucted from Baker-Akhiezer functions. These functions are represented here in terms of ratios of the $\Theta$-functions corresponding to the Riemann surface $\Sigma$. The reader is referred to \cite{FK80} for a general review of the theory of Riemann surfaces and our construction of the Baker-Akhiezer functions follows \cite{Dub81}. As a first step towards finding $\vec{v}$, we introduce a function $\vec{f}$ with the following properties:
\begin{itemize}
	\item The component functions  $f_1(P)$ and $f_2(P)$ are meromorphic functions on $\Sigma$ such that 
	\begin{align*}
		(f_1)  + P_{iq} + P_{\alpha^*} - \infty_2 \geq 0, \\
		(f_2)  + P_{iq} + P_{\alpha^*} - \infty_1 \geq 0. 
	\end{align*}
	\item Each component $f_k$ is normalized such that $f_k(\infty_k) = 1,\ k=1,2$.
\end{itemize}
The existence and uniqueness of such a function can be proved abstractly, but by direct inspection, the functions 
\begin{align}\label{Felliptic}
	f_{1,2}(P) = \frac{1}{2} \left[ 1 \pm \frac{(z(P)-\alpha)(z(P)+iq)}{\mathcal{R}_\Sigma(P)} \right]
\end{align}
satisfy the above properties. Here $z(P)$ is the projection of $P \in \Sigma$ onto $\C$ and $\mathcal{R}_\Sigma$ is the lifting of \eqref{two band R} to the Riemann surface such that $\mathcal{R}_\Sigma(P) \sim z^2$ as $P \rightarrow \infty_1$. Clearly, each component $f_k$ vanishes at the appropriate infinity and has the correct singularities at $P_{\alpha^*}$ and $P_{iq}$. Each component $f_k$ necessarily vanishes at precisely one additional (finite) point in $\Sigma$ which we label $P_k$. Thus, the complete divisor of each component is given by
\begin{align}\label{Fdivisor}
	\begin{split}
	(f_1) &= P_1 + \infty_2 - P_{iq} - P_{\alpha^*}, \\
	(f_2) &= P_2 + \infty_2 - P_{iq} - P_{\alpha^*}.
	\end{split}
\end{align}

We now consider the ratio $\vec{\zeta}$ defined by the component-wise product 
\begin{align} \label{zeta def}
	\vec{v} = \vec{f} \cdot \vec{\zeta}.
\end{align}
The resulting unknown $\vec{\zeta}$ is the aforementioned Baker-Akhiezer function. That is, the components of $\vec{\zeta}$ have the following properties:
\begin{itemize}
	\item Each component $\zeta_k$ is meromorphic on $\Sigma \backslash\{\infty_1, \infty_2\}$ admitting at most a single simple pole at the point $P_k$.

	\item The local behavior of $\vec{\zeta}$ in a neighborhood of each infinity is given by
	\begin{align} \label{zeta sing}
		\begin{split}
		&\vec{\zeta}(P) e^{ip(z)} \rightarrow \vect{1}{c_2} \qquad P \rightarrow \infty_1, \\ 
		&\vec{\zeta}(P) e^{-ip(z)} \rightarrow \vect{c_1}{1} \qquad P \rightarrow \infty_2.
		\end{split}
	\end{align}
	where $p(z)$ is the linear function given by \eqref{p poly} and $c_k,\  k=1,2$ are unknown constants. 
\end{itemize}

To construct $\zeta$ we introduce several standard devices on the Riemann surface $\Sigma$. First, we fix a homology basis. As $\Sigma$ is genus one, the basis consists of only two elements, $\{ a,\, b \}$.  We take the $a$-loop to lie on both sheets, oriented over the first sheet from $(\gamma_b)_-$ to $(\gamma_b^*)_-$ and the $b$-loop we take completely on the first sheet as a clockwise loop enclosing $\gamma_b$ without intersecting or enclosing $\gamma_b^*$,  see Figure \ref{Fig: homology}. 
\begin{figure}[htb]
	\begin{center}
		\includegraphics[width=.2\textwidth]{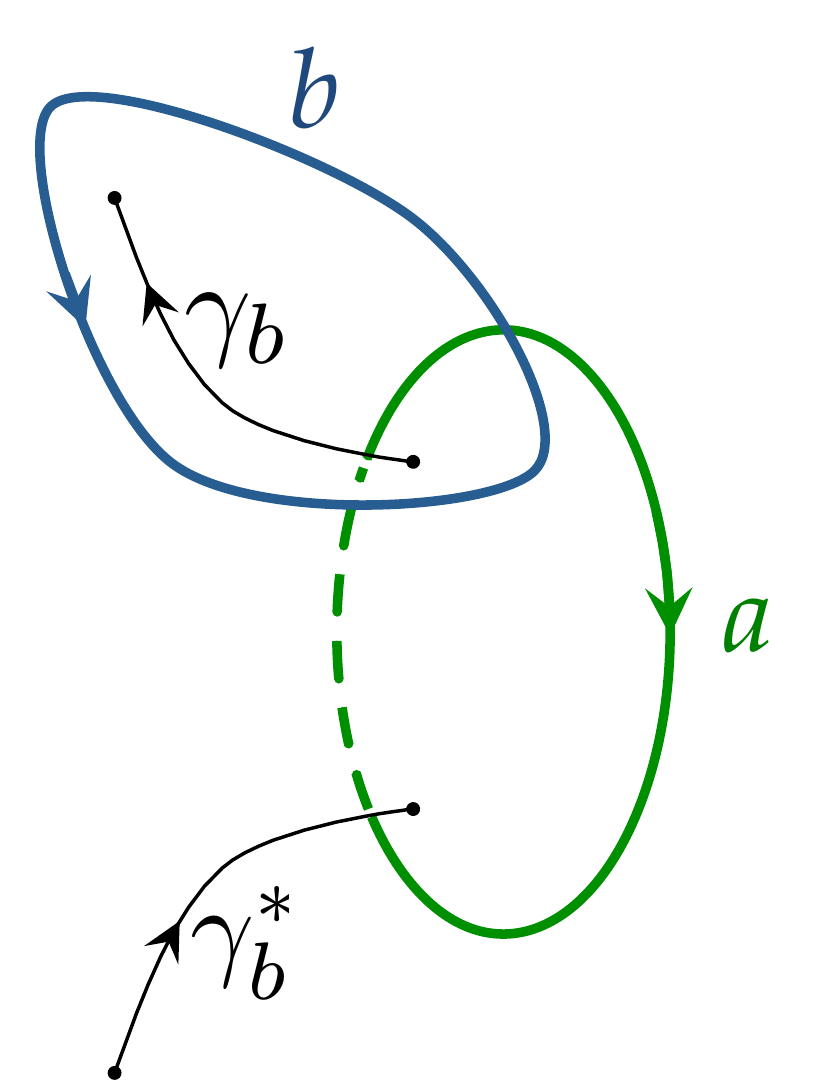}
		\caption{Choice of our homology basis $\{ a,\, b \}$ on $\Sigma$. Solid and dotted lines indicate the contour lines on the first or second sheet respectively. 
		\label{Fig: homology}}
	\end{center}
\end{figure}

Our choice of homology fixes the basis of holomorphic differentials which consists of exactly one element,
\begin{equation}
	\nu = c_{\nu} \frac{dz(P)}{\mathcal{R}_\Sigma(P)},
\end{equation}
where $c_\nu$ is a normalization constant chosen so that 
\begin{equation*}
	\oint_a \nu = 2\pi i.
\end{equation*}
We will also need the $b$-period of our basis differential, let
\begin{align}\label{b-periods}
	H = \oint_b \nu = 2\pi i  \oint_b \frac{dz(P)}{\mathcal{R}_\Sigma(P)}\, \Bigg/ \oint_a \frac{dz(P)}			{\mathcal{R}_\Sigma(P)} .
\end{align}
A standard result in the theory of Riemann surfaces is that $\re(H)$ is strictly negative \cite{FK80}. In this case we can say more; the symmetries $\nu(P^*)^* = \nu(P)$ and $b^* = b$  imply that $H$ is always a purely real (negative) number. $H$ allows us to define two important objects. The Riemann constant $K$, we take to be
\begin{equation} \label{Riemann Constant}
	K = i \pi +\frac{1}{2}H,
\end{equation}
and we take 
\begin{equation}\label{Theta}
	\Theta(w ; H) = \sum_{n \in \Z} \exp \lp \frac{1}{2}n^2H-n w \rp.
\end{equation}
as the definition of our $\Theta$ functions. $\Theta$ is an entire and even function of $w$ which satisfies the automorphic relations
\begin{align}\label{Theta periods}
	\Theta(w + 2\pi i ; H) = \Theta(w ; H), \quad \text{and} \quad
	\Theta(w + H ; H) = e^{-\frac{1}{2}H-w}  \Theta(w ; H ).
\end{align}
The zeros of $\Theta$ lie on the lattice $w = K + 2m \pi i + nH,\ n,m \in \Z$. We choose as the base point on $\Sigma$ the branch point $P = iq$, and define an Abel map $A:\Sigma \rightarrow \Jac(\Sigma)$ by
\begin{align*}\label{Abel}
	A(P) = \int_{iq}^P \nu.
\end{align*}
Finally, let $\tau$ be the abelian differential of the second kind with double poles at each pre-image of infinity, $\tau = \tau_0 + \eps^{-1} \tau_1$:
\begin{equation}
	\tau_k = \frac{z(P)^2 - \re(\alpha)z(P) +c_\tau}{\mathcal{R}_\Sigma(P)}  dp_k(z(P)), \quad k=0,1,
\end{equation}
where $p(z) = p_0(z) + \eps^{-1} p_1(z)$ is real linear polynomial \eqref{p poly}. 
The constant $c_\tau$ (independent of $k$) is chosen so that 
\begin{equation*}\label{tau norm}
		\oint_a \tau = 0.
\end{equation*} 
Note that at each infinity
\begin{align}\label{tau sing}
	\tau = dp + \bigo{ \frac{dz}{z^2}} , \ P \rightarrow \infty_1, \quad \text{and} \quad
	\tau = -dp + \bigo{ \frac{dz}{z^2}} , \ P \rightarrow \infty_2,
\end{align}
 
Let $T = T_0 + \eps^{-1} T_1$ represent the $b$-period of $\tau$: 
\begin{equation}\label{T}
	T_k = \oint_b \tau_k, \quad k = 1,2, 
\end{equation} 
and observe that the symmetries $\tau(P^*)^* = \tau(P)$ and $b^* \equiv b$ imply that $T$ is necessarily real. 

\begin{lem}\label{lem: tau} 
The differential $\tau_1$ satisfies the relations, $\oint _b \tau_1 = -\Omega$ and $s_1(z)\exp\lp -\frac{i}{\eps} \int_{-iq}^z \tau_1 \rp = 1$, where the path of integration in the exponential is restricted to the first sheet, $\Sigma_1$, and does not intersect $\Gamma_\nu \cup \Gamma_\nu^*$. 
\end{lem}

\begin{proof}
From \eqref{two band s}, $s_1(z)$ takes the form $s_1(z) = \exp ( \eps^{-1} \rho(z) )$, where 
\begin{equation*}
	\rho(z) = -i\eta/2 + \frac{ \mathcal{R} (z) } { 2\pi i } \int_{ \gamma_g \cup \gamma_g^* } 
	\frac{-i\Omega}{\mathcal{R}(\lambda)} \frac{ d \lambda}{\lambda -z} . 
\end{equation*}
Using this notation, the second half of the lemma is proved if we can show that $\rho(z) - i \int_{iq}^z \tau_1 =0$. To this end, consider the differential $d \rho$.  Clearly, $d\rho$ is holomorphic on $\C \backslash (\gamma\cup\gamma^*)$; differentiating the jump relations for $\rho$ we see that for $z \in \gamma_b \cup \gamma_b^*$, $d\rho_+  + d\rho_- = 0$  and for $z \in \gamma_g \cup \gamma_g^*$, $d\rho_+ - d\rho_- = 0$. Together these facts imply that $d\rho$ extends naturally to a meromorphic differential on the Riemann surface $\Sigma$ with singularities only at the two infinities. The difference $d\rho - i \tau$, from \eqref{tau sing}, is then a holomorphic differential. Integrating over an $a$-cycle, we have for any $\hat z \in \gamma_b$, $\oint_a d\rho - i \tau_1 = \rho_+(\hat z^*) + \rho_-(\hat z^*) - \lp \rho_+(\hat z) + \rho_-(\hat z) \rp= 0$. As no nontrivial holomorphic differential can have all its $a$-cycles vanish, the difference must vanish identically. It follows immediately, letting $\tilde z$ represent an arbitrary point in $\gamma_g$, that $T_1 = \oint_b \tau_1 = -i\oint_b d \rho = -i \lp \rho_+(\tilde z) - \rho_-(\tilde z) \rp = -\Omega$. 
\end{proof}

We now combine these many devices to give Kriechever's formula for the Baker-Akhiezer function. Define
\begin{align}\label{zeta}
\zeta_k(P) = N_k \frac{\Theta(A(P) - A(P_k) - K - iT)}{\Theta(A(P) - A(P_k) - K)} \exp\lp -i\int_{iq}^P \tau \rp 
\end{align} 
where the path of integration in the exponential factor is the same as the path in the Abel map $A(P)$ and the constant $N_k$ is uniquely chosen to satisfy the normalization condition \eqref{zeta sing}.
 
\begin{lem}\label{lem: baker}  
The function $\vec{\zeta}(P)$ defined by \eqref{zeta} is a well defined function from $\Sigma \rightarrow \C^2$ for every $\eps >0$ and solves the Baker-Akhiezer problem defined by \eqref{zeta def} with the parameters $N_k$ given by: 
\begin{equation}\label{N_k}
	N_k = \frac{ \Theta(0) }{\Theta(iT_0 - i\eps^{-1}\Omega)} \exp\lp iT_0+i(-1)^k Y_0 \rp,
\end{equation}
where
\begin{equation}\label{Y}
	Y_0 = \lim_{P \rightarrow \infty_1} \lp p_0(z(P)) - \int_{iq}^P \tau_0 \rp.
\end{equation}
\end{lem}

\begin{proof}
That $\vec \zeta(P)$ is well defined follows immediately from the automorphic relations \eqref{Theta periods} for $\Theta(P)$ and the periods of $\tau$. Let us now show that $\vec{\zeta}$ is the required Baker-Akhiezer function. Clearly, for every choice $N_k \neq 0$, $\zeta_k$ is meromorphic and its single pole is the unique zero of the $\Theta$-function in the denominator. By construction this zero is located at $P_k$. One needs only to check that $N_k$ can by chosen such that the normalization condition \eqref{zeta sing} is satisfied. Clearly, one must set
\begin{equation}\label{Nk prelim}
	\frac{1}{N_k} = \frac{ \Theta(A(\infty_k ) - A(P_k) - K - iT) }
	{\Theta(A(\infty_k) - A(P_k) - K)} 
	\lim_{P \rightarrow \infty_k} \exp \lp i \left[ p(z(P)) - \int_{iq}^P \tau \right] \rp 
\end{equation}
which defines $N_k$ provided the right hand side is not identically zero. Recall that
\begin{equation*}
	(f_k) = P_k + \infty_{k'} - P_{iq} - P_{\alpha^*}, \qquad k,k' \in {1,2},\ k \neq k'.
\end{equation*}
By Abel's Theorem $A( (f) ) = 0$ so $A(P_k) = A(iq) + A(\alpha^*) - A(\infty_{k'})$. Now $A(\infty_1) + A(\infty_2) \equiv   0$, and the evaluation of the Abel map at the branching points $iq$ and $\alpha^*$ can be exactly evaluated for any concrete choice of paths. Choosing to integrate, for convenience, on the first sheet along the + sides of the band and gap contours we have $A(iq) = 0$ and $A(\alpha^*) = i\pi+H/2 = K$. Inserting these values into \eqref{Nk prelim} and using Lemma \ref{lem: tau} and the periodicity relations \eqref{Theta periods} we have
\begin{equation*}
	\frac{1}{N_k} = \frac{\Theta(iT_0 - i\eps^{-1} \Omega)}{\Theta(0)} \exp \lp -iT_0 + i (-1)^{k+1} Y_0 \rp.
\end{equation*}
Which is always well defined as $T _0$ and $\Omega$ are always real and thus avoids the zeros of the $\Theta$-function. 
\end{proof}

We now remove the Riemann surface from the picture to write our outer solution $O(z)$ in terms of integrals lying completely on the complex plane. We do so by exploiting the antisymmetry of the differentials: both differentials $\eta =\tau$ and $\eta =\nu$ share the property that
\begin{equation*}
	\int_{iq}^{P^1(z)} \eta \equiv - \lp \int_{iq}^{P^2(z)} \eta \rp.
\end{equation*}      
To descend the Abel map onto a function from $\C$ to $\C$ we make the following restrictions. By $A(z)$ we mean an integral from $iq$ to $z$ lying completely on the first sheet such that the path does not intersect the band or gap contours $\Gamma_\nu \cup \Gamma_\nu^*$. These restriction on the path eliminate the addition of $a$ or $b$-cycles to the value of the Abel map making $A$ well defined on $\C$. 

Recalling the definitions \eqref{elliptic model def} of $O^{(1)}(z)$ and \eqref{lifting} of $\vec{v}(P)$ we have the following formula for the solution of the outer model problem RHP \ref{rhp: two band outer}:
\begin{equation}\label{OuterSolution}
O(z) = \frac{\Theta(0)} {\Theta(iT)}  \exp \lp -i Y_0 \rp^\sig \, \mathcal{T}(z) 
\left[ s_0(z)\exp\lp -i \int_{iq}^z \tau_0 \rp \delta(z) \right]^\sig,
\end{equation}
where $\mathcal{T}(z)$ is the matrix
\begin{equation*}
\mathcal{T}(z) = \begin{pmatrix}
\frac{ \beta(z)+\beta(z)^{-1} }{2_\Bspace} \frac{\Theta(A(z) - A(\infty) - i T )}{\Theta(A(z) - A(\infty) ) } &
\frac{ \beta(z)-\beta(z)^{-1} }{2} \frac{\Theta(A(z) + A(\infty) + i T )}{\Theta(A(z) + A(\infty) ) } \\
\frac{ \beta(z)-\beta(z)^{-1\Tspace} }{2} \frac{\Theta(A(z) + A(\infty) - i T )}{\Theta(A(z) + A(\infty) ) } &
\frac{ \beta(z)+\beta(z)^{-1} }{2} \frac{\Theta(A(z) - A(\infty) + i T )}{\Theta(A(z) - A(\infty) ) } 
\end{pmatrix} .
\end{equation*}
That $O(z)$ as defined by \eqref{OuterSolution} is a solution of the outer model problem follows directly from Prop. \ref{prop: two band s}, Lemma \ref{lem: baker}, and the observation that $\frac{1}{2}(\beta(z) +(-1)^{j+k} \beta(z)^{-1})$ are the projections of $\beta(z)f_j(P^k(z)),$ $j,k\in\{1,2\},$ onto $\C$. To control the error problem later it is necessary to know that the outer model remains bounded as $\eps \rightarrow 0$. As a consequence of Lemma \ref{lem: tau}, formula \eqref{OuterSolution} for $O(z)$ is $\eps$-dependent only through the parameter $iT = iT_0(x,t) - i\eps^{-1}\Omega(x,t)$ appearing in the arguments of the $\Theta$ functions and as $\Omega(x,t)$ is always real valued, and $\Theta$ is $2\pi i$ periodic it follows that the outer solution is uniformly bounded as $\eps \rightarrow 0$. 

As one expects the outer solution encodes the leading order asymptotic behavior of the solution $\psi(x,t; \eps)$ of the NLS equation \eqref{nls} for each $(x,t) \in \mathcal{S}_2$ where the genus-one ansatz successfully controls the Riemann-Hilbert problem. We record here, for out later use, that leading order behavior:
\begin{equation}\label{two band leading order}
	\begin{split}
	2i \lim_{z \rightarrow \infty} z O_{12}(z; x,t) &=  \\
	(q - \imag \alpha)) &\frac{\Theta(0)}{\Theta(2A(\infty))}\frac{\Theta(2A(\infty) + i T_0 - i\eps^{-1}\Omega)}{\Theta(iT_0 - i\eps^{-1} \Omega) } e^{-2iY_0 }.
\end{split}
\end{equation}

\subsubsection{Construction of the local models}  		

We now describe the construction of the local models near the band-gap endpoints $\alpha$, $\alpha^*$ and the stationary points $\xi_0$ and $\xi_1$. At each of these points the jump matrices $V_Q^0$ are no longer uniformly near identity which in turn implies that the outer model $O(z)$ is no longer a uniform approximation of the solution $Q(z)$ to RHP \ref{rhp: two band Q} in any neighborhood of these points. The local models near $\xi_0$ and $\xi_1$ are again constructed from the solution $\Psi_{PC}$ of the parabolic cylinder problem, RHP {\ref{rhp: PC}. The new feature in the genus one case are the local models near the band-gap endpoints $\alpha$ and $\alpha^*$.  The construction of the local model near such points is nearly canonical. The local error near these points in controlled by installing a local model constructed from Airy functions.

Let us consider first the real stationary phase points $\xi_0$ and $\xi_1$. The jump matrix $V_Q$, given by \eqref{two band Q jumps}, near these points is essentially the same jump matrix appearing in both the genus zero case and the quiescent case for $x$ outside the support (cf. \eqref{one band Q jumps} and \eqref{Q jumps}).  The local models at each of these points are constructed from parabolic cylinder functions; the reader is referred to Sections \ref{sec: outside xi_0} and \ref{sec: outside xi_1} for each model's motivation and additional details.   

Define the pair of locally analytic and invertible functions $\zeta = \zeta_k(z)$:
\begin{equation}\label{two band zeta_k}
	\begin{split}
	\frac{1}{2} \zeta_0^2 &:= \frac{ \varphi_{0+}(z) - \varphi_{0+}(\xi_0) } { \eps} = \frac{1}{\eps}\int_{\xi_0}^z 4tS(\lambda) (\lambda -\xi_0) \dd \lambda \\
	\frac{1}{2} \zeta_1^2 &:= \frac{ (\varphi_1(z) - \varphi(\xi_1) } { \eps} =  \frac{1}{\eps}\int_{\xi_1}^z 4t S(\lambda) (\lambda -\xi_0) +4L\frac{\lambda}{\nu(\lambda)} \dd \lambda 
	\end{split}
\end{equation}  
Each $\zeta_k$ introduces a rescaled local coordinate at $\xi_k$ and we choose suitably small, fixed size, neighborhoods $\U_k$ of $\xi_k$ such that the $\zeta_k$ are analytic inside $\U_k$ and the images $\zeta = \zeta_k(\U_k)$ are disks in the $\zeta$-plane. Next, using Prop. \ref{prop: delta expansions} and \eqref{delta expansions} we define the nonzero and locally holomorphic scaling functions:
\begin{equation}\label{two band h's}
	\begin{split}
	h_0 &= \left[ \lp \frac{\eps}{\varphi_0''(\xi_0)} \rp^{i\kappa(\xi_0)} \lp \delta_0^{\text{hol}}(z) s_+(z)\rp^2 e^{-i\varphi_0(\xi_0) / \eps} \right]^{1/2}, \\
	h_1 &= \left[ \lp \frac{\eps}{\varphi_1''(\xi_1)} \rp^{i\kappa(\xi_1)} \delta_{1+}^{\text{hol}}(z)  \delta_{1-}^{\text{hol}}(z) \ e^{-i\varphi_1(\xi_1) / \eps} \right]^{1/2}. 
	\end{split}
\end{equation}

With the above definitions in hand we define our local models as follows:
\begin{equation}\label{two band local form} 
	\begin{split}
	A_0(z) &= \beta(z) O^{(1)}(z) \, \hat{A}_0(z) \, s(z)^\sig \delta(z)^\sig, \\
	A_1(z) &= \beta(z) O^{(1)}(z)  s(z)^{\sig} \, \hat{A}_1(z) \, \delta(z)^\sig
	\end{split}
\end{equation}
where the functions $\hat{A}_k$ are built from the solution $\Psi_{PC}(\zeta,a)$ of the parabolic cylinder local model, RHP \ref{rhp: PC}: 
\begin{subequations}\label{two band local models}
	\begin{align}
	\hat{A}_0(z) &= h_0^\sig \Psi_{PC} ( \zeta_0(z),\, r_{0+}(\xi_0) ) h_0^{-\sig}, \label{ two band xi_0 model}\\
	\hat{A}_1(z) &= h_1^\sig \Psi_{PC} ( \zeta_1(z),\, -r_0(\xi_1) ) h_1^{-\sig} \times 
		U^{-1}(z) e^{ig(z)\sig / \eps} F^{-1}(z) e^{-ig(z) \sig/ \eps}. \label{ two band xi_1 model} \end{align}	
\end{subequations}
Here, we recall that $F(z)$ and $U(z)$ are the folding and unfolding factorizations defined by \eqref{folding factorization} and \eqref{A3 def} respectively. The definitions \eqref{two band local form}, \eqref{two band zeta_k}, and the large $\zeta$ expansion $\Psi_{PC}(\zeta) = I + \bigo{\zeta^{-1}}$ imply that on the boundaries $\U_k$ we have 
\begin{equation}\label{two band PC matching error}
	O^{-1}(z) A_k(z) = I + \bigo{\sqrt{\eps}}.
\end{equation}

Let us now consider the situation near the band-gap endpoints $\alpha$ and $\alpha^*$. The jump matrices, up to orientation of the contours, satisfy the symmetry $V(z) = V(z^*)^\dagger$ so it is sufficient to consider only the problem at $z=\alpha$. The local jump matrices of $V_Q$ near $z=\alpha$ are depicted in Figure \ref{fig: alpha exact jumps}. Here we have used the relation \eqref{R hat sing} to write the reflection terms in the jump matrices in terms of the single locally analytic and nonzero function $w(z)$. We make no approximations of the local jump matrices, the only error introduced by the local model will be on the boundary where it meets the outer model. 
\begin{figure}
	\begin{center}
		\includegraphics[width=.6\textwidth]{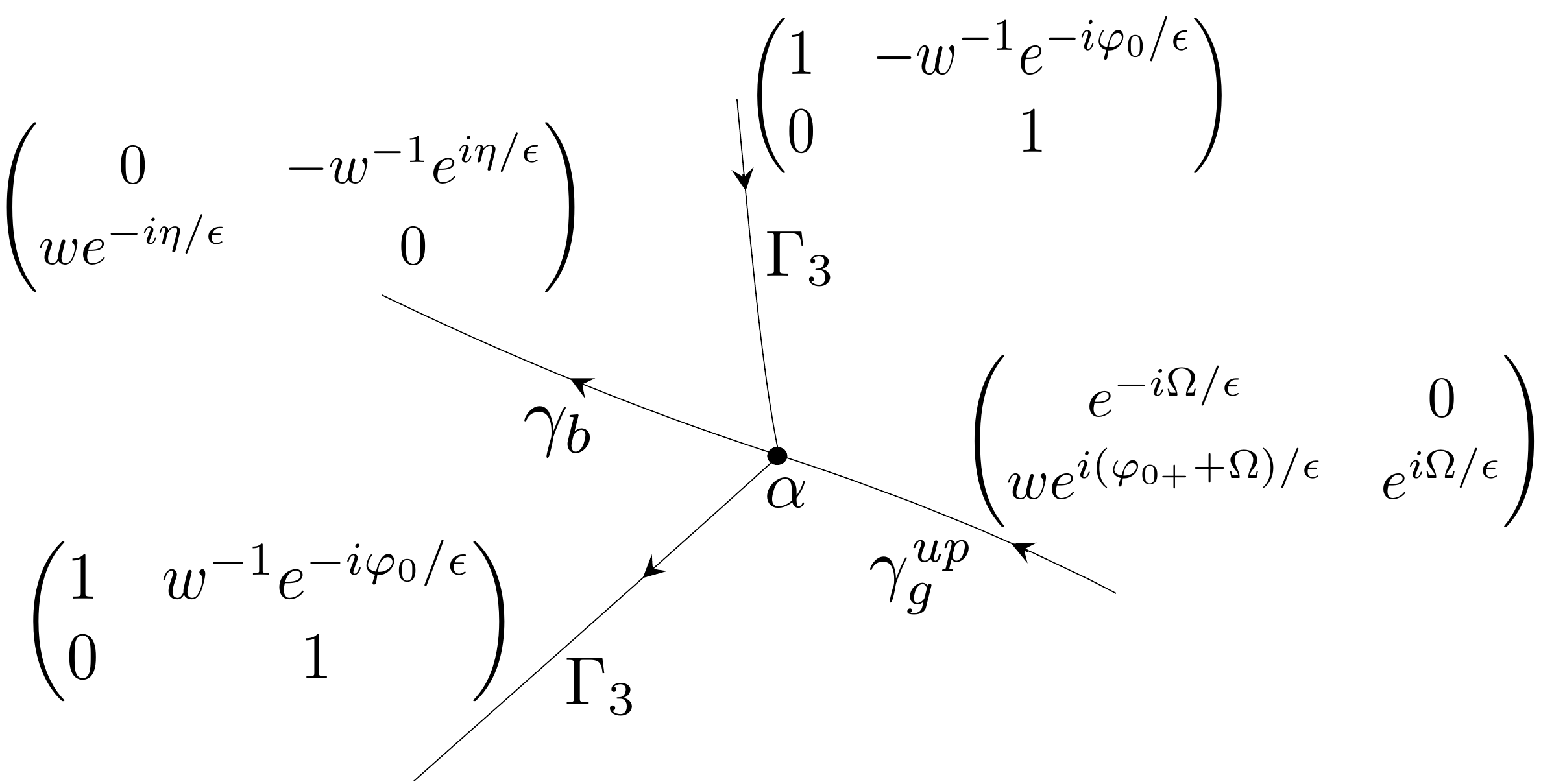}
		\caption{The jump matrices of $V_Q$ near the point $z = \alpha$.
		\label{fig: alpha exact jumps}
	}
	\end{center}
\end{figure}
To introduce our local model, let $\U_\alpha$ be a neighborhood of $\alpha$, small enough for $\Gamma_\nu$ to disconnects $\U_\alpha$ into two components; let $\U_{\alpha +}$ $\lp \U_{\alpha -} \rp$ denote the component to the left (right) of $\Gamma_\nu$ with respect to orientation. Define the pair of functions on $\U_\alpha$:
\begin{equation*}
	f(z) := \begin{cases} 
		\varphi_0 +\eta + \Omega & z \in \U_{\alpha +} \\
		\varphi_0 +\eta - \Omega & z \in \U_{\alpha -} 
		\end{cases} 
		\quad \text{and} \quad 
	\sgn_{\Gamma_\nu}(z) := \begin{cases}
		1 & z \in  \U_{\alpha +} \\
		-1 & z \in  \U_{\alpha +} \\	
		\end{cases}
\end{equation*}
The continuity of $\varphi_0$ and jump relations \eqref{two band g jumps} imply that $f(z)$ is analytic in $\U_\alpha \backslash \gamma_b$. The function $f(z)$ motivates the change of variables
\begin{equation}\label{zeta alpha}
	\zeta_\alpha(z) := \lp -\frac{3i}{4\eps} f(z) \rp^{2/3} = \lp -\frac{ 3it}{\eps} \int_\alpha^z S(\lambda)(\lambda - \xi_0) \dd \lambda \rp^{2/3} 
\end{equation}
which is a locally analytic and invertible and maps $\U_\alpha$ to a neighborhood of the origin in the $\zeta$-plane. We choose the branch such that $\zeta_\alpha(\gamma_b \cap \U_\alpha)$ lies along the negative real axis and use our freedom to deform $\Gamma_3$ and $\gamma_g$ such that $|\arg(\zeta_\alpha(\Gamma_3\cap \U_\alpha)) | = 2\pi/3$ and $\arg(\zeta_\alpha(\gamma_g \cap \U_\alpha))=0$. We seek our outer model in the form 
\begin{equation}\label{alpha model}
	A_\alpha(z) = K_\alpha(z) \Psi_{\Ai}(\zeta_\alpha(z))~\sigma_2 \left[ w(z) e^{- \frac{i}{\eps}(\eta+\Omega \sgn_{\Gamma_\nu}(z) )} \right]^{\sig/2}	
\end{equation}
where $K_\alpha(z)$ is an analytic pre-factor which we will determine later. The matrix $\Psi_{\Ai}(\zeta)$ then solves the following well known Airy Riemann-Hilbert problem in the $\zeta$-plane.
\begin{rhp}[ for $\Psi_{\Ai}$ (The Airy RHP):]\label{rhp: Airy}
Find a $2 \times 2$ matrix valued function $\Psi_{\Ai} (\zeta)$ such that
\begin{enumerate}[1.]
	\item $\Psi_{\Ai}$ is analytic for $\zeta \in \C \backslash\, \{ \zeta\, : \, |\arg(\zeta)| = 0, 2\pi/3, \pi \}$.

	\item As $\zeta \rightarrow \infty$,
	\begin{equation}\label{Airy asymp}
		\Psi_{\Ai}(\zeta) = \frac{e^{\frac{i\pi}{12}}}{2\sqrt{\pi}}
        		\zeta^{-\sig/4} \begin{pmatrix} 1 & 1 \\ -1 & 1
        		\end{pmatrix} e^{-\frac{i\pi}{4}\sig} \left[ I + \bigo{\zeta^{-3/2}} \right]
    	\end{equation}
 
 	\item $\Psi_{\Ai}$ assumes continuous boundary values $(\Psi_{\Ai})_+$ and $(\Psi_{\Ai})_-$ on each jump contour satisfying the relation $(\Psi_{\Ai})_+(\zeta) = (\Psi_{\Ai})_-(\zeta) V_{\Psi_{\Ai}}(\zeta)$ where
\begin{equation*}\label{Airy jumps}
	V_{\Psi_{\Ai}}(\zeta) = \begin{cases}
		\offdiag{1}{-1}_{\Bspace} & \arg(\zeta) = \pi \\
		\triu{e^{\frac{4}{3} \zeta^{3/2}}}_{\Bspace} & |\arg(\zeta)| = 2\pi / 3 \\
		\tril{e^{-\frac{4}{3} \zeta^{3/2}}} & \arg(\zeta) = 0
	\end{cases} 
\end{equation*}	 	
\end{enumerate}
\end{rhp}	
This problem, like the parabolic cylinder problem, RHP \ref{rhp: PC}, is one of the standard local models that emerge in the literature of both inverse scattering and random matrix theory, see \cite{DKMVZ99a} or \cite{DKMVZ99b}. 
\begin{prop}
	The solution of RHP \ref{rhp: Airy} is given by 
	\begin{equation}\label{Airy solution}
	\Psi_{\Ai}(\zeta) = \left\{ \begin{array}{l@{\quad}c@{\quad}l}
    		\widetilde{\Psi}(\zeta) e^{(\frac{2}{3}\zeta^{3/2}-\frac{\pi
    		i}{6})\sig} &:& \arg(\zeta) \in (0, \frac{2\pi}{3}) \\
     		& & \\
  		\widetilde{\Psi}(\zeta) e^{(\frac{2}{3}\zeta^{3/2}-\frac{\pi
  		i}{6})\sig} \tril{-e^{\frac{4}{3}\zeta^{3/2}}} &:& \arg(\zeta) \in (\frac{2\pi}{3}, \pi) \\
     		& & \\
    		\widetilde{\Psi}(\zeta) e^{(\frac{2}{3}\zeta^{3/2}-\frac{\pi
   		 i}{6})\sig} \tril{e^{\frac{4}{3}\zeta^{3/2}}} &:& \arg(\zeta) \in (-\pi, -\frac{2\pi}{3}) \\
     		& & \\
    		\widetilde{\Psi}(\zeta) e^{(\frac{2}{3}\zeta^{3/2}-\frac{\pi
    		i}{6})\sig} &:& \arg(\zeta) \in (-\frac{2\pi}{3}, 0)
    		\end{array} \right. ,
	\end{equation}
	where $\omega := e^{2\pi i/ 3}$ and   
\begin{equation*}
	   \widetilde{\Psi}(\zeta) = \left\{ \begin{array}{l@{\quad}c@{\quad}l}
        \begin{pmatrix} \Ai(\zeta) & \Ai(\omega^2\zeta) \\ \Ai'(\zeta) &
        \omega^2 \Ai'(\omega^2\zeta) \end{pmatrix}_{\Bspace} &:& \zeta \in
        (0,\pi) \\
        \begin{pmatrix} \Ai(\zeta) & -\omega^2\Ai(\omega\zeta) \\ \Ai'(\zeta) &
        -\Ai'(\omega\zeta) \end{pmatrix} &:& \zeta \in
        (-\pi,0)
    \end{array}\right. .
\end{equation*}
\end{prop}
The solution clearly satisfies the jump conditions along the rays $|\arg(\zeta)| = 2\pi/3$. The normalization conditions follows immediately from the asymptotic behavior of the Airy functions \cite[p.~448]{AS64} and to verify the jump conditons on $\R$ one needs the Wronskian relations between the various Airy functions \cite[p.~446]{AS64} and the fundamental identity $\Ai(\zeta) + \omega\Ai(\omega \zeta) + \omega^2 \Ai (\omega^2 \zeta) = 0$. As the solution is standard in the literature, we leave the verification of these facts to the interested reader.

The last step in the construction of the local model inside $\U_\alpha$ is to select the analytic pre-factor $K_\alpha(z)$ in order to match the outer model along the boundary $\partial \U_\alpha$. From the asymptotic behavior \eqref{Airy asymp} if we choose
\begin{equation*}
	K_\alpha(z) = O(z) \left[ w(z) e^{-\frac{i}{\eps}(\eta + \Omega \sgn_{\Gamma_\nu}(z) ) } \right]^{-\sig/2} \sigma_2
	 \left[  \frac{e^{\frac{i\pi}{12}} }{2\sqrt{\pi}} \zeta_\alpha (z)^{-\sig/4} 
	 \begin{pmatrix} 1 & 1 \\ -1 & 1\end{pmatrix} e^{-\frac{i\pi}{4}\sig} \right]^{-1} 
\end{equation*}	
then in light of the $\eps$ scaling in the definition \eqref{zeta alpha} of $\zeta_{\alpha}(z)$ it follows that 
\begin{equation}\label{Airy error}
	O^{-1}(z) A_{\alpha}(z) = I + \bigo{\eps}, \quad z \in \partial \U_\alpha.
\end{equation}
It remains to show that $K_\alpha$ is analytic in $\U_\alpha$. From the definition it is clear that $K_\alpha$ is analytic in $\U_\alpha \backslash (\gamma_b \cup \gamma_g)$, and by explicitly calculating the jumps along these contours using \eqref{two band O jumps} one sees that $K_\alpha$ has identity jumps along these contours. It follows that at worst $K_\alpha$ has an isolated singularity at $z=\alpha$. However, the growth condition \eqref{two band O sing} together with the factor $\zeta_\alpha(z)^{-\sig/4}$ imply that at most $K_\alpha$ grows as an inverse square root and thus the singularity is removable. This completes the construction of the local model at $z=\alpha$. 

Building the model at $z=\alpha^*$ is completely analogous to the construction of the model at $z=\alpha$. We summarize the calculations here. We take as our neighborhood of $\alpha^*$ the conjugate neighborhood of $\alpha$: $\U_{\alpha^*} = \lp \U_\alpha \rp^*$. Then the model inside $\U_{\alpha^*}$ is given by
\begin{equation}\label{alpha* model}
	A_{\alpha^*}(z) = K_{\alpha^*}(z) \Psi_{\Ai} (\zeta_{\alpha^*}(z) ) \left[-w^*(z)^{-1} e^{-\frac{i}{\eps} ( \eta + \Omega \sgn_{\Gamma_\nu^*}(z) )} \right]^{\sig/2},
\end{equation}
where 
\begin{equation*}
	\zeta_{\alpha^*}(z) = 
		\lp \frac{4i}{3\eps} \lp \varphi_0 + \eta + \Omega \sgn_{\Gamma_\nu^*}(z) \rp \rp^{2/3}
\end{equation*}
and 
\begin{equation*}
	K_{\alpha^*}(z) = O(z) \left[-w^*(z)^{-1} e^{-\frac{i}{\eps} ( \eta + \Omega \sgn_{\Gamma_\nu^*}(z) )} \right]^{\sig/2}
	 \left[  \frac{e^{\frac{i\pi}{12}} }{2\sqrt{\pi}} \zeta_{\alpha^*} (z)^{-\sig/4} 
	 \begin{pmatrix} 1 & 1 \\ -1 & 1\end{pmatrix} e^{-\frac{i\pi}{4}\sig} \right]^{-1} .
\end{equation*}	
Essentially repeating the calculations at $z=\alpha$, this model exactly removes the jumps of $V_Q$ inside $\U_{\alpha^*}$ while on the boundary we get the same error estimate as in \eqref{Airy error}:
\begin{equation}\label{Airy error 2}
	O^{-1}(z) A_{\alpha^*}(z) = I + \bigo{\eps}, \quad z \in \partial \U_{\alpha^*}.
\end{equation}

\subsection{The Error RHP, proof of Theorem \ref{thm: main} part three.}
We now show that the parametrix defined by \cref{two band parametrix,OuterSolution,two band local form,alpha model,alpha* model} is a uniformly accurate approximation of the solution $Q(z)$ to RHP \ref{rhp: two band Q}. We do this by proving that the error matrix
\begin{equation}\label{two band error}
	E(z) := Q(z) P^{-1}(z)
\end{equation} 
satisfies a small-norm Riemann-Hilbert problem. We can then prove the existence of and derive an asymptotic expansion for $E(z)$. By unravelling the series of explicit transformations this yields an asymptotic expansion for $m(z)$ the solution of the inverse scattering problem, RHP \ref{rhp: square barrier} which in turn gives an asymptotic expansion for the solution $q(x,t)$ of \eqref{nls} for each $(x,t) \in \mathcal{S}_2$ where the genus-one $g$-function successfully stabilizes the inverse analysis. 

As both $Q$ and $P$ are piecewise analytic whose components take continuous boundary values on their respective domains of definition, their ratio, $E(z)$ is also piecewise analytic and satisfies its own Riemann-Hilbert problem.
\begin{figure}[htb]
\begin{center}
	\includegraphics[width=.6\linewidth]{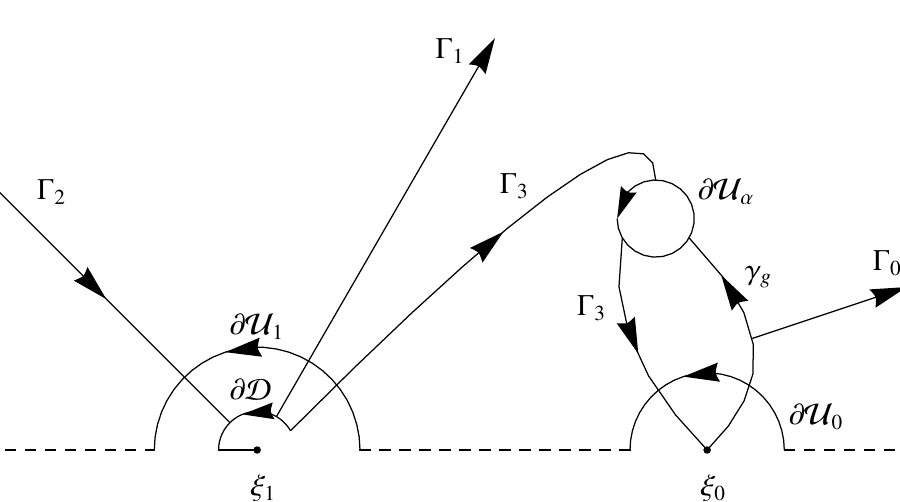}   
	\caption{The jump contour for the error matrix $E$ in the upper half-plane. The boundaries $\partial \U_1, \ \partial \U_2,$ and $ \partial \U_\alpha$ are all fixed sized while the inner loop $\partial \mathcal{D}$ scales like $\eps^{1/2}$. The jumps in the lower half-plane follow from Schwartz reflection symmetry. 
	\label{fig: two band error contours}
	}
\end{center}
\end{figure}
\begin{rhp}[for the error matrix, $E(z)$.]\label{rhp: two band error}
Find a $2\times 2$ matrix-valued function $E(z)$ such that 
\begin{enumerate}[1.]
	\item $E$ is bounded and analytic for $z \in \C \backslash \Gamma_E$, where $\Gamma_E$ is the collection of contours depicted in Figure \ref{fig: two band error contours}.
	\item $E(z) \to I + \bigo{1/z}$ as $z \to \infty$.
	 \item $E(z)$ assumes continuous boundary values $E_+(z)$ and $E_-(z)$ for $z \in\Gamma_E$ satisfying the jump relation $E_+(z) = E_-(z) V_E(z)$ where 
	 \begin{equation}\label{two band E jump}
	V_E(z) = P_-(z) V_Q(z) V_P(z)^{-1} P_-(z)^{-1}.
	\end{equation}
	and both $V_Q$ and $V_P$ are understood to equal identity where $Q$ or $P$ respectively is analytic. 
\end{enumerate}
\end{rhp}
All of the properties of $E(z)$ above are a straightforward consequence of \eqref{two band error} except the boundedness of $E(z)$ near the endpoints $\pm iq$. Since the parametrix $P(z)$ has jumps exactly matching those of $Q(z)$ along $\gamma_b \cup \gamma_b^*$, $E(z)$ has at worst isolated singularities at $\pm iq$. The local growth conditions \eqref{two band Q sing} and \eqref{two band O sing} imply that, at worst, $E_{jk}(z) = \bigo{|z\pm iq|^{-1/2}  }$. Therefore the singularities are removable and $E$ is bounded.

 The following lemma establishes that $V_E$ is uniformly near identity and decays sufficiently fast at infinity to admit an asymptotic expansion for large $z$. 
\begin{lem}\label{lem: two band small norm}  
For any $(x,t) \in K \subset \mathcal{S}_2$ compact and $\ell \in \N_0$ the jump matrix $V_E$ defined by \eqref{two band E jump} satisfies
\begin{equation}
	\| z^\ell ( V_E - I) \|_{L^p(\Sigma_E)} = \bigo{ \eps^{1/2} \log \eps}
\end{equation}
for every sufficiently small $\eps$ and $p=1,2,$ or $\infty$.
\end{lem}
\begin{proof} 
The jump matrix $V_E$ up to the replacement of the phases $\theta_k$ with their $g$-function conjugated counterparts $\varphi_k$ (and the resulting deformation of the contours and stationary phase points) are nearly identical to those in Lemma \ref{lem: outside error}. For $(x,t)$ in any compact $K \subset \mathcal{S}_2$ the endpoints $\alpha$ and $\alpha^*$ lie an order one distance away from $\pm iq$ and each other, and the real stationary phase points $\xi_0$ and $\xi_1$ exist and are well separated so we can take the four disks $\U_0,\ \U_1,\ \U_{\alpha}$, and $\U_{\alpha^*}$ to have fixed radii depending only on $K$. Furthermore, the genus one $g$-function and the phases $\varphi_k$ satisfy the necessary inequalities for the estimates in Lemma \ref{lem: outside error} to go through with only cosmetic changes away from the disks $\U_\alpha$ and $\U_{\alpha^*}$. The Airy models used to define $P$ inside each disks exactly match the jumps of $Q$, so the only new errors are on the disk boundaries. However, \eqref{Airy error} and \eqref{Airy error 2} imply that $\| V_E - I \|_{\infty} = \bigo{\eps}$ for $z \in \partial \U_{\alpha} \cup \partial \U_{\alpha^*}$ which is subdominant to the $\bigo{\eps^{1/2} \log \eps}$ error terms produced in $\U_0$ and $\U_1$.  
\end{proof}

The small norm estimate in Lemma \ref{lem: two band small norm} allows us to express the solution $E(z)$ of RHP \ref{rhp: two band error} explicitly 
\begin{equation*}
	E(z) = \frac{1}{2\pi i} \int_{\Gamma_E} \frac{ \mu(s) (V_E(s) - I) }{ s-z} \dd s
\end{equation*}
where $\mu(s)$ is the unique solution of $(1 - C_{V_E})\mu = I$ and $C_{V_E} f :=C_- [ f (V_E-I)]$ here $C_-$ denotes the Cauchy projection operator. In particular, the moment bounds in Lemma \ref{lem: two band small norm}, justify the large $z$ expansion 
\begin{equation}
	 E(z) = I + \frac{E^{(1)}(x,t)}{z} + \ldots, \qquad \text{where,} \qquad \left| E^{(1)}(x,t)\right|= \bigo{\eps^{1/2} \log \eps}  
\end{equation}

By unravelling the sequence of explicit transformations $m \mapsto M \mapsto N \mapsto Q \mapsto E$, the asymptotic bounds on $E(z)$ give us an asymptotic expansion for $m$.  Taking $z$ to infinity along the positive imaginary axis (the direction does not affect the answer) we have $m(z) =  E(z) O(z) e^{i g(z) \sig/ \eps}$ and it follows from \eqref{nls recovery} and \eqref{two band leading order} that 
\begin{equation}
	\psi(x,t) = (q - \imag \alpha)) \frac{\Theta(0)}{\Theta(2A(\infty))}\frac{\Theta(2A(\infty) + i T_0 - i\eps^{-1}\Omega)}{\Theta(iT_0 - i\eps^{-1} \Omega) } e^{-2iY_0 } + \bigo{\eps^{1/2} \log \eps}
\end{equation}	 	
gives the leading order term and error bound of the solution of \eqref{nls} with initial data \eqref{square barrier} for each $(x,t) \in \mathcal{S}_2$ in the semi-classical limit.

\section*{Acknowledgements}
We wish to thank Peter Miller for his comments which vastly improved the final draft of this manuscript. The authors are supported in part by the National Science Foundation under grants  DMS-0200749, and DMS-0800979.
 
\bibliographystyle{alpha}
\bibliography{Square_Barrier.bib}

\end{document}